\documentclass[3p,11pt]{elsarticle}
\usepackage[utf8]{inputenc}
\usepackage[colorlinks,citecolor=red,urlcolor=red]{hyperref}
\usepackage[usenames,dvipsnames]{xcolor}
\usepackage{amsmath,amsthm,amssymb,amsfonts,graphicx,subfigure,moreverb} 
\usepackage[english]{babel}
\usepackage[ruled,linesnumbered]{algorithm2e} 
\usepackage{myshortcuts} 
\usepackage{siunitx,tabularx,etoolbox,xcolor,booktabs} 
\usepackage{mathtools}
\usepackage{booktabs}
\usepackage{grffile} 
\usepackage{comment} 

\providecommand*{\mrm}[1]{\mathrm{#1}}
\newcommand{\diff}{\,\mrm{dx}}

\newcommand{\R}{{\mathbb R}}
\newcommand{\C}{{\mathbb C}}

\newcommand\ce[1]{\textcolor{blue}{#1}}

\newcommand\reviewerOne[1]{#1}  

\usepackage{pgfplots}
\usetikzlibrary{external}
\usetikzlibrary{decorations.pathreplacing}
\usetikzlibrary{positioning,arrows}

\newcommand{\rot}{\displaystyle\nabla\times}  

\newcommand{\Alin}{\mathcal{L}_0}
\newcommand{\Blin}{\mathcal{L}_1}
\newcommand{\Slin}{\mathcal{S}}

\DeclareMathOperator{\arctanh}{arctanh}
\DeclareMathOperator*{\argmax}{arg\,max}

\newcommand{\dis}{\displaystyle}
\newcommand{\om}{\omega}
\newcommand{\Om}{\Omega}

\newcommand{\scatt}{Q}

\newcommand{\eps}{\epsilon} 

\newcommand{\n}[1]{{\bf #1}}

\newcommand{\fr}[2]{\displaystyle\frac{#1}{#2}}

\newcommand{\med}{\tfrac{1}{2}}

\newcommand{\mx}[1]{\displaystyle\mathbb{#1}} 
\newcommand{\mc}[1]{\mathcal{#1}} 

\newcommand{\dpar}[2]{\displaystyle\frac{\partial #1}{\partial #2}}
\newcommand{\ddpar}[2]{\displaystyle\frac{\partial^2 #1}{\partial #2^2}}
\newcommand{\dtot}[2]{\displaystyle\frac{d #1}{d #2}}

\newcommand{\fem}{{hp}} 
\newcommand{\dtn}{a} 
\newcommand{\pml}{\ell} 
\newcommand{\PML}{{P\!M\!L}}
\newcommand{\Omdtn}{I}

\newcommand{\Honedtn}{H^{1}(\Omdtn_{\dtn})}

\usepackage[noend]{algpseudocode}

\makeatletter
\def\BState{\State\hskip-\ALG@thistlm}
\makeatother

\providecommand{\myceil}[1]{\left \lceil #1 \right \rceil } 
\providecommand{\myfloor}[1]{\left \lfloor #1 \right \rfloor } 
\iftrue 
  \usepackage{tikz}
  \usepackage{pgfplots}
  \pgfplotsset{
    compat=newest,
    tick label style={font=\scriptsize},
    label style={font=\scriptsize},
    legend style={font=\scriptsize}
  }
  \usetikzlibrary{decorations.pathreplacing}
  \iffalse 
     \usepgfplotslibrary{external}
  \else
     \renewcommand{\tikzsetnextfilename}[1]{}
  \fi
\else
  \usepackage{tikzexternal}
  \tikzsetexternalprefix{gfx_tmp/}
  
\fi

\newtheorem{theorem}{Theorem}

\newtheorem{lemma}[theorem]{Lemma}

\newtheorem{definition}[theorem]{Definition}

\newtheorem{remark}[theorem]{\textit{Remark}}

\newtheorem{goal}[theorem]{Goal}

\journal{arXiv}

\begin{document}

\begin{frontmatter}

\title{Computation of scattering resonances in absorptive and dispersive media with applications to metal-dielectric nano-structures}

\author{Juan C. Araújo C.\fnref{umit}}
\fntext[umit]{Department of Mathematics and Mathematical Statistics, Umeå University, MIT-Huset, 90187 Umeå, Sweden}

\author{Carmen Campos\fnref{UPV}}
\fntext[UPV]{Departament de Sistemes Informatics i Computacio, Universitat Polit\`ecnica de Val\`encia, Camı de Vera, s/n, 46022 Valencia (Spain)}

\author{Christian Engström\fnref{lnu}}
\fntext[lnu]{Department of Mathematics, Linnaeus University, Hus B, 35195  Växjö, Sweden}

\author{Jose E. Roman\fnref{UPV}}

\begin{abstract}
In this paper we consider scattering resonance computations in optics when the resonators consist of frequency dependent and lossy materials, such as metals at optical frequencies. The proposed computational approach combines a novel $hp$-FEM strategy, based on dispersion analysis for complex frequencies, with a fast implementation of the nonlinear eigenvalue solver NLEIGS.
Numerical computations illustrate that the pre-asymptotic phase is significantly reduced compared to standard uniform $h$ and $p$ strategies. Moreover, the efficiency grows with the refractive index contrast, which makes the new strategy highly attractive for metal-dielectric structures. The $hp$-refinement strategy together with the efficient parallel code result in highly accurate approximations and short runtimes on multi processor platforms.
\end{abstract}

\end{frontmatter}

\section{Introduction}\label{sec:intro} 
Metallic nano-structures play an important role in many applications in physics, including surface enhanced Raman scattering and optical antennas \cite{Schuller2010193}. Surface plasmons that may exist in these structures cause an enormous electromagnetic field enhancement near the surface of noble metals. In nanomedicine gold nanoparticles are used in the forefront of cancer research since they not only support plasmon resonances but also have excellent biocompatibility \cite{Jorgensen2016}.

The material properties of metals are characterized by the complex relative permittivity function $\epsilon$, which changes rapidly at optical frequencies $\omega$.  The most common accurate material model is then the Drude-Lorentz model
\begin{equation}
	\eps_{metal} (\om):=\eps_\infty + \sum_{j=0}^{N_p}\fr{f_j\om_p^2}{\om_j^2-\om^2-i\om\gamma_j},
	\label{eq:drude_lorentz}
\end{equation}
where $\eps_\infty\geq 1$ and $f_j$, $\om_p,\,\om_j$, $\gamma_j$ are non-negative \cite{Cessenat1996}. Hence, the Maxwell eigenvalue problem in the spectral parameter $\omega$ is nonlinear for metal-dielectric nanostructures. Research in operator 
theory for this type of non-selfadjoint operator functions is in its infancy and has been focused on photonic crystal applications \cite{Engstrom2018442}. 

In this article, we consider open systems in nano-optics, where the material properties are modeled by \eqref{eq:drude_lorentz}.  
The most common approach to characterize the optical properties of open metal-dielectric nanostructures is to solve a source problem in time-domain and search for peaks in the amplitude of the field \cite{Lesina201510481}. Another common strategy is to solve a source problem for a fixed real frequency and perform a frequency sweep in a region of interest \cite{Hoffmann2009}. These two strategies give valuable information of the structure for a given source. A highly attractive alternative that is used in this paper is to characterize the behavior of the system using scattering resonances \cite{lecture+zworski,MR1350074}. Scattering resonances are a discrete set of complex frequencies that refer to a metastable behavior (in time) of the corresponding system and the corresponding functions are called scattering modes.  Scattered waves can be expanded in terms of scattering resonances and scattering modes and replace in a sense Fourier series expansions for problems posed on non-compact domains \cite{lecture+zworski}.
Here, we assume that the scalar relative permittivity $\eps (x_1,x_2,\omega)$ is independent of the space coordinate $x_3$ but dependent on the frequency $\omega$.  Furthermore, we assume that $\eps=1$ outside a cylinder of radius $r_0$ and consider electromagnetic waves propagating in the $(x_1,x_2)$-plane. The Maxwell problem in $\R^3$ is then reduced to Helmholtz type of equations in $\R^2$ for the so called TM and TE polarizations \cite{Cessenat1996}.

Resonances are solutions to a nonlinear eigenvalue problem with a Dirichlet-to-Neumann map (DtN) on an artificial boundary \cite{lenoir92,araujo+engstrom+jarlebring+2017}. An attractive alternative is to use a perfectly matched layer (PML) \cite{kim09,gopal08}. This method was introduced for source problems in electromagnetics by Berenger \cite{Berenger1994} and it is related to complex coordinate stretching developed in quantum mechanics \cite[Chapter 16]{MR1361167}. The application of the PML method for resonance problems has the advantage that for non-dispersive refractive indices the resulting matrix eigenvalue problem is linear, and the eigenvalue problem is rational when a Drude-Lorentz model is used. Resonance computations with a dispersive refractive index are demanding since nonphysical eigenvalues may appear in the region of interest if the approximation properties of the used finite element space are not very good; See \cite{araujo+engstrom+2017} for a discussion of spurious eigenvalues in the one dimensional case. 

The linear algebra problem that must be solved in this kind of computations is a rational eigenvalue problem, a particular case of the nonlinear eigenvalue problem $T(\omega)\xi=0$. Recently, several numerical methods have been proposed to compute a few eigenvalues $\omega$ (and corresponding eigenvectors $\xi$) of large-scale nonlinear eigenvalue problems \cite{Kressner2009,Jarl2012,guttel17}. Some of these methods are available in the SLEPc library \cite{slepc05+roman}. Essentially, there are three types of methods: Newton-type methods, contour integral methods, and linearization methods. Newton-type methods rely on having a good initial guess, otherwise the iteration may converge to an eigenvalue far from the search region. Contour integral methods compute all eigenvalues contained in a prescribed region of the complex plane, but they require having a good estimate of the number of enclosed eigenvalues, and on the other hand they have a high computational cost since they require a matrix factorization at each integration point. 
In this paper, we consider a method of linearization type, namely NLEIGS, see \S\ref{sec:NLEIGS}.
Two of the authors of this paper are also developers of SLEPc and during this work a SLEPc solver that implements NLEIGS was developed and tuned using as benchmarks the challenging computational examples presented in \S\ref{sec:results}.

\section{Resonances in optical nano-structures}\label{sec:resonances}
Our aim is to compute resonances in nano-structures using accurate material models for e.g. metals at optical frequencies. This requires a permittivity function $\epsilon$ that depends on the spectral parameter $\omega$. For many metals, the real part of $\epsilon$ is negative in the optical region, which is explored in plasmonics \cite{Schuller2010193}. Below, we state well known properties for isotropic passive materials that are valid for any fixed $x\in\R^d$. Let 
\[
	\C_+:=\{z\in\C\,:\, 0\leq\,\mrm{arg}\,z<\pi,\,z\neq 0 \}.
\]
Then, $\omega\eps(\omega)\in\C_+$ for $\omega\in\C_+$, where $\eps$ never vanishes in $\bar\C_+$ \cite{Cessenat1996}. The most common material model for solid materials such as Gold, Silver, and Silica is the 
Drude-Lorentz \eqref{eq:drude_lorentz} model. This rational model of $\omega$ satisfies the stated analytical requirements and will be used in the applications part of the article. 

Assume that $\eps (x,\omega)=\eps (x_1,x_2,\omega)$ is independent of $x_3$ and consider waves propagating in the $(x_1,x_2)$-plane. The $x_3$-independent electromagnetic field  $(\n E,\n H)$ is then decomposed into transverse electric (TE) polarized waves 
$(E_1,E_2,0,0,0,H_3)$ and transverse magnetic (TM) polarized waves $(0,0,E_3,H_1,H_2,0)$ \cite{Cessenat1996}. This decomposition reduces Maxwell's equations to one scalar equation for $H_3$ and one scalar equation for $E_3$. 
The TM-polarized waves and the TE-polarized waves satisfy formally
\begin{equation}\label{eq:TMTE}
	-\Delta E_3-\omega^2\eps E_3=0 \quad \text{and}\quad -\nabla\cdot\left (\frac{1}{\eps} \nabla H_3\right )-\omega^2 H_3=0,
\end{equation}
respectively. The full vector fields $(\n E,\n H)$ are then obtained from Maxwell's equations
\begin{equation}\label{eq:auxiliar_eq}
	\n E=\fr{-1}{i\om \,\eps (\om)}\rot \n H,\quad\hbox{and}\quad  \n H=\fr{1}{i\om}\rot \n E.
\end{equation}
For simplicity, we consider first resonances in the TM-case with an $\omega$-independent permittivity function $\eps\geq 1$, where $\eps-1$ has compact support. Let $L^2_{\mrm{comp}}$ denote the space of $L^2$- functions vanishing outside 
some compact set and let $L^2_{\mrm{loc}}$ denote the space of functions that are in $L^2(K)$ for every compact subset $K$ of $\R^d$. Define the operator $A: L^2(\R^d)\rightarrow  L^2(\R^d)$ with domain $\text{dom}\, A=H^2(\R^d)$ by $Au:=-\eps^{-1}\Delta u$. The spectrum $[0,\infty )$  is then continuous \cite{MR1361167} and we denote by 
$R(\omega):  L^2(\R^d)\rightarrow  L^2(\R^d)$ the resolvent
\[
	R(\omega):=(A-\omega^2)^{-1},\quad \im \omega>0.
\]
The operator function $R$ is a meromorphic family of operators that can be extended to 
\[
	\hat R(\omega):  L^2_{\mrm{comp}}(\R^d)\rightarrow  L^2_{\mrm{loc}}(\R^d),\quad \hat R(\omega):=(A-\omega^2)^{-1},\quad \im \omega>0.
\]
The scattering resonances are then defined as the poles of the meromorphic continuation of $\hat R$ to $\C$. The functions in $ L^2_{\mrm{loc}}(\R^d)$ that correspond to a scattering resonance are called resonance modes \cite{lecture+zworski}. Note that for metal-dielectric nanostructures the operator $A$ in the TM-case is replaced with an operator function in $\omega$. In the next sections, we will describe two common approaches to compute resonances and the restriction of resonance modes to a compact subset of $\R^d$. In the following, we use the notation
\begin{equation}
\begin{array}{rcl}
	-\nabla\cdot \left( \rho \nabla u\right) - \om^2 \eta u &=& 0,\\
\end{array} \label{eq:master_eq}
\end{equation}
where $u:=E_z,\,\rho:=1,\,\eta:=\eps$ for the TM-case and $u:=H_z,\,\rho:=1/\eps,\,\eta:=1$ for the TE-case. 

\subsection{Scattering resonances in $\R$}\label{sec:resonances1D}
Resonances as discussed in section \ref{sec:resonances} can also be determined from a problem with a Dirichlet-to-Neumann (DtN) map \cite{lenoir92,schenk2011optimization}. In one space dimension the resonance problem restricted to $\Omdtn_\dtn:=(-\dtn,\dtn)$ is formally: Find a non-zero $u$ and a complex $\omega$ such that
\begin{equation}\label{eq:Helmholtz}
	-\left( \rho u'\right)' - \om^2 \eta u=0 \,\,\,\hbox{for}\,\, x\in\Omdtn_\dtn,
\end{equation}
where the  Dirichlet-to-Neumann (DtN) map at $x=\pm \dtn$ is
\begin{equation}\label{eq:formalDtN}
	u'(-\dtn)=-i \omega \, \,u(-\dtn), \quad u'(\dtn)  =i \omega \, u(\dtn).
\end{equation}
Let $\mc Z$ denote the set of values $\om$ that are zeros or poles of $\eps$ and set $\mc D:=\mx C\setminus \mc Z$. Define for $u,v\in\Honedtn$ and $\omega\in \mc D\subset\C$ the forms
\[
\hat t_{0}(\omega)[u,v] :=\int_{-\dtn}^{\dtn}\rho u'\overline{v}'\diff,\quad \hat t_{1}[u,v] :=-iu(\dtn)\overline{v}(\dtn)-i u(-\dtn)\overline{v}(-\dtn),\quad \hat t_{2}(\omega)[u,v]:=-\int_{-\dtn}^{\dtn}\eta u\overline{v}\diff .
\]
The nonlinear eigenvalue problem is then as follows: Find vectors $u\in\Honedtn\backslash\{0\}$ and  $\omega\in \mc D$ satisfying 
\begin{equation}
	t_1(\omega)[u,v]:=\omega^2 \hat t_2(\omega)[u,v]+\omega \hat t_1[u,v]+ \hat t_0(\omega)[u,v]=0
	\label{eq:1D}
\end{equation}
for all $v\in \Honedtn$.    

Let $\Omdtn_\dtn=\Omdtn_0\cup\dots\Omdtn_{N_r}$ denote a partitioning of $\Omdtn_\dtn$ and $\chi_{\Omdtn_m}$  the characteristic function of the subset $\Omdtn_m$. For material properties that are piecewise constant in $x$, we assume a permittivity function in the form
\begin{equation}
	\eps(x,\omega):=\sum_{m=0}^{N_r}\eps_m(\omega)\chi_{\Omdtn_m}(x),\quad x\in \Omdtn_\dtn,\quad \omega\in \mc D,
	\label{eq:permittivity_function}
\end{equation}
where the dependencies on $\omega$ in $\eps_m$ for $m=0,1,\dots$ are of Drude-Lorentz type \eqref{eq:drude_lorentz}. Note that \eqref{eq:1D} is a quadratic eigenvalue problem if $\eps$ is independent of $\omega$ and a rational eigenvalue problem for Drude-Lorentz type of materials.

\subsection{Scattering resonances in $\R^2$}
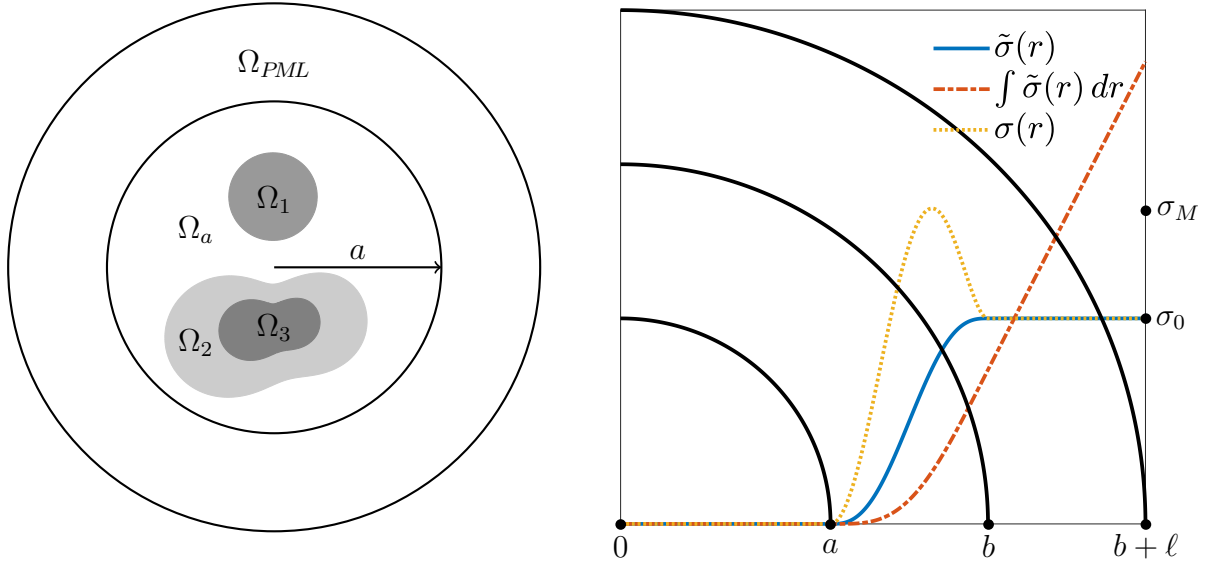
\begin{figure}
\centering
	\begin{tikzpicture}[thick,scale=1.0, every node/.style={scale=1.0}]	
		\tikzstyle{ann} = [fill=none,font=\large,inner sep=4pt]
		
		\draw( -4.10, -0.1) node { \includegraphics[scale=0.7,angle=0,origin=c]{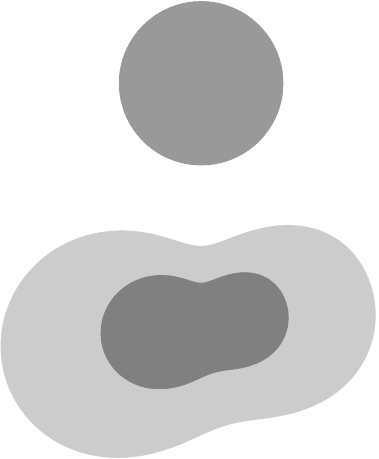} };
		
		\node[ann] at (-2.9,0.2) {$a$};
		\draw[arrows=->,line width=0.8pt](-4.0,0.0)--(-1.80,0.0);

		\draw (-4.0,0) circle (2.2cm);
		\draw (-4.0,0) circle (3.5cm);
		
		\node[ann] at ( -4.00, 2.7) {$\Omega_{\PML}$};
		\node[ann] at ( -5.03, 0.5) {$\Omega_a$};
		
		\node[ann] at ( -4.00, 0.9 ) {$\Omega_1$};
		\node[ann] at ( -5.03,-1.00) {$\Omega_2$};
		\node[ann] at ( -4.00,-0.80) {$\Omega_3$};
		
		\draw( 4.00, 0.0) node { \includegraphics[scale=0.7] {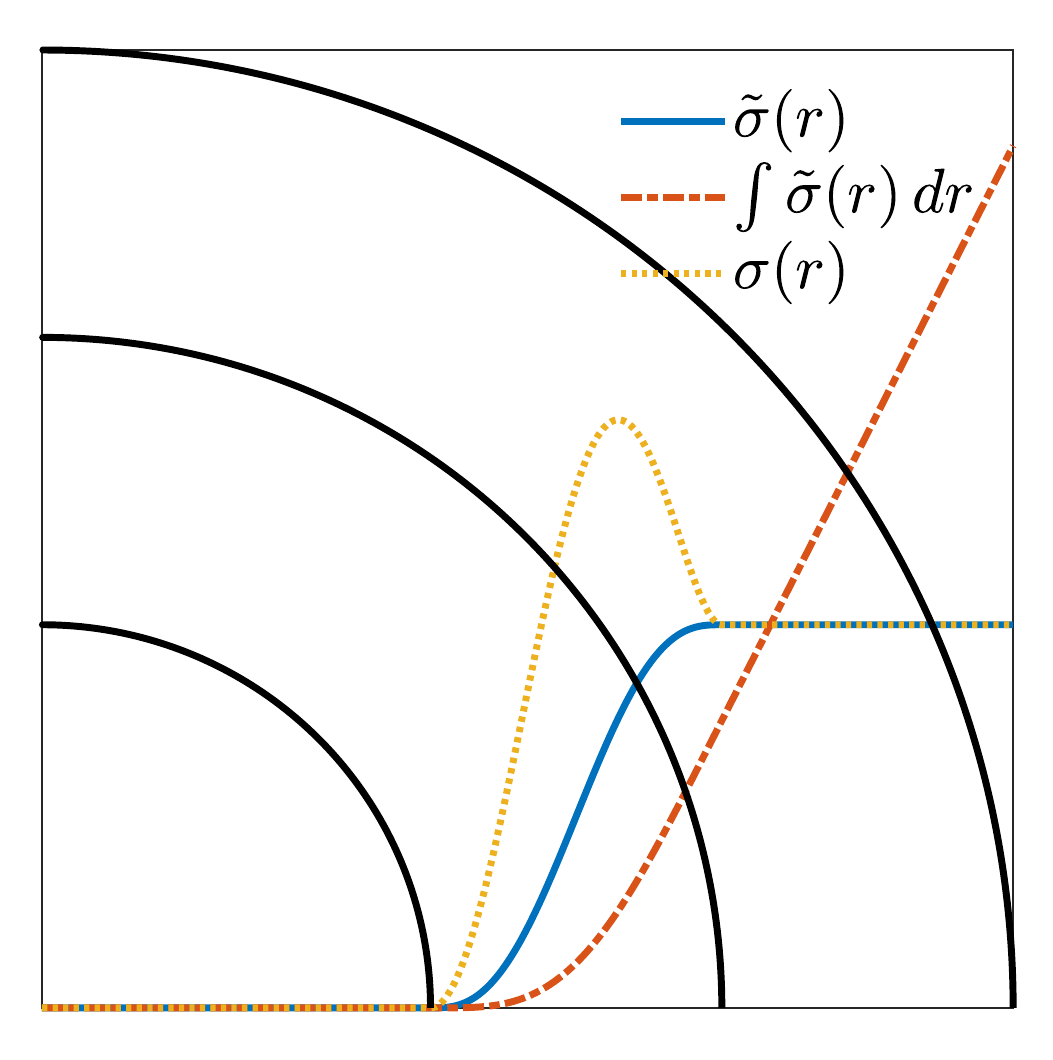} };
		
		\fill (0.56,-3.41,0.0) circle (2pt);
		\node[ann] at ( 0.56,-3.7) {$0$};
		
		\fill (3.32,-3.41,0.0) circle (2pt);
		\node[ann] at ( 3.32,-3.7) {$a$};
		
		\fill (5.40,-3.41,0.0) circle (2pt);
		\node[ann] at ( 5.40,-3.7) {$b$};
		
		\fill (7.47,-3.41,0.0) circle (2pt);
		\node[ann] at ( 7.47,-3.7) {$b+\ell$};
		
		\fill (7.47,-0.68,0.0) circle (2pt);
		\node[ann] at (7.80,-0.70) {$\sigma_0$};
		
		\fill (7.47,0.75,0.0) circle (2pt);
		\node[ann] at (7.90, 0.75) {$\sigma_M$};
	
	\end{tikzpicture}	
	\caption{\emph{Left) Arbitrary configuration of resonators. Right) 
	PML stretching function. }}\label{fig:domain_pml}
\end{figure}

Resonances in $\R^2$ can also be approximated by \eqref{eq:master_eq} with a DtN-map on an artificial boundary \cite{lenoir92,araujo+engstrom+jarlebring+2017}. However, the nonlinearity in the DtN-map is more complicated in dimensions larger than one. Then, an attractive alternative to the DtN-map is a complex coordinate stretching technique called Perfectly Matched Layers (PML). This approach does not add any non-linearity to the problem. Hence, in our setting we will obtain a rational eigenvalue problem. Approximation of resonances using a radial PML was analyzed in \cite{kim09} and we consider the truncation of the infinite PML problem to the disc $\Om$ in $\R^2$.

Let $\Omega_a$ denote a disk of radius $a$, and let $\Omega_1,\,\Omega_2,...,\Omega_{N_r}$ denote the subsets of $\Om_a$ corresponding to the resonators. Set $\Om_r:=\cup_{i=1}^{N}\Om_i$, $\Om_0:=\Om_a\setminus\Om_r$ and attached to $\Om_a$ an outer layer $\Om_\PML$.  Then, the computational domain is the disc $\Om:=\Om_a\cup\Om_\PML$ as illustrated in Fig. \ref{fig:domain_pml}.

We define the complex stretching functions in polar coordinates $(r,\theta)$, similarly as presented in \cite{lassas1998b,kim09}:
\begin{equation}
	\tilde \sigma(r):=\left\{
		\begin{array}{ll}
			0,& \hbox{if}\,\, r < a \\
			P(r),& \hbox{if}\,\, a \leq r \leq b \\
			\sigma_0,& \hbox{if}\,\, r > b 
		\end{array}
	\right.\!\!\!\!,\quad
	\begin{array}{lllll}
			\tilde \alpha(r)  \!\!\!& := 1+i\tilde \sigma(r), & \, &
			\tilde r(r) \!\!\!& := (1+i\tilde \sigma)r=\tilde \alpha(r)\,r, \\[2mm]
			\sigma(r) \!\!\!& := \tilde \sigma(r)+r\dpar{\tilde \sigma}{r}, & \, &
			\alpha(r) \!\!\!& := \dpar{\tilde r}{r}=1+i\sigma(r),
		\end{array}
		\label{eq:sig_pol}
\end{equation}
where the polynomial $P(r)$ is required to be increasing in $[a,b]$. 
Moreover, $\tilde\sigma(r)\in C^{2}(0,b+\pml)$, and $\sigma(r) = \partial (r\tilde \sigma)/\partial r$. For this we introduce the fifth order polynomial $P(r)$ satisfying
$P(a)=P'(a)=P''(a)=P'(b)=P''(b)=0$ and $P(b)=\sigma_0$.

From the given curved coordinate representation, we transform to Cartesian coordinates and define 
 $\mc A\in C^2(\Om)^{2\times 2}$ and $\mc B\in C^2(\Om)$ by
\begin{equation}
		\mc A:=\left(
		\begin{array}{cc}
		\tfrac{\tilde\alpha}{\alpha}\cos^2\theta+\tfrac{\alpha}{\tilde\alpha}\sin^2\theta &
		\bigg(\tfrac{\tilde\alpha}{\alpha} -\tfrac{\alpha}{\tilde\alpha}\bigg)\sin\theta\cos\theta \\ 
		\bigg(\tfrac{\tilde\alpha}{\alpha} -\tfrac{\alpha}{\tilde\alpha}\bigg)\sin\theta\cos\theta &
		\tfrac{\tilde\alpha}{\alpha}\sin^2\theta+\tfrac{\alpha}{\tilde\alpha}\cos^2\theta
		\end{array}
		\right),\quad 
		\mc B:=\alpha \tilde\alpha.
		\label{eq:def_pml_r}
\end{equation}
The PML coefficients are illustrated in Fig. \ref{fig:domain_pml}, and it can be seen that $\mc A$ and $\mc B$ are identities for $r\leq a$.

Let $(\cdot,\cdot)_{\Om_j}$ denote the inner product in $L^2(\Om_j)$. The nonlinear eigenvalue problem is then:
Find $u\in H_0^1(\Om)\setminus \{0\}$ and $\om  \in\mc D$ such that for all $v\in H_0^1(\Om)$
\begin{equation}
	t_2(\om)[u,v]=0,
\label{eq:eig_prob2D}
\end{equation}	
where $t_2(\om)[u,v]:=\tilde t_0(\om)[u,v]+\tilde t_1(\om)[u,v]$ with
\[
	\tilde t_0(\om)[u,v]:=(\rho\nabla u,\nabla v)_{\Om_a}-\om ^2 (\eta u,v)_{\Om_a},\quad \tilde t_1(\om)[u,v]:=(\mc A\nabla u,\nabla v)_{\Om_\PML}-\om ^2(\mc B u,v)_{\Om_\PML}.	
\]	   

\section{A-priori based $hp$-FEM for eigenvalue problems}\label{sec:dispersion}
It is well known that the accuracy of a finite element approximation of the Helmholtz problem $-\Delta u -\om^2u=f$ deteriorates with increasing frequency $\omega$. A major problem is that the discrete frequency of the FE solution is different from the frequency of the exact solution. This effect called pollution has been studied intensively. Particularly, for a uniform mesh size $h$ the asymptotic error estimates for linear elements \cite[Sec. 4.4.3]{ihl98} yield the condition $\om^2 h<1$, which for large $\om$ results in prohibitively expensive meshes. However, the dispersion analysis \cite{ihl98,thompson94,Ainsworth04} yields pre-asymptotic estimates of the form $\om h<1$, which is a significant improvement. Moreover, it was realized that higher order elements are advantageous to reduce the pollution effect.

A-posteriori estimators are powerful tools when the pollution is negligible, but in the presence of pollution the error in the solution is typically underestimated \cite{MR3094621}. 
For Helmholtz equation with $\omega>1$ and FE of order $p$ the conditions $p=O(log(\om))$ and $\om h/p=O(1)$ are sufficient for accurate a posteriori error estimation \cite{MR3094621}. Recently, error estimates that are explicit in the eigenvalue $\om^2$ have also been developed \cite{Sauter10}. However, the minimal dimension of a finite element space such that the relative eigenfunction error is below $100\%$ is unknown even in the selfadjoint case with analytic coefficients (see \cite[Remark 6.1]{Sauter10}).

\reviewerOne{Our aims are (i) to extend the dispersion analysis in \cite{Ainsworth04} to the case with a complex frequency $\omega$, (ii) 
 
to propose an a-priori $hp$-strategy for non-selfadjoint eigenvalue problems with piecewise constant coefficients based on an element-wise application of (i).}
The a-priori strategy for enriching the finite element space developed in this paper can in principle also be combined with an a-posteriori based strategy such as \cite{Giani2016,MR3529053}. 

\subsection {Numerical dispersion for a real frequency $\omega$}\label{sec:err1D}

The case with a real frequency $\omega$ and constant coefficients has been studied extensively and \cite{ihl98,thompson94,Ainsworth04,chau15} derived explicit estimates depending only on $\om,\,h$, and $p$. In this subsection we review those results and consider in the following subsection extensions to complex $\om$.

In the one dimensional setting, the normalized (wave speed $c=1$) homogeneous wave equation reads
\begin{equation}
	\ddpar{w}{t}-\ddpar{w}{x}=0.
	\label{eq:wave1d}
\end{equation}
The general solution of the wave equation can be expressed as the superposition 
\begin{equation}
	w(x,t)=\int_{-\infty}^\infty \left[a(k)e^{i(kx+\om t)}+b(k)e^{i(kx-\om t)}\right]dk,
	\label{eq:superp1d}
\end{equation}
for some functions $a$ and $b$.
The frequency $\om$ and the wave number $k$ are in this case related by the exact dispersion relation $w^2=k^2$.

We now turn into the numerical solution of \eqref{eq:wave1d}, where the \emph{discrete wave number} 
$k_\fem$ is a FE approximation to $k$. Let $\{x_j\},\,j\in \mx Z$
be a uniform distribution of points on $\R$, with mesh size $h:=x_{j+1}-x_j$, and let $\varphi_j$ be the nodal shape functions of polynomial degree $p$.
Then, semi-discrete solutions are written in the form $w_\fem (x,t)=u_\fem(x)e^{-i\om t}$, and at nodal values the FE space representation becomes $u_\fem(x_j)=\sum_j \xi_j \varphi_j(x_j)$. By analogy with \eqref{eq:superp1d}, we search for solutions of the form
$w_\fem(x_j,t)=b_\fem (k_\fem) e^{i(k_\fem x_j-\om t)}$, which implies $\xi_j=b_\fem e^{ijk_\fem h}$.

The variational formulation of the problem is then: 
Find $u_\fem\in V_\fem\subset H^1 (\R)$ such that
\begin{equation}
	B_\om(u_\fem,v_\fem):=(u_\fem',v_\fem')-\om^2 (u_\fem,v_\fem)=0,\quad (u,v):=\int_{\R} u\bar v\,dx,
	\label{eq:var_lag1d}
\end{equation}
for all $v_\fem\in V_\fem$. 
The explicit form of $\varphi_j$ in \eqref{eq:var_lag1d} leads to a discrete dispersion relation of the form $\cos(k_\fem h)=R_p(\om h)$, where the \emph{numerical cosine} $R_p(\om h)$ consists of rational terms involving $\om,\,h$, and $p$; see \cite{ihl98,thompson94,Ainsworth04} for further details.
The dispersive error for \eqref{eq:wave1d} is defined as $\mc E^p:=R_p(\om h)-\cos(\om h)$, from where \emph{dispersion analysis} refers to studying the convergence of $|\mc E^p|$ with respect to $\om, h$ and $p$. The outcome of the analysis is that $\mc E^p$ is an excellent measurement of the finite element space approximative properties for wave problems as motivated by \cite{ihl98,thompson94,Ainsworth04}. 

We use the following notations: $\kappa=\om h/2$, $N_e=\myfloor{p/2}$, $N_o=\myfloor{(p+1)/2}$, where $\myfloor{x}$ stands for the integer part of $x$. Ainsworth \cite[Sec. 4]{Ainsworth04} proved that when $\om\in\R$, the function $\mc E^p$ can be written in the form
\begin{equation}
\begin{array}{l}
\mc E^p(\om h)=\fr{\sin \om h}{\om h}\left\{
\mc E^p_o \sin^2\left(\fr{\om h}{2}\right)+\mc E^p_e \cos^2\left(\fr{\om h}{2}\right)
\right\}
\left\{
1+\fr{\sin \om h}{2\om h}(\mc E^p_o-\mc E^p_e)
\right\}^{-1}, \\[3mm]
\mc E_e^p(\kappa) \fr{\cos^2 (\om h/2)}{\om h}=-Q_{2N_e+3/2}(\kappa)\{1-Q_{2N_e+3/2}(\kappa)\tan \kappa\}^{-1}, \\[3mm]
\mc E_o^p(\kappa) \fr{\sin^2 (\om h/2)}{\om h}=-Q_{2N_o+1/2}(\kappa)\{1+Q_{2N_o+1/2}(\kappa)\cot \kappa\}^{-1},
\end{array}
\label{eq:Ep}
\end{equation}
with
\begin{equation}
Q_m(\kappa):=\fr{J_m(\kappa)}{Y_m(\kappa)}, \,\, m=\hbox{integer}+\fr{1}{2}.
\label{eq:Qm}
\end{equation} 
It was shown \cite[Theorem 3.3]{Ainsworth04} that the error $\mc E^p$ for real $\omega$ passes through three phases as the order $p$ is increased: An oscillatory phase, a transition zone, and finally superexponential decay of $\mc E^p$. 
In the remaining of the section we consider numerical dispersion analysis for $hp$-FEM computations of Helmholtz type of problems with a complex frequency $\om$. 

\reviewerOne{This is of interest since scattering resonances are complex and the results in Section \ref{sec:description_lag} are the base for the $hp$-FEM strategy proposed in Section \ref{sec:con_assum}.}

\subsection{Numerical dispersion for a complex frequency $\omega$}\label{sec:description_lag}
First, we show that the results in \cite{ihl98,Ainsworth04} 
can be extended from $\om\in\R$ to a region in the complex plane. This extension requires that several issues are addressed. Namely, that the  expressions can be analytically continued to the complex plane, and the identification of possible branch cuts and poles of the different expressions involved when deriving the estimates in \cite{Ainsworth04}. We rely on the results in \cite{olver_DLMF,olver54,olver_new}, where many of the subtleties of working with Bessel functions of complex argument are addressed.

It can be verified that \eqref{eq:Ep} also holds for $\om \in \mx C$ with $|\arg\om|<\pi$. First, by introducing standard sesquilinear forms and following the derivations in \cite{Ainsworth04}. Particularly, equation \cite[(4.12)]{Ainsworth04} is reached by using \cite[(8.461), (8.465)]{grad07}, which in turn hold for complex arguments. Note that in the case $\om\in \mx R$, the subscripts $o,\,e$ in \eqref{eq:Ep} are reserved for \emph{odd}, and \emph{even}, respectively. 

\subsubsection{Numerical dispersion analysis for small $|\om h|$}\label{sec:small_wh}
In this subsection, we consider $\mc E^p$ for small $|\om h|$ and address the case with large $|\om h|$ in the next section. In the procedure we need the following lemma.
\begin{lemma}\label{lemmaQ1} Let $m=n+1/2$ for $n\in\mx Z$ and define $Q_m$ as in \eqref{eq:Qm}. Then,
\begin{equation}\label{eq:Qexp}
Q_m(\kappa)=-\fr{1}{2}\left[\fr{(m-\med)!}{(2m-1)!}\right]^2\fr{(2\kappa)^{2m}}{2m}+\cdots,
\end{equation}
for all $\kappa\in\C$ with $|\kappa| \ll 1$ and $|\arg \kappa| <\pi$.
\end{lemma}
\begin{proof} We follow the steps in the proof of \cite[Lemma A1]{Ainsworth04}, which is based on the representation formulas in \cite[(8.440)]{grad07}. Those formulas are under the assumption $|\arg \kappa| <\pi$, also valid for complex $\kappa$. Since, in addition $|\kappa| \ll 1$, the representation formulas \cite[(8.440), (8.465-1) ]{grad07} hold:
\begin{equation}
\begin{array}{l}
J_\nu(\kappa)=\left( \fr{\kappa}{2} \right)^\nu\dis \sum_{k=0}^{\infty} \fr{(-1)^k}{k!\,\Gamma(\nu+k+1)}
\left( \fr{\kappa}{2} \right)^{2k}, \\[5mm]
Y_{n+1/2}(\kappa)=(-1)^{n-1}J_{-n-1/2}(\kappa), \qquad |\arg \kappa| <\pi.
\end{array}
\end{equation}
By retaining only the first term in this series and using properties of the $\Gamma$ function, we obtain
\eqref{eq:Qexp}.
\end{proof}

The following theorem extends \cite[Thm 3.2]{Ainsworth04} to complex frequencies $\om$.
\begin{theorem}\label{thm:low_frequency} Let $p\in\mx N$, and $|\om h|\ll 1$. The discrete dispersion relation $\mc E^p$ is then
\begin{equation}%
	R_p(\om h)-\cos (\om h)=\frac{1}{2} \left[ \fr{p!}{(2p)!} \right]^2 \fr{(\om h)^{2p+1}}{2p+1}+\mc O(\om h)^{2p+4}.
	\label{eq:small_wh}
\end{equation}
\end{theorem}
\begin{proof} 
Lemma \ref{lemmaQ1} is stated for complex $\kappa$, with $|\arg \kappa| <\pi$, then by plugging \eqref{eq:Qexp} into \eqref{eq:Ep}, the result follows from the discussion in the proof of \cite[Thm 3.2]{Ainsworth04}.
\end{proof}

\subsubsection{Numerical dispersion analysis for large $|\om h|$}\label{sec:large_wh}
\begin{figure}
\centering
	\begin{tikzpicture}[thick,scale=1.0, every node/.style={scale=1.0}]	
		\tikzstyle{ann} = [fill=none,font=\large,inner sep=4pt]
		
		\draw( -0.18, 0.03) node { \includegraphics[scale=0.484,angle=0,origin=c]{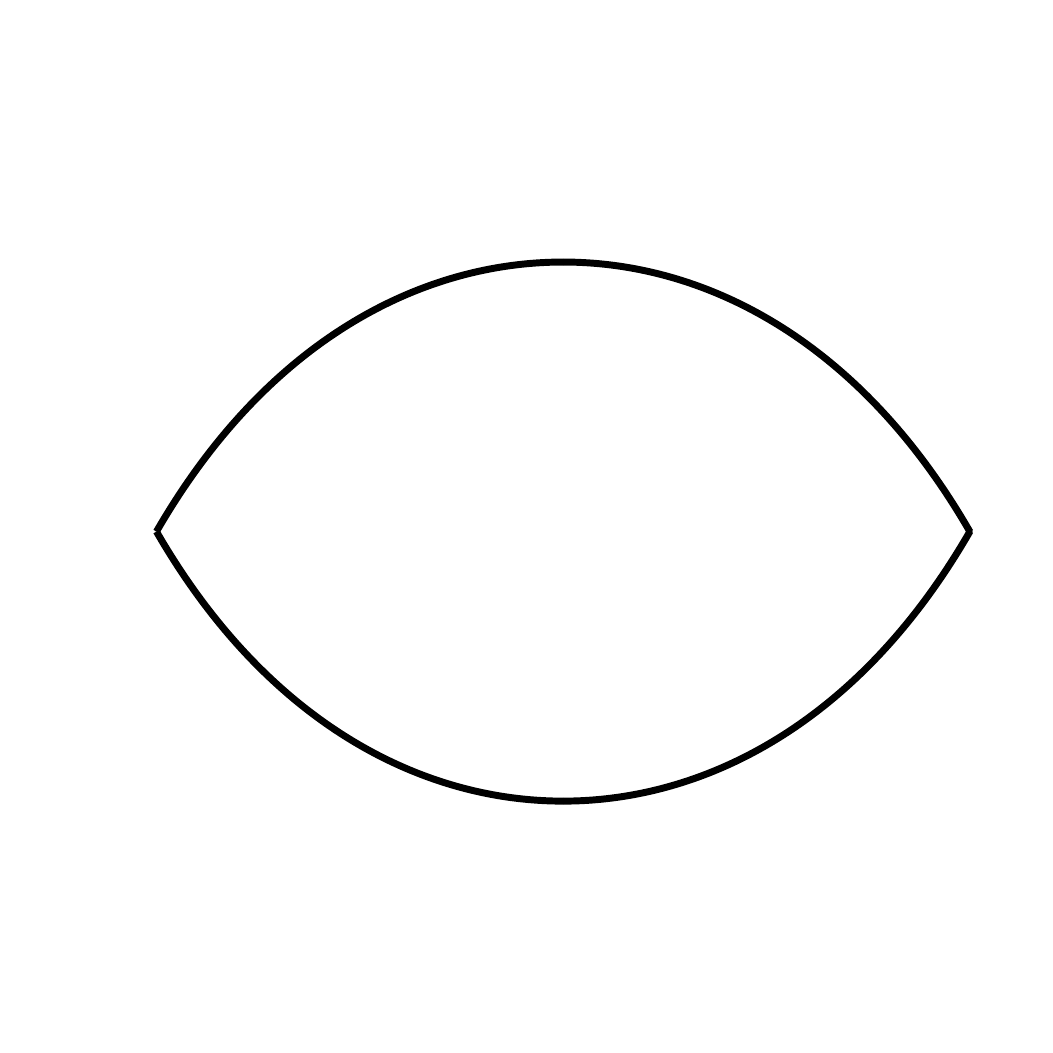} };
		
		\draw[<->,thick] (-3,0)--(3,0) node[right]{};
		\draw[<->,thick] (0,-2)--(0,2) node[above]{};
		
		\fill ( 0.0,0.0) circle (2.0pt) node[below right]{$0$};
		\fill ( 2.0,0.0) circle (2.0pt) node[below right]{$1$};
		\fill (-2.0,0.0) circle (2.0pt) node[below left]{$-1$};
		
		\fill ( 0.0, 1.32548) circle (2.0pt) node[above right]{$ic$};
		\fill ( 0.0,-1.32548) circle (2.0pt) node[below right]{$-ic$};
		
		\draw (8.0,0) circle (2.0cm) node[left]{$ $};
		\node[ann] at (8.0,0.0) {$|z|<1$};
		
		\draw[->] (3,1) to[out=30,in=150] node[midway,above] {$g(w)$} (5,1);
		\draw[<-] (3,-1) to[out=-30,in=-150] node[midway,below] {$g^{-1}(z)$} (5,-1);
	\end{tikzpicture}
	\caption{\emph{Illustration of Definition \ref{def1} and the action of the mapping $g(z)$: Left) Domain $K$, used to characterize the behavior of Bessel functions of complex argument. Right) the domain $D$ is the unit disk.} }\label{fig:domain_map}
\end{figure}
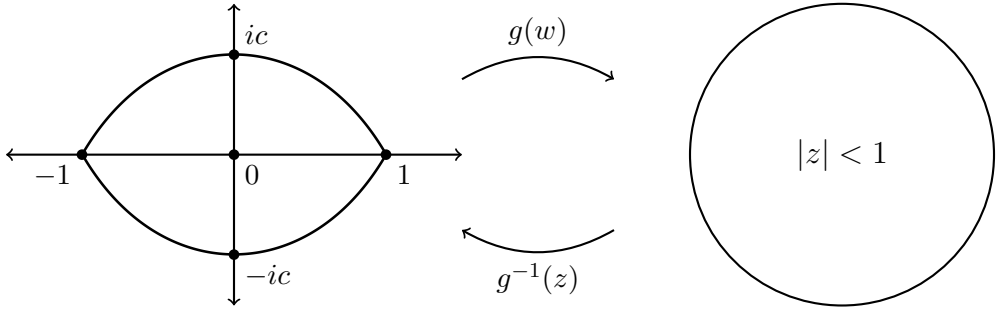

The case $|\kappa|\gg 1$, with $\kappa=\om h/2$, is of central importance for the paper. To simplify the presentation of this case, we map a particular region of the complex plane on concentric disks.
\begin{definition}\label{def1} Let $D=\{z:|z|<1\}$ and denote by $K\subset \mx C$ the open region enclosed by the parametric curve $w=\pm (\tau \coth\tau-\tau)^{1/2}\pm i(\tau^2-\tau\tanh \tau)^{1/2}$, $0<\tau<\tau_0$, where $\tau_0$ is the solution of $\coth \tau=\tau$. For $\delta>0$, define $S_\delta:=\{z:1-\delta<|z|<1+\delta\}$. Then, we define a continuous bijective mapping $g:\mx C\rightarrow \mx C$, where the range satisfies $\text{Ran}\,g|_{K}=D$ and $g$ maps the set $\{z:\,\text{dist}\,(z,\partial K)<\delta\}$ on $S_{\delta}$. Finally, $g$ is the identity map on $\R$ (see Fig. \ref{fig:domain_map}).
\end{definition}

The mapping $g$ in Definition \ref{def1} allows us to split the complex plane in three zones by using concentric disks. 
The region $K$ has previously been used \cite{olver54,olver_DLMF} to derive a uniform expansion of the Bessel functions $J_{\nu}(\nu z),\,Y_{\nu}(\nu z)$ for large order $\nu$ and complex argument $z$. The behavior of  $J_{\nu}(\nu z), Y_{\nu}( \nu z)$ depends on the location of $z$ with respect to $K$. 
Particularly, we use the result that 
$|J_{\nu}(\nu z)|$ decays or grows rapidly
with $|\im z|$, 
depending on whether $z$ lies inside or outside $K$. 
Similarly, $|Y_{\nu}(\nu z)|$ grows unbounded 
as dist$(z,\partial K)$ increases. Finally, $Y_{\nu}(\nu z)$ has complex zeros outside $K$, which become poles of $Q_\nu(\nu z)$. Those zeros are located close to $\partial K$, in the transition zone.

The following lemma is used to extend the dispersion analysis to complex $\omega$.
\begin{lemma}\label{lemmaQ2} Define $\delta_\nu=(x_2-x_1)/2$, where $x_1,\,x_2\in\mx R^+$ are the first two real roots of $Y_\nu (x)$. Let 
$\mx H^+:=\{z : 0<\arg z<\pi,\,\,\hbox{with}\,\,\im z > \delta_\nu\}$ and 
$\mx H^-:=\{z : -\pi<\arg z<0,\,\,\hbox{with}\,\,\im z < -\delta_\nu\}$. 
Then, for $\nu,|z|$ large, and $|\nu g(z/\nu)|>\nu+\nu^{1/3}$ the following approximations hold:
\begin{equation}
Q_\nu(z)\approx\mp i,\quad z\in \mx H^{\pm}.
\label{eq:lemma2}
\end{equation}
\end{lemma}
\begin{proof} 
By the conditions stated above, \cite[eqs. (9.2.3), (9.2.4)]{abramowitz+stegun} hold, and $g(z/\nu)\in \C\setminus (\overline{D\cup S_\delta})$. From the identities 
$
J_\nu(z) = \tfrac{1}{2}(H^{(1)}_\nu(z)+H^{(2)}_\nu(z)),\,
Y_\nu(z) = \tfrac{1}{2i}(H^{(1)}_\nu(z)-H^{(2)}_\nu(z))
$ and \cite[eqs. (9.2.3), (9.2.4)]{abramowitz+stegun}, we obtain the quotient
\begin{equation}
Q_\nu(z)=i\,\fr{H^{(1)}_\nu(z)+H^{(2)}_\nu(z)}{H^{(1)}_\nu(z)-H^{(2)}(z)}
\approx 
i\,\fr{e^{ix}e^{-y}e^{i\theta}+e^{-ix}e^{y}e^{-i\theta}}
{e^{ix}e^{-y}e^{i\theta}-e^{-ix}e^{y}e^{-i\theta}}\,,\quad
x,y\in \mx R, \,\,\theta=-\frac{\nu\pi}{2}-\frac{\pi}{4},
\label{eq:lemma2_proof}
\end{equation}
where we set $z:=x+iy$. Assume that $z\in\mx H^\pm$, then \eqref{eq:lemma2_proof} implies that $Q_\nu(z)\approx \mp i$. 
\end{proof}

The following theorem extends \cite[Thm 3.3]{Ainsworth04} to complex frequencies $\omega$.
\begin{theorem} \label{thm:high_frequency}
 Let $p\in \mx N,\,\, \om,\kappa\in\mx C$, assume that $|\om h|\gg 1$, and take $\sigma=(2p+1)\cdot g(\om h/(2p+1))$. Then the error $|\mc E^p|$ in the discrete dispersion relation passes through three distinct phases as the order $p$ is increased:
\begin{enumerate}
	\item [{\rm i)}]Non-decaying zone: For $2p+1<|\sigma|-o(|\sigma|^{1/3})$, the difference $|\mc E^p|$ does not decay as $p$ is increased. For the case with small $|\im \om|$, 
then $|\mc E^p|$ oscillates, but does not decay, as $p$ is increased.
	\item [{\rm ii)}] Transition zone: For $|\sigma|-o(|\sigma|^{1/3})<2p+1<|\sigma|+o(|\sigma|^{1/3})$, and $\om h$ not a pole of \eqref{eq:Ep}, the error $|\mc E^p|$ decays at rate:
	\begin{equation}
	|\mc E^p|\approx \left| \sin(\om h) \fr{\textup{Ai}(\xi)}{\textup{Bi}(\xi)} \right|, 
	\,\,\xi=-\left(\fr{2}{p}\right)^{1/3}\fr{\om h-2p}{2},
	\label{eq:airy_decay}
	\end{equation}
	where \textup{Ai,\,Bi} denote Airy functions.
	\item [{\rm iii)}] Superexponential decay: For $2p+1>|\sigma|+o(|\sigma|^{1/3})$, $|\mc E^p|$ decreases at a superexponential decay rate:
	\begin{equation}
	|\mc E^p|\approx \left| \fr{\sin(\om h)}{2} f(\sqrt{1-(\om h/(2p+1))^2})^{p+1/2}\right|,
	\label{eq:exp_decay}
	\end{equation}
	where $f: w\rightarrow (1-w)/(1+w)\exp(2w)$, with $|f(w)|<1$. In particular, for the case $2p+1>|\om h|e/2$ with $e=\exp(1)$, we have
	\begin{equation}
	|\mc E^p|\approx \left| \fr{\sin(\om h)}{2} \left[\fr{\om h e}{2(2p+1)}\right]^{2p+1}\right|.
	\label{eq:exp_decay_faster}
	\end{equation}	
\end{enumerate}
\end{theorem}

\begin{proof} In the case $\om\in\mx \R^+$, then $\sigma=\om h$ and the theorem was proved in \cite[Thm. 3.3]{Ainsworth04}. In the rest of the proof we assume that $\im \omega \neq 0$ and set $\kappa=\om h/2$. For $\om h$ fixed, the error $|\mc E^p|$ in \eqref{eq:Ep} decays as $|Q_m|$ goes to zero. Then we describe the behavior of $|Q_m(\kappa)|$ in different regions of the complex plane as $p$ is increased. 

\rm i) In this regime $|g(\kappa/m)|>1$, which implies that the point $\kappa$ is in the complement of $K$. Let $\delta_m=(x_2-x_1)/2$, where $x_1,x_2$ are the first two positive roots of $Y_m(x)$. Then, $Q_m(\kappa)$ is close to a pole on the real line if $-\delta_m < \im \kappa < \delta_m$. If $|\im \kappa|<\delta_m$, then by \cite[eqs. (9.2.1), (9.2.2)]{abramowitz+stegun} the function $Q_m$ is oscillatory and dominated by $\re\kappa$, with an error $e^{|\im \kappa|}\mc O(|\kappa|^{-1})$; Compare with \cite[Sec. A.1.1]{Ainsworth04}. 

For $|\im \kappa|>\delta_m$, we let $\hat z=\kappa/m$, and write $Q_m(m \hat z)$ in terms of Airy functions. 
For this, we use the uniform asymptotic expansions \cite[(10.20.4),(10.20.5)]{olver_DLMF} with $k=0$, 
and neglect the $m^{-5/3}$ term.
The transformation \cite[(10.20.3)]{olver_DLMF} can be analytically continued to the complex plane, provided that $\hat z$ is located outside $K$. This is true by the assumptions of theorem. 
Then, we use \cite[(9.6.6),(9.6.8)]{olver_DLMF} to obtain a representation in terms of the Bessel functions $J_{1/3},\,J_{-1/3}$ of fixed order, which is analogous to \cite[(A.7)]{Ainsworth04}. 
Similarly as in \cite[Sect. A.1.1]{Ainsworth04}, we use Watson formulas \cite[(8.440-1),(8.440-2)]{grad07} to obtain \cite[(A.8)]{Ainsworth04} that holds for complex $\om$. 
Finally, the argumentation given in \cite[Sect. A.1.1]{Ainsworth04} also holds in the present case. 
Additionally, Lemma \ref{lemmaQ2} implies that for the current region $|Q_m(\kappa)|\approx 1$, while the order of $|\mc E^p|$ is dominated by $|\sin \kappa|$, which grows exponentially with $|\im \kappa|$.

\rm ii) As $\kappa$ is not a pole of $Q_m(\kappa)$, we use asymptotic expansions for Bessel functions,
which are valid in the transition zone. Particularly, we truncate the series \cite[(10.19.8)]{olver_DLMF}
with $k=0$. For $\kappa$ fixed, 
the resulting formula becomes
\begin{equation}
	Q_m(\kappa)\approx -\fr{\textup{Ai}(\xi)}{\textup{Bi}(\xi)},\,\,\, \xi=-\left(\fr{2}{m}\right)^{1/3}(\kappa-m).
\label{eq:Q_airy}
	\end{equation}
Then for large $p$, we have that $2m\approx 2p$, and with the use of \eqref{eq:Ep}
we obtain \eqref{eq:airy_decay}. The same result is obtained by using  
Olver's uniform expansions in \cite[(10.20.4), (10.20.5)]{olver_DLMF}.
As a remark, we mention that even if \cite[(A.9)]{Ainsworth04} is a valid
linearization of \eqref{eq:Q_airy}, the approximation is quite rough in the complex case because 
$\min_{m\in \mathbb{Z}+1/2} |\kappa-m|\geq |\im \kappa|$ may be large. However, in the case $\im \kappa=0$, the linearization becomes a close approximation to \eqref{eq:Q_airy} in the transition zone, which implies that $\mc E^p$ decays algebraically at rate $\mc O(p^{-1/3})$.  \\


\rm iii) In this region $|g(\kappa/m)|<1$, or equivalently $\kappa/m\in K$. By \cite[Sec. 4]{olver54}, it follows that $|Q_m(\kappa)|$ decays. The approximation \cite[(A.10)]{Ainsworth04} is justified for complex arguments if the transformations $w=\sqrt{1-(\kappa/m)^2}$, $z=m(\arctanh w-w)$, and $z=\tfrac{2}{3}\xi^{3/2}$ are analytic continuations of its real valued versions. This is easily verified by writing $z\circ w(\kappa)$ explicitly, using the identity $\arctanh w\equiv \log((1+w)/\sqrt{1-w^2})$. After a direct calculation we obtain $\sqrt{1-w^2}=\kappa/m$. Substitution of these into $z/m=(\arctanh w-w)$ results in equation \cite[eq. (4.6)]{olver54}, which holds for $z/m\in K$. Finally, by having validated \cite[(A.10)]{Ainsworth04}, we proceed as in \cite{Ainsworth04} and derive (A.12), (A.13), (A.14), and (A.15). Hence, the results \cite[Sec. A.2.2]{Ainsworth04} hold, which finalizes the proof. Alternatively, \eqref{eq:exp_decay_faster} is straightforwardly obtained by the use of \cite[(10.19.1),(10.19.2)]{olver_DLMF} and \eqref{eq:Ep}.
\end{proof}


As an illustration of the results in this section, we present in Fig. \ref{fig:lag_1d}, a dispersion comparison between $\cos(\om h),\,R_p(\om h)$ for $\om=20-0.5i$ and polynomial order $p=20$, followed by a convergence plot showing exponential decay for both real and imaginary parts of the difference $R_p(\om h)-\cos(\om h)$. Notice that convergence starts at $p=20$ and $h=2$, where we see in the two first panels that $h=2$ is the largest mesh size where the difference is small.
\begin{figure}
	\centering
	\begin{tikzpicture}[thick,scale=1.0, every node/.style={scale=1.0}]	
	\tikzstyle{ann} = [fill=none,font=\large,inner sep=4pt]
	
	\draw( 0.00, 0.0) node { \includegraphics[scale=0.7]{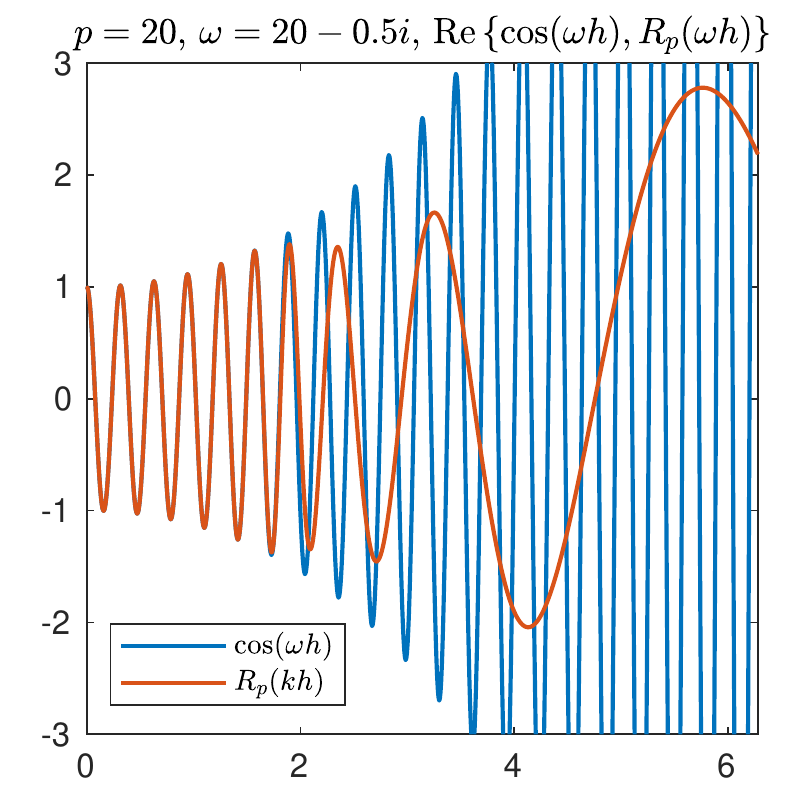} };
	\draw( 5.60, 0.0) node { \includegraphics[scale=0.7]{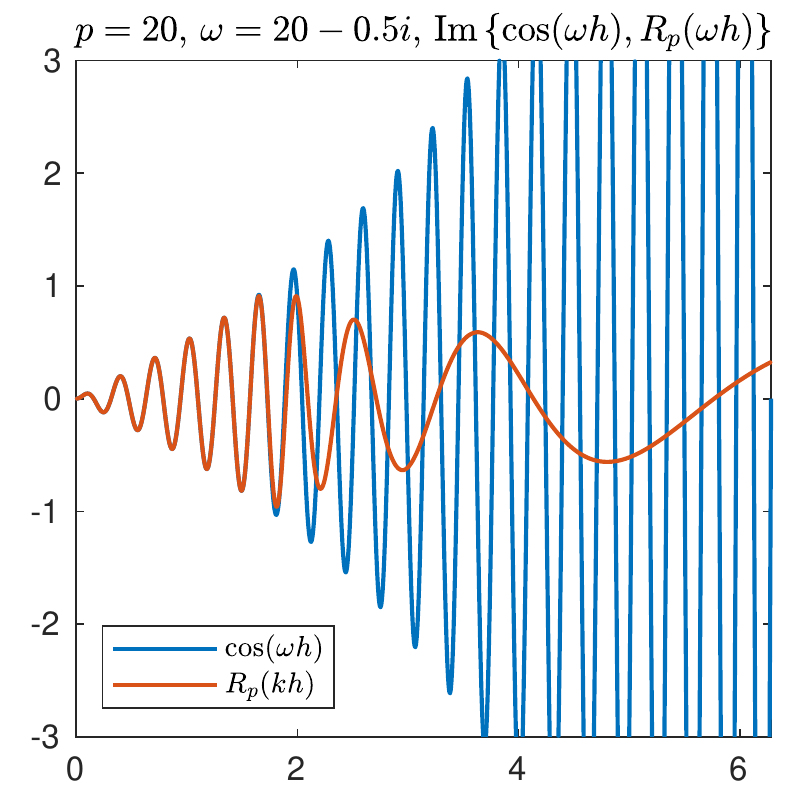} };
	\draw(10.50, 0.0) node { \includegraphics[scale=0.7]{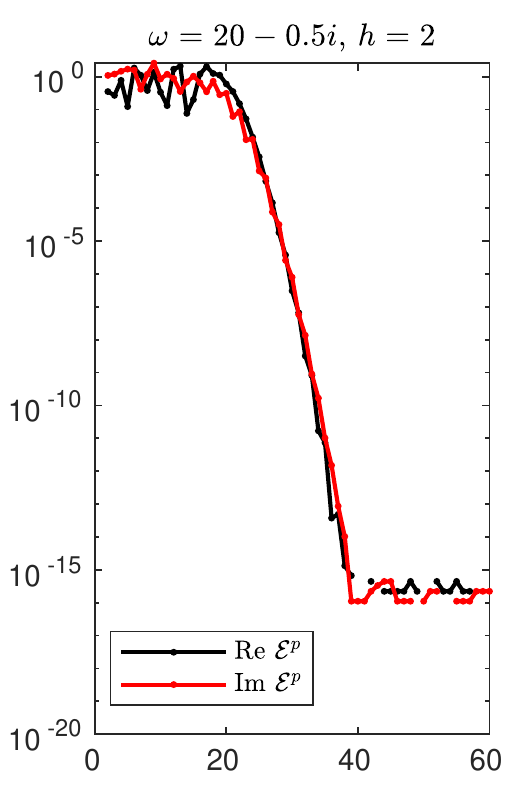} };
	
	\end{tikzpicture}
	\caption{\emph{In the first two panels we present a comparison of $\cos(\om h),\,R_p(\om h)$ for $\om\in\mx C$, by using $p=20$ and show left) real parts, middle) imaginary parts versus $h$ in the horizontal axis. In the right panel, we show convergence for $\re {\mc E^p},\, \im {\mc E^p}$ vs. $p$, with fixed $h=2.0$.}}\label{fig:lag_1d}
\end{figure}
The natural extension of the discrete dispersion relation to higher dimensions on tensor product meshes is presented in \cite[Sec. 2.3]{Ainsworth04}. 
\ce{We} refer the interested reader
to the work in \cite{AinsworthMaxwell04}, where it is
shown that the results from \cite{Ainsworth04} are also important 
for the analysis of the dispersive properties of high order edge
FE used for the full Maxwell equations.

\reviewerOne{In the following sections we make use of the dispersion analysis revised in the current section, for the FE computation of resonances in one and two dimensions with quadrilateral elements. 
Particularly, we design \emph{a-priori} strategies for problems with piecewise constant coefficients, considering each element in our triangulation separately.}

\noindent
\subsubsection*{Dispersion analysis for piecewise constant refractive index:}\label{sec:multidis1d}
Consider a problem similar to \eqref{eq:var_lag1d} with $\om^2$ replaced with $\om^2n^2$, and a refractive index profile $n$ defined by 
the constants $n_j$ for $x\in I_j$. The problem for $u$ can then be formulated as the solution of a linear system with matrix entries $\scatt_{l,m}$ given by 
the exponentials $c_{l,m} e^{i n_j \om x_l}$ defined in $I_{j}$.
For a piecewise polynomial approximation $u^\fem$, we obtain a corresponding matrix $\scatt^\fem(\om)$ that approximates $\scatt(\om)$.
Then, each entry $\mc E_{l,m}=|\scatt_{l,m}-\scatt_{l,m}^\fem|$ can be treated similarly as the dispersive error \eqref{eq:Ep}, which motivates the use of the dispersion analysis described in Section \ref{sec:description_lag}.


\section{Discretization, a-priori refinement strategies, and solution of the nonlinear eigenvalue problem}\label{sec:Disc_n_NEP}
In this section we describe the computational details used to obtain the approximated resonant pairs as the solution of the nonlinear eigenvalue problems described in Section \ref{sec:resonances}. In particular, we introduce an initial FE triangulation, which by assumption is conforming and regular. Additionally, we are given a region in $\C$ where we search for eigenvalues. 
Then, the initial triangulation
is refined  
depending on the permittivity function 
defined in the computational domain. 
The mesh refinement is performed following a-priori strategies that are presented below.
We motivate the extension to higher dimensions and describe how to obtain the resulting matrix problem in dimensions $d=1,2$. Finally, we describe our strategy for the solution of the resulting nonlinear eigenvalue problem.

\subsection{Discretization with the FE method}\label{sec:discretization}
The domain $\Omega\subset\mx R^d$ is covered with a regular and quasi uniform finite element mesh $\mc T(\Om_a)$ consisting of elements $\{K_j\}^N_{j=1}$. The mesh is designed such that the permittivity function $\eps (\om)$ is constant in each $K_j$. 
Let $h_j$ be the length of the largest diagonal of the non-curved primitive $K_j$
and denote by $h$ the maximum mesh size $h:=\max_j{h_j}$. 

In the following, $\mc P_p$ denotes the space of polynomials on $\mx R^d$ of degree $\leq p$ in each coordinate and the script $\fem:=\{h,p\}$ labels the discrete pairs.
Furthermore, we assign per element $K_j$ a local polynomial degree $p_j$ satisfying $1\leq p_j \leq p$. We define the finite element space $S^{\fem}(\Omega):=\{u\in H^1(\Omega):\left. u\right|_{K_j} \in\mc P_{p_j}(K_j)\,\,\hbox{for}\,\,K_j\in\mc T\}$, and $N:=\dim(S^{\fem}(\Omega))$ \cite{Babuska92}. Furthermore, in the case $d=2$, all the computations are done in the approximated domain $\Omega^\fem$ by using curvilinear elements following standard procedures \cite{Babuska92}. The used FE meshes are \emph{shape regular} in the sense of \cite[Sec. 4.3]{Schwab1998}, and consist of  quadrilaterals with curvilinear edges that deviate slightly from their \emph{non-curved} primitives. Finally, we assume that the PML is set up following the discussions in \cite{kim09,araujo+engstrom+2017}, which accounts for large enough $\ell$ and $\sigma_0$ such that the search region is feasible \cite{araujo+engstrom+2017}. We assume that the FE space in $\Om_{\PML}$ is good enough and concentrate on the physical region $\Om_a$.

\subsection{A-priori refinement strategies}\label{sec:con_assum}

In the current section, we present two novel \emph{a-priori} refinement strategies to be used for the computation of Helmholtz resonances with piecewise constant coefficients. 
Following the dispersion analysis sketched in Section \ref{sec:dispersion}, 
the aim is that a given initial mesh is refined (a-priori)
such that the resulting mesh satisfies the conditions
for superexponential decay of the dispersive error in Theorems \ref{thm:low_frequency} and \ref{thm:high_frequency}.
Resonances are then approximated by the eigenvalues of a rational matrix-valued function. The refractive index is by assumption the constant $n_j$ on element $K_j$. Then for $x\in K_j$, both the TM-case and the TE-case can be written in the form 
$-\Delta u_j -(\om_j n_j)^2 u_j =0$. A shift value $\mu$ is then introduced, and eigenvalues are approximated in the closed ball $\Lambda:=\overline{B(\mu,r_\mu)}$ centered at $\mu$ with radius $r_\mu$ in the complex plane. 
In practice $r_\mu$ should be large enough to allow the computation of several eigenvalues at once, but small enough such that $n_j(\om)$ does not vary excessively. We choose $\Lambda$ such that all $n_j(\om)$ are continuous functions for $\om\in\Lambda$.

Additionally, we define
\begin{equation}\label{eq:piecewise_refract}
	\lambda_j:=\argmax_{\om\in \Lambda} |n_j (\om)|,\,\,\,
	\tilde n_j:= n_j(\lambda_j),
	\,\,\,\,\hbox{and let}\,\,\,\,
	k_j:=
	\begin{cases} 
		|\mu| & \text{if}\,\, |\tilde n_j|<1 \\
		|\tilde n_j\mu| & \text{otherwise},
	\end{cases}
\end{equation}
where we assume that the arguments of the maxima consists of one point $\lambda_j$. For the description of our strategy, we use the following definition of an extended mesh.
\begin{definition}\label{def:ext_grid} 
{\bf Extended mesh:} Let $K_j$ be a one dimensional element of size $h_j:=x^j_2-x^j_1$ defined by the nodes in $[x^j_1,x^j_2]$. The extended mesh $\mc M(K_j)$ is then defined as the partition with points $\hat x_l=lh_j$, for all $l\in\mx Z$.
\end{definition} 

From this stage we estimate locally the dispersive properties of the finite element space, motivated by the results in Section \ref{sec:description_lag}. 
Refinement strategies can be designed so that for given $k_j$, a finite element space defined over $\mc M(K_j)$ satisfies the conditions in Theorems \ref{thm:low_frequency}, and \ref{thm:high_frequency} for superexponential decay.
In this way, we obtain target values for $p_j$ and $h_j$ that can be easily computed in each FE cell. In Section \ref{sec:strategies}, we present two alternative ways to achieve this.

In an a-priori refinement strategy, we first estimate the initial state of an input FE mesh by using a \emph{global mesh indicator} $\gamma_0$. This is done by first checking the constraints of the desired mesh: the minimum mesh size or maximum polynomial degree that we are allowed to use. This information is contained in $k_0,h_0$, defined below. Consecutively, we check each cell $K_j$ and refine $h_j,p_j$ if needed according to a specific goal.

\reviewerOne{
The goal of the refinement is to guarantee that in each $K_j$, our FE eigenfunctions restricted to $K_j$, satisfy the conditions for superexponential decay of the error on the extended mesh $\mc M(K_j)$
as it has been described in Section \ref{sec:dispersion}. Particularly, 
equations \eqref{eq:small_wh} and \eqref{eq:exp_decay_faster} suggest a relationship between $p$ and $h$ when the dispersive error $\mc E^p$ is in the superexponential decay region.
In the case of small $|\om h|$, we can rewrite the leading term in \eqref{eq:small_wh} as
$\mc E^p\approx c \cdot \gamma^2$, 
with $c$ and $\gamma$ given by
\begin{equation}\label{eq:dispersive_control}
	c=\fr{1}{4} \fr{\om h}{2p+1} \left( \fr{e}{2} \right)^{2p},\quad \gamma:=\left( \fr{\om h}{2p} \right)^{p},
\end{equation}
where we have used Stirling's approximation. 
For increasing $p$, the decay of $|\gamma|^2$ is faster than the increase of $|c|$. 
From Theorem \ref{thm:high_frequency} with $2p>|\om h|e/2$ it follows that the decay of $|\gamma|$ in \eqref{eq:dispersive_control} implies the decay of $|\mc E^p|$ in \eqref{eq:exp_decay_faster}.

Finally, for Helmholtz problems in higher space dimensions with real $\omega$, 
the quotient $\gamma$ is found to play an important role in the estimation of the dispersive FE errors \cite{ihl98,Sauter10,MR3094621}. For all these reasons, $\gamma$ is a natural indicator to be used for designing a-priori strategies for the control of dispersive FE errors.
}

The proposed refinement strategy for enriching the finite element space is based on the goal
\begin{goal}\label{def:goal}
For given $k_j$ and $\gamma_0$, find $h_j$ and $p_j$ such that the condition
\begin{equation}
	\left(\frac{k_j h_j}{2p_j} \right)^{p_j}\leq \gamma_0
	\label{eq:mesh_goal}
\end{equation}
is satisfied in $K_j$.
\end{goal}
\begin{remark}\label{rem1}
Depending on $k_j$ and $\gamma_0$, Goal \ref{def:goal} may be unfeasible. Then we say that the target finite element space is unreachable for the given input parameters.
\end{remark}

\subsection{Proposed refinement strategies}\label{sec:strategies}
In the following, we present two strategies in order to verify that the condition \eqref{eq:mesh_goal} is satisfied in each FE cell $K_j$.
The proposed a-priori 
refinement 
procedure is sketched in Algorithm \ref{algo}.

In the $h$-strategy we assume that $p_j$ is fixed, and we perform a standard $h$-refinement: split the cell $K_j$ in $2^d$ new cells and update $\mc T(\Om_a)$ \cite{bangerth09}. Similarly, in the $p$-strategy we keep $h_j$ fixed and find a suitable $p_j$.
\begin{algorithm}
  \caption{A-priori $hp$-FE refinement strategy\label{algo}}
  \DontPrintSemicolon
  \SetKwData{KwTo}{\!$,\dots,$}
  \SetKwInput{Input}{Input}
  \Input{$p_0$, $K_j$, $h_j$, $n_j$, $\Lambda$, and $\mu\in\Lambda$. Each element is assigned $p_j=p_0$.}
  Compute $k_j$ in each element \;
  Set $h_0$, $k_0$ and compute $\gamma_0$
  according to the strategies \ref{sec:h_str} or \ref{sec:p_str} \; 
  Check feasibility of refinement:
  \lIf{$\gamma_0\geq 1$}{restart with modified input parameters} \;
  \For{$j=1$ \KwTo $N$}{
     Check Goal \ref{def:goal} for element $K_j$ \;
     \lIf{Goal \ref{def:goal} not satisfied}{refine $h_j$ or $p_j$} \;
  }
  Start the assembly of the FE matrices \;
  Start the  NEP solver with shift $\mu$ and compute the pairs $(u^\fem_m,\om^\fem_m)$ \;
\end{algorithm}


The initial mesh $\mc T(\Om_a)$ is by assumption a conforming triangulation of $\Omega_a$ without \emph{ghost nodes}; see \cite{bangerth09}. Then, a fixed polynomial degree $p_j=p_0\geq 1$ is assigned to each element $K_j$, $j\in\mc I_0:=\{j:1\leq j\leq N\}$.
The refractive index profile $n_j$ is known per element, and the region $\Lambda\subset\C$ containing the shift $\mu$ has been specified. From \eqref{eq:piecewise_refract}, we assign the constants $k_j$ to each element $K_j$ of mesh size $h_j$.
Then, a parameter $\gamma_0$ is introduced in order to 
account for the state of the initial mesh.
In the case $\gamma_0<1$, we perform an a-priori refinement of the mesh, and continue with the steps of Algorithm \ref{algo}. Otherwise we go back to beginning of Algorithm \ref{algo} and ask the user to modify the input parameters.

Below we propose two strategies for achieving Goal \ref{def:goal}.
\subsubsection{$h$-strategy ($sh$-FE)}\label{sec:h_str}  
Let $k_0=\min_{j\in\mc I_0} k_j$ and define $\mc I:=\{j:k_j=k_0\}$, $h_0:=\min_{j\in\mc I} h_j$. 
Then, we check the state of the initial mesh 
for given $k_0,\,h_0 \in \mx R$, and define the global mesh indicator as
\[
	\gamma_0:=\left(\frac{k_0 h_0}{2p_0} \right)^{p_0}.
\]
Consecutively, we perform $h$-refinements in all cells $K_j$, with $j\in\mc I_0$, such that Goal \ref{def:goal} is satisfied.
The last statement implies satisfying the condition
\begin{equation}
	h_j \leq  \left(\frac{k_0}{k_j}\right) h_0.
	\label{eq:h_goal}
\end{equation} 
After this, we iteratively refine cells $K_j$ such that each cell has neighboring cells that are at most one level of refinement higher than itself. For this, we allow ghost nodes, and we do not coarsen cells.
\subsubsection{$p$-strategy ($hp$-FE)}\label{sec:p_str}
Let $k_0=\max_{j\in\mc I_0} k_j$ and define $\mc I:=\{j:k_j=k_0\}$, $h_0:=\max_{j\in\mc I} h_j$. We compute the $p_j$ corresponding to the element $K_j$ such that \eqref{eq:mesh_goal} is satisfied. The last statement requires solving for the zeros $z_i$ of the nonlinear equation
\begin{equation}
	F_j(z):=\left( \frac{k_j h_j }{2z} \right)^{z}-\gamma_0=0,
	\,\,\,\text{with}\,\,\,\gamma_0:=\left(\frac{k_0 h_0}{2p_0} \right)^{p_0}.
	\label{eq:scheme_nonlinear}
\end{equation}
We choose the solution $z_i\geq 1$ that minimizes $|p_0-z_i|$. Finally, we take 
$p_j:=\myceil{z_i}$,  
where $\myceil{z_i}$ is the smallest integer greater than or equal to $z_i$.
\begin{remark}\label{rem2}
In order to solve \eqref{eq:scheme_nonlinear} we compute the derivative with respect to $z$, and solve by using a scalar Newton-Raphson root finder. We use $z_0=p_0$ as initial guess and search for solutions in $z_i\in[1,p_0]$. If the only roots are such that $z_i<1$, then the resulting $z_i$ is not feasible. Possible workarounds are to increase the input parameter $p_0$, or a further uniform $h$-refinement may be needed before starting the strategies. 
\end{remark} 
\subsection{Assembly of FE matrices}\label{assembly}

In this subsection, we consider the assembly of the FE matrices for the 1D problem  \eqref{eq:1D} and the 2D problem \eqref{eq:eig_prob2D}. Assume that the set of
shape functions $\{\varphi_1,\dots,\varphi_N\}$ is a basis of the space $S^{\fem}(\Omega^\fem)$ defined in section \ref{sec:discretization}. Then $u_\fem\in S^{\fem}(\Omega^\fem)$ has the representation
\begin{equation}
u_\fem=\sum_{j=1}^N \xi_j\,\varphi_j.
\label{eq:fem_ans}
\end{equation}
\subsubsection{Discrete problem in 1D}\label{sec:dtn_1d}

From \eqref{eq:1D}, with $\Omdtn^\fem\subset\R$, we state the corresponding finite element problem: Find $u_\fem\in S^{\fem}(\Omdtn^\fem)\setminus \{0\}$ and $\om_\fem \in\mc D$, such that $t_1(\om_\fem)[u_\fem,v]=0$ is satisfied for all $v\in S^{\fem}(\Omdtn^\fem)$.

Similarly, we state the corresponding matrix problem: Find the eigenpairs $(\om_\fem,\xi)\in\mc D\times \mx C^{N}$ such that
\begin{equation}\label{eq:dtn_discrete}
	T_1(\om_\fem) \xi:=\left(\sum_{m=0}^{N_r} 
	\left\{\rho_m(\om_\fem )A_m-\om_\fem^2 \eta_m (\om_\fem )M_m\right\}
	-i\om_\fem \rho_0 E \right){\xi}=0,
\end{equation}
with finite element matrices
\begin{equation}
	A^m_{ij}=\int_{\Omdtn_m} \varphi'_j\varphi'_i\,dx,\,\,M^m_{ij}=\int_{\Omdtn_m} \varphi_j\varphi_i\,dx,\,\,
	E_{ij}=\left(\varphi_j(-\dtn)\varphi_i(-\dtn)+\varphi_j(\dtn)\varphi_i(\dtn)\right),
	\label{eq:mat_dtn}
\end{equation}
for $m=0,\ldots,N_r$.
\subsubsection{Discrete problem in 2D}\label{sec:pml_2d}
From \eqref{eq:eig_prob2D}, with $\Om^\fem\subset\R^2$, we state the  corresponding finite element problem: Find $u_\fem\in S^{\fem}(\Om^\fem)\setminus \{0\}$ and $\om_\fem \in\mc D$, such that $t_2(\om_\fem)[u_\fem,v]=0$ is satisfied for all $v\in S^{\fem}(\Om^\fem)$.  The entries in the finite element matrices become
\begin{equation}
\begin{array}{ll}
	A^0_{ij}= (\nabla\varphi_j, \nabla\varphi_i)_{\Om^\fem_0}+
						(\mc A\nabla\varphi_j, \nabla\varphi_i)_{\Om^\fem_\PML},	
	& 
	M^0_{ij}=(\varphi_j, \varphi_i)_{\Om^\fem_0} +
				 	 (\mc B \,\varphi_j, \varphi_i)_{\Om^\fem_\PML} \\[2mm]
	A^m_{ij}= (\nabla\varphi_j, \nabla\varphi_i)_{\Om^\fem_m}, & M^m_{ij}=(\varphi_j, \varphi_i)_{\Om^\fem_m},
\end{array}
\label{eq:def_fem}
\end{equation} 
with $m=1,\ldots,N_r$.

The nonlinear matrix eigenvalue problem reads: Find the eigenpairs $(\om_\fem ,{\xi})\in\mc D\times \C^N\setminus \{0\}$ such that
\begin{equation}
T_2 (\om_\fem) \,{\xi}:=\left(\sum_{m=0}^{N_r} \rho_m(\om_\fem )A_m-\om_\fem^2 \eta_m (\om_\fem )M_m\right){\xi}=0.
\label{eq:eig_matrix}
\end{equation}

All numerical experiments have been carried out using the finite element library \emph{deal.II} \cite{dealII82} with Gauss-Lobatto shape functions \cite[Sec. 1.2.3]{solin04}. For fast assembly and computations with complex numbers the package PETSc \cite{petsc-efficient} is used.

The computational platform used for the executions is Tirant 3, consisting of 336 computing nodes and on each of them two Intel Xeon SandyBridge E5-2670 processors (16 cores each). The processors, running at 2.6 GHz with 32 GB of memory, are interconnected with an Infiniband FDR10 network. All runs are scheduled for at most 4 MPI processes per node.

\subsection{Solution of the nonlinear eigenvalue problem}\label{sec:NLEIGS}

For solving the nonlinear eigenvalue problems we use SLEPc \cite{slepc05+roman} and in particular its NEP module \cite{campos18}. We provide a target value $\mu$ and request to compute a few eigenvalues (and corresponding eigenvectors) close to that value. This process is repeated for several values of $\mu$ in order to cover the region of interest.

The user interface to SLEPc allows the representation of the nonlinear eigenproblem by passing a list of matrices and a list of corresponding scalar nonlinear functions.
In our case, the matrix problem to be solved is \eqref{eq:eig_matrix},
from where the functions that multiply the matrix coefficients are either polynomial ($1$ and $-\omega^2$) or rational ($-\omega^2\eps(\omega)$ and $1/\eps(\omega)$). SLEPc provides a simple mechanism to define these functions, either by providing the coefficients of numerator and denominator, or by combining other functions (e.g., additive combination as required in \eqref{eq:drude_lorentz}).

In this work, we use SLEPc's implementation of the NLEIGS method \cite{guttel14}, which was developed in the course of this paper. We next provide a brief description of this method, together with some implementation details that improve the solver's efficiency. We express the eigenvalue problem as
\begin{equation}\label{eq:split}
T(\omega)\xi=0,\qquad\text{with}\quad T(\omega)=\sum_{i=1}^{d}A_if_i(\omega),
\end{equation}
where $\omega$ is the eigenvalue, $\xi$ is the eigenvector, $A_i$ are constant matrices and $f_i$ are scalar nonlinear functions. NLEIGS aims at finding eigenvalues located inside a certain region of the complex plane $\Sigma$.

For this, it first approximates $T$ in that region with a rational matrix $F_d$ whose poles are selected from the set of singularities of $T$, denoted by $\Gamma$. The rational approximation has the form
\begin{equation}\label{eq:nleigsq}
F_d(\omega):=\sum_{j=0}^{d}b_j(\omega)D_j
\end{equation}
and is constructed so that it interpolates $T$ at nodes $\sigma_j\in\partial\Sigma$ (the boundary of $\Sigma$), using the rational basis functions with poles at $\gamma_j\in\Gamma$ defined by the recursion
\begin{equation}\label{eq:nleigsbj}
b_0(\omega)=1,\quad b_{j}(\omega)=\frac{\omega-\sigma_{j-1}}{\beta_{j}(1-\omega/\gamma_{j})}b_{j-1}(\omega),\quad j=1,2,\dots.
\end{equation}
The $\beta_j$'s are normalization factors chosen so that $\max_{\omega\in\partial\Sigma}|b_j(\omega)|=1$. The interpolation nodes and poles that determine the approximation $F_d(\omega)$ are obtained as a sequence of Leja--Bagby points for $(\Sigma,\Gamma)$~\cite{guttel14}. The interpolation conditions $R_j(\sigma_j)=T(\sigma_j)$ determine that the coefficient matrices $D_j$ of~\eqref{eq:nleigsq} (called rational divided differences) can be computed via the recurrence
\begin{equation}\label{eq:nleigsdd}
D_0=\beta_0T(\sigma_0),\quad
D_{j}=\frac{T(\sigma_j)-R_{j-1}(\sigma_j)}{b_j(\sigma_j)},\quad j=1,2,\dots.
\end{equation}
Since $T$ is expressed in the form~\eqref{eq:split}, they can be written as
\begin{equation}\label{eq:nleigsdsplit}
D_j=\sum_{i=0}^{d}d_{i}^jA_i,\quad j\ge 0,
\end{equation}
where $d_{i}^j$ denotes the $j$th rational divided difference corresponding to the scalar function $f_i$, which can be computed with a cheap and numerically stable procedure detailed in~\cite{guttel14}. Here we use this latter form~\eqref{eq:nleigsdsplit}, but in our implementation the $D_j$ matrices are not computed explicitly. Instead the solver works with them implicitly, operating with the $A_i$ matrices that appear in the definition of $T$.

The solver is implemented for the case of a general nonlinear function $T$. However, if the problem is rational, which is the case we are concerned in this paper, the degree $d$ is equal to $\max\{p,q\}$ if $T$ is a rational matrix-valued function of type $(p,q)$, and also the singularity set $\Gamma$ is equal to the set of poles of $T$. In that case, the interpolant $F_d$ of~\eqref{eq:nleigsq} is exact for any choice of the sampling points $\sigma_k$. In other words, $F_d$ is just a rewrite of $T$.

Once $F_d$ has been obtained, the problem $F_d(\omega)\xi=0$ is solved via linearization, that is, a linear eigenvalue problem is constructed
\begin{equation}\label{eq:nleigslin}
\Alin y=\omega \Blin y,
\end{equation}
whose eigenvalues $\omega$ are the same and whose eigenvectors have the form
\begin{equation}
y=\begin{bmatrix}
b_0(\omega)\xi\\\vdots\\b_{d-1}(\omega)\xi
\end{bmatrix}.
\end{equation}
The matrices $\Alin$ and $\Blin$ of the linearization~\eqref{eq:nleigslin} have a particular block structure,
\begin{equation}\label{eq:nleigsmatlin}
  \Alin=\begin{bmatrix}
    D_0       & D_1      & \dots  & D_{d-2}       & (D_{d-1}-\frac{\sigma_{d-1}}{\beta_d}D_d) \\
    \sigma_0I & \beta_1I &        &               & \\
              & \ddots   & \ddots &               & \\
              &          & \ddots & \beta_{d-2}I  & \\
              &          &        & \sigma_{d-2}I & \beta_{d-1}I
  \end{bmatrix},\quad
  \Blin=\begin{bmatrix}
    0 & 0                      & \dots  & 0                              &-\frac{D_d}{\beta_d} \\
    I & \frac{\beta_1}{\gamma_1}I &        &                                & \\
      & \ddots                 & \ddots &                                & \\
      &                        & \ddots & \frac{\beta_{d-2}}{\gamma_{d-2}}I & \\
      &                        &        & I                              & \frac{\beta_{d-1}}{\gamma_{d-1}}I
  \end{bmatrix}.
\end{equation}
We use the static NLEIGS variant~\cite{guttel14}, where the linearization matrices are created a priori and then the linear eigenproblem~\eqref{eq:nleigslin} is solved (as opposed to the dynamic variant where the approximation and linearization are built incrementally as the Krylov subspace grows). To solve the linear eigenproblem, we implement a customized version of the shift-and-invert Krylov--Schur method~\cite{stewart01}, as described next.

The dimension of the linear eigenproblem~\eqref{eq:nleigslin} is equal to $d\cdot N$, where $d$ is the number of terms in the rational approximation and $N$ is the dimension of the original nonlinear problem. Since this dimension may be quite large, it is important to exploit the block structure of the linearization matrices in order to solve the linear problem efficiently (in terms of memory and computational effort). The block structure is considered when operating with the matrices, and also in the management of the subspace basis, as explained next.

The linearization matrices~\eqref{eq:nleigsmatlin} are never built explicitly, and instead the Krylov--Schur method proceeds by operating with their nonzero blocks only.
To generate a new Krylov vector, we need to multiply the last vector of the basis with matrix
\begin{equation}\label{eq:nleigssinvert}
\Slin=(\Alin-\mu \Blin)^{-1}\Blin,
\end{equation}
where $\Alin$ and $\Blin$ are given in~\eqref{eq:nleigsmatlin}, and $\mu\in\Sigma$ is the shift (a value around which the eigenvalues are sought). We derive a set of recurrences that implicitly apply matrix $\Slin$ to a vector, by considering a block LU factorization of $(\Alin-\mu \Blin)$. The operations in these recurrences are expressed in terms of the problem matrices $A_i$ instead of the divided differences $D_i$. 
The required computations involve vector \emph{axpy} operations, 
matrix-vector products with the matrices, and the construction of $F_d(\mu)$ and its inverse. Rather than building the inverse explicitly, a sparse linear solver is used (usually via a factorization).

The block structure of the matrices $\Alin$ and $\Blin$ can also be exploited to allow a compact representation of the Krylov basis. The idea is to derive linear dependency relations among the $d$ blocks of the generated Krylov vectors, whose global dimension is $d\cdot N$. In this way, it is possible to define a basis $U_{k+d}$ of vectors of length $N$, from which all blocks of the Krylov basis $V_k$ can be reconstructed, resulting in the relation
\begin{equation}\label{eq:toarrep}
V_k=(I_d\otimes U_{k+d})G_k,
\end{equation}
for some matrix of coefficients $G_k$. That is, if the Krylov basis $V_k$ is divided in $d$ blocks of $N$ consecutive rows, $\{V_k^i\}_{i=0}^{d-1}$, then they can be expressed as
\begin{equation}\label{eq:toarblocksd}
V_{k}^i=U_{k+d}G_{k}^i,\quad i=0,\dots,d-1,
\end{equation}
where $\{G_k^i\}_{i=0}^{d-1}$ are the blocks of $G_k$. Both $U_{k+d}$ and $G_k$ must have orthonormal columns by construction. This is called the TOAR representation and has been described in~\cite{campos16} in the context of polynomial eigenvalue problems. We have adapted this technique to the NLEIGS linearization~\eqref{eq:nleigsmatlin}, so that a number of vector recurrences are employed to compute the columns of $U_{k+d}$ together with the entries of $G_k$ in the context of the Arnoldi iteration. The compact representation not only reduces the storage requirements for the basis in roughly a factor $d$, but it also reduces the computational cost associated with the orthogonalization of the basis vectors. Our implementation also incorporates additional optimizations such as restart and eigenvalue locking. 
Further details about the NEP module can be found in \cite{campos18}.
In summary, the main operations involved in the NLEIGS solver are orthogonalization and other operations with length-$N$ vectors, sparse matrix-vector products with the problem matrices $A$, matrix \emph{axpy} operations to form $F_d(\mu)$ explicitly, and sparse factorization of $F_d(\mu)$ for linear solves (also of size $N$). The last one is the most expensive operation, but it is much cheaper than a size $d\cdot N$ factorization in matrix $\Slin$~\eqref{eq:nleigssinvert} that would be required in a naive implementation of NLEIGS. Moreover, all the computation can be done in parallel (using MPI), enabling the solution of large-scale problems.

\section{Applications to metal-dielectric nanostructures}\label{sec:benchmarks}
In this section we study four interesting metal-dielectric configurations, from where numerical approximations to resonances and resonant modes are computed. 
These configurations are used in Section \ref{sec:results} for comparing the error convergence in standard $h$- and $p$-FE, against the novel strategies presented in Section \ref{sec:strategies}.
First, geometries with simple symmetries are introduced. 
This allow us to determine exact pairs explicitly for both TM and TE polarizations. Finally, we describe a more demanding test case.  

The first two configurations serve as Benchmarks for testing strategies \ref{sec:h_str} and \ref{sec:p_str} applied to problems with non-dispersive and piecewise constant material properties.
In the last two configurations we are motivated by realistic applications in nano-photonics, where a metal coating is introduced. For these, three different relative permittivity models are used:
$\eps_v:=1$ (\emph{Vacuum}), $\eps:=2$ (\emph{Silica}), and $\eps_{metal}$ (\emph{Gold}), modeled by a sum of Drude-Lorentz terms \eqref{eq:drude_lorentz}. For $\eps_{metal}$ we use the data given in Table \ref{gold_data} gathered in \cite{Rakic98}. This model of Gold has been extensively tested and has validity for $\om\in[0.5,6.5]\,eV$, where $eV$ denotes \emph{electron volt}.

\subsection{Scaling}\label{sec:scaling}
In finite precision arithmetic we prefer to work with dimensionless quantities, where we transform from dimensionless variables to physical variables (denoted with \textasciitilde). We use common physical constants in SI units: $\hbar$ is the scaled Planck's constant, $c$ is the speed of light in vacuum, and 
$e$ is the electron charge. 
In the numerical computations, we use the scaling factors $W=eV/\hbar$ in \emph{Hertz} and $L=2\pi c/W$ in \emph{meters}. Then, we define the dimensionless quantities
\begin{equation}
	x =\fr{\tilde x}{L},\quad {\omega}=\fr{\tilde \om}{W}\quad \hbox{satisfying}\quad LW=2\pi c.
	\label{scaling}
\end{equation}
The resulting length factor is $L=1239.842\,nm$, from where our spectral window becomes numerically equivalent to $eV$ scaling.

\begin{table}[]
\centering
\begin{tabular}{lll}
$\eps_\infty=1$  & $\om_p=$ 9.03           &            -          \\
$f_0=$ 0.76             & $\om_0=$ 0              & $\gamma_0=$ 0.053     \\
$f_1=$ 0.024            & $\om_1=$ 0.415          & $\gamma_1=$ 0.241     \\
$f_2=$ 0.01             & $\om_2=$ 0.83           & $\gamma_2=$ 0.345     \\
$f_3=$ 0.071            & $\om_3=$ 2.969          & $\gamma_3=$ 0.87      \\
$f_4=$ 0.601            & $\om_4=$ 4.304          & $\gamma_4=$ 2.494     \\
$f_5=$ 4.384            & $\om_5=$13.32           & $\gamma_5=$ 2.214   
\end{tabular}
\caption{Drude Lorentz data for Gold, taken from \cite{Rakic98}, with time convention $e^{-i\om t}$.}
\label{gold_data}
\end{table}

\subsection{Benchmarks in 1D}\label{sec:models_1D}

We focus on the problem described in Section \ref{sec:resonances1D} for \emph{even} refractive index profiles.
The computational domain is reduced to $I:=I_a^+:=(0,a)$, by imposing $u(0)=0$. This choice allows us to approximate the odd eigenfunctions of \eqref{eq:Helmholtz} and \eqref{eq:formalDtN}. 
For the derivation of reference solutions, we consider the following problem with $a=1$.


Let $\{x_j\}_{j=0}^{N}$ 
denote nodes with $x_0=0,\,x_N=1$ and introduce the partition consisting of $I_j:=(x_{j-1},x_{j})$, $j=1,2,\ldots,N$. Assume that the refractive index $n=\sqrt{\epsilon}$ is the constant $n_j$ over $I_j$ and let $u_j$ denote the restriction of $u$ to $I_j$. Furthermore, we assume that $n=1$, for $x>1$. Then $(\rho_j,\eta_j)=(1,n_j^2)$ for the TM-case and $(\rho_j,\eta_j)=(1/n_j^2,1)$ for the TE-case. The coupled problems for the TE/TM-case reads:
Find $(u_1,u_2,\ldots,u_N,\om)$ such that
\begin{equation}
	-\dtot{}{x}\left(\rho_j u_j'\right)-\om^2 \eta_j u_j=0,\,\,x\in I_{j},
	\label{eq:model_piecewise}
\end{equation}
where $u_{1}$ and $u_{N}$ satisfy the boundary conditions
\begin{equation}
	u_{1}(0)=0,\,\,\hbox{and}\,\,u'_{N}(1)=i\om\, u_{N}(1),
	\label{eq:boundary_piecewise}
\end{equation}
and the solutions of \eqref{eq:model_piecewise} are subject to the compatibility conditions
\begin{equation}
	u_{j}(x_{j})=u_{j+1}(x_{j}),\,\,\hbox{and}\,\,\rho_{j} u'_{j}(x_{j})=\rho_{j+1} u'_{j+1}(x_{j}),\,\,\hbox{for}\,\,j=1,\ldots,N-1.
	\label{eq:compatibility_piecewise}
\end{equation}
The general solutions to \eqref{eq:model_piecewise} can be written in the form
\begin{equation}
	u_j:=A_j e^{i n_j\om x} + B_j e^{-i n_j \om x}, \,\,x\in I_{j},
	\label{eq:solution_piecewise}
\end{equation}
with the $2N$ unknowns $A_j,\,B_j,\,j=1,2,\cdots,N$. The conditions \eqref{eq:boundary_piecewise} and \eqref{eq:compatibility_piecewise} imply that the unknowns are solutions of a matrix system 
\begin{equation}
	\scatt (\om)z=0,\quad z=(A_1,B_1,A_2,B_2,\cdots,A_N,B_N)^T,
	\label{eq:system_piecewise}
\end{equation}
where the entries corresponding to boundary conditions are placed in the last two rows.
If there exist nontrivial solutions to \eqref{eq:system_piecewise}, they satisfy $\det [\scatt(\om_m)]=0$ for some value $\om_m$ that corresponds to a resonance of the system.

\subsubsection{Slab problem}\label{sec:slab1d} 
In this section, we consider the problem \eqref{eq:model_piecewise}, \eqref{eq:boundary_piecewise}, and \eqref{eq:compatibility_piecewise} for the case $N=2$, with $n(x)=n_1$ for $x\in I_1:=(0,0.5)$, and $n(x)=1$ for $x\in I_2:=(0.5,1)$. The corresponding exact resonances for TM polarization are given by
\begin{equation}
	e^{2in_1 \om a}=-\mu,\,\, \om_m=\fr{(2m+1)\pi-i\hbox{Log}(\mu)}{2n_1 a}, \,\, \mu=\fr{n_1+1}{n_1-1},
	\label{eq:eig_slabTM}
\end{equation}
with the corresponding eigenfunctions as in \eqref{eq:solution_piecewise}, with
\begin{equation}
		2\frac{A_2}{A_1}=(n_1+1)e^{i \om a(n_1-1)}+(n_1-1)e^{-i \om a(n_1+1)}, B_1=-A_1,\,B_2=0.
	\label{eq:fun_slabTM}
\end{equation}
Similarly, the corresponding exact resonances for TE polarization are given by:
\begin{equation}
	e^{2in_1 \om a}=\mu,\,\, \om_m=\fr{2m\pi-i\hbox{Log}(\mu)}{2n_1 a}, \,\, \mu=\fr{n_1+1}{n_1-1},
	\label{eq:eig_slabTE}
\end{equation}
with the corresponding eigenfunctions as in \eqref{eq:solution_piecewise}, with
\begin{equation}
		2n_1\frac{A_2}{A_1}=(n_1+1)e^{i \om a(n_1-1)}-(n_1-1)e^{-i \om a(n_1+1)}, B_1=-A_1,\,B_2=0.
	\label{eq:fun_slabTE}
\end{equation}

\subsubsection{Multiple slab problem}\label{sec:multislab1d} 

Split the interval $I:=(0,1)$ in four uniform intervals $I_j$ of length $1/4$ and let $n:=(1,10,2,5)^T$ denote a vector with the
refractive indexes $n_j$. Using this refractive index profile results in an eigenvalue problem that is more demanding for FEM than the \emph{slab problem}. We compute very accurate Newton reference eigenvalues $\om_m$ from 
$\det [\scatt(\om_m)]=0$, with $\scatt (\om)$ given in \eqref{eq:system_piecewise}. 
For simplicity, we only study eigenvalue convergence of this problem for
\begin{equation}
	\begin{array}{lr}
		T\!M: & \om_{14}=10.105\,348\,365\,841-0.065\,215\,027\,533i, \\
		T\!E: & \om_{14}=10.156\,176\,418\,185-0.048\,229\,922\,564i.
	\end{array}
	\label{eq:reference_eigs_multislab}
\end{equation}


\subsection{Benchmarks in 2D}\label{sec:models_2D}
\begin{table}
\robustify\bfseries
\centering
\small
\begin{tabular}{rrS[table-format=2.9, detect-weight]
                  S[table-format=2.9, detect-weight]
                  S[table-format=2.9, detect-weight]
                  S[table-format=2.9, detect-weight] }
\toprule
{$m$} & {$j$} & {$\re \om_{T\!M}$} & {$\im \om_{T\!M}$} & {$\re \om_{T\!E}$} & {$\im \om_{T\!E}$} \\
\midrule
0 & 1  & 1.771128241  & -0.040209598  & 3.028519953  & -0.249632742  \\
1 & 2  & 2.507165546  & -0.308861246  & 1.276108857  & -0.022849842  \\
2 & 3  & 2.637054638  & -0.400052296  & 1.857593240  & -0.103922955  \\
3 & 4  & 3.312034818  & -0.666590209  & 2.444174749  & -0.314200015  \\
4 & 5	 & 3.406691805  & -0.693670033  & 2.506083838  & -0.291213845  \\
5 & 6	 & 3.525244074  & -0.743331707  & 2.324925787  & -0.200153901  \\
6 & 7  & 3.613595702  & -0.818203122  & 3.126303493  & -0.462545189  \\
7 & 8  & 3.671987538  & -0.878964710  & 2.510146419  & -0.300384680  \\
8 & 9  & 3.720376782  & -0.925532212  & 5.549482036  & -0.741472675  \\
9 & 10 & 3.762296208  & -0.963866600  & 3.201508932  & -0.529576832  \\
\bottomrule
\end{tabular}




\caption{\emph{Reference eigenvalues for the \emph{single coated disk} problem described in Section \ref{sec:SDC}.}}
\label{tab:CSD_reference}
\end{table}
The next two problems have radial symmetry centered at the origin, and the solutions expressed 
in polar coordinates $(r,\theta)$, will be written in terms of Bessel and Hankel functions of integer order $m$.
In this simple case outgoing solutions of \eqref{eq:master_eq} satisfy
\begin{equation}
	u=H_m^{(1)}(\om R)\left(
	\begin{array}{c}
		\cos m\theta \\
		\sin m\theta
	\end{array} \right),
	\,\,\,\hbox{for}\,\,\, x\in\partial B(0,R),\,\,\,\hbox{and}\,\,\, m\in \mx Z,
	\label{eq:outgoing}
\end{equation}
where \,supp$\,(n-1)\subset B(0,R)$. 
In subsections \ref{sec:SD} and \ref{sec:SDC}, we present solutions satisfying \eqref{eq:master_eq} and \eqref{eq:outgoing} for specific permittivity profiles.

\begin{figure}
	\begin{tikzpicture}[thick,scale=1.0, every node/.style={scale=0.9}]
		
		\draw(  0.00,14.95) node {\includegraphics[scale=0.53]{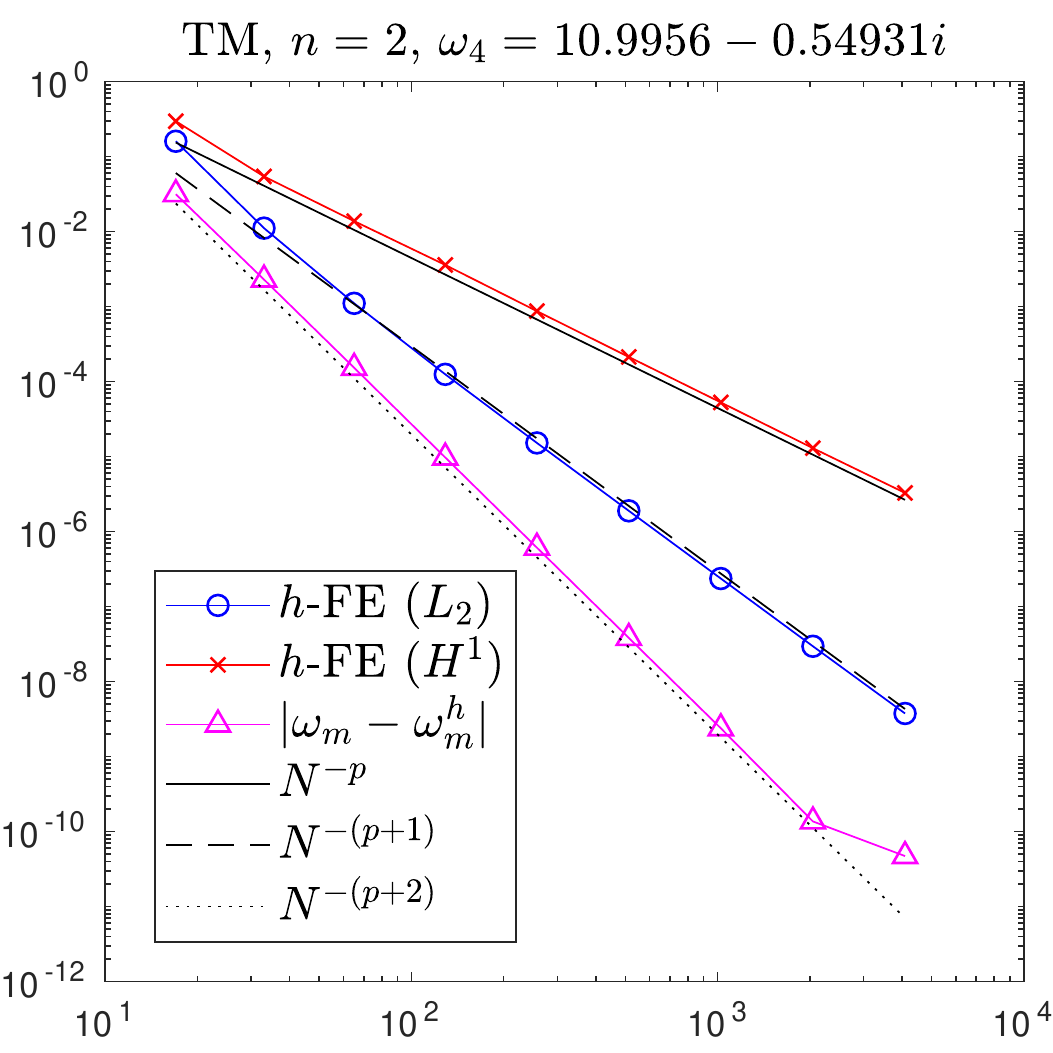}  };
		\draw(  5.50,14.95) node {\includegraphics[scale=0.53]{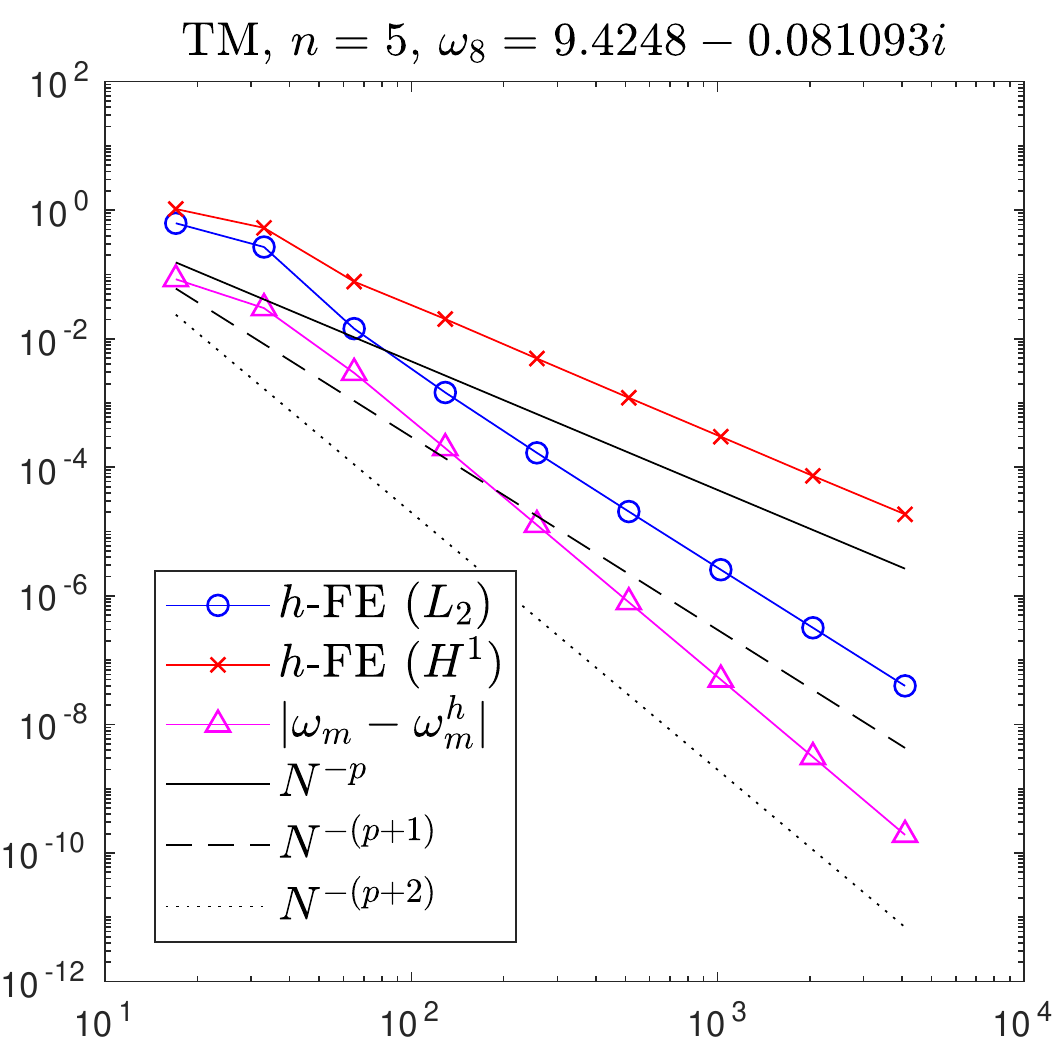}  };
		\draw( 11.00,14.95) node {\includegraphics[scale=0.53]{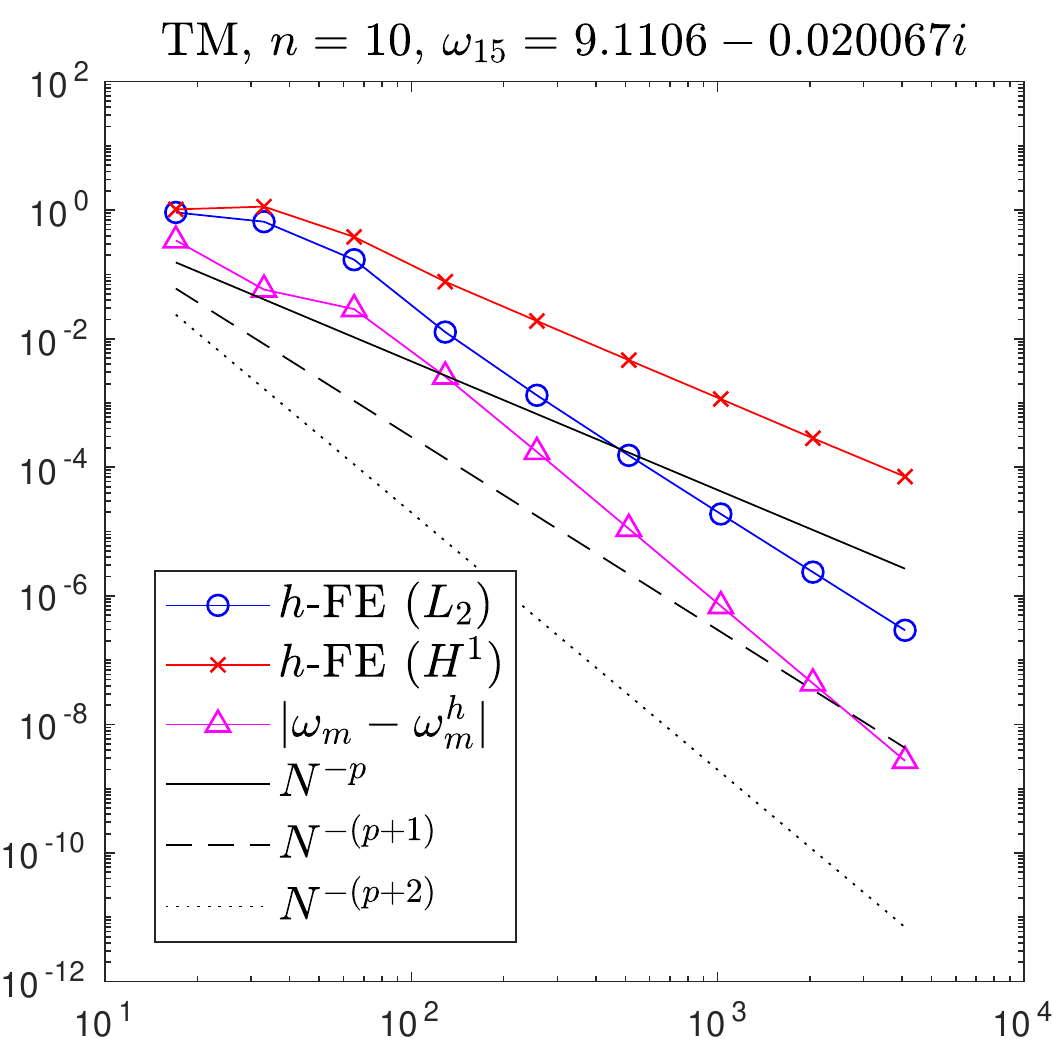} };
		
		\draw(  0.00,10.35) node {\includegraphics[scale=0.53]{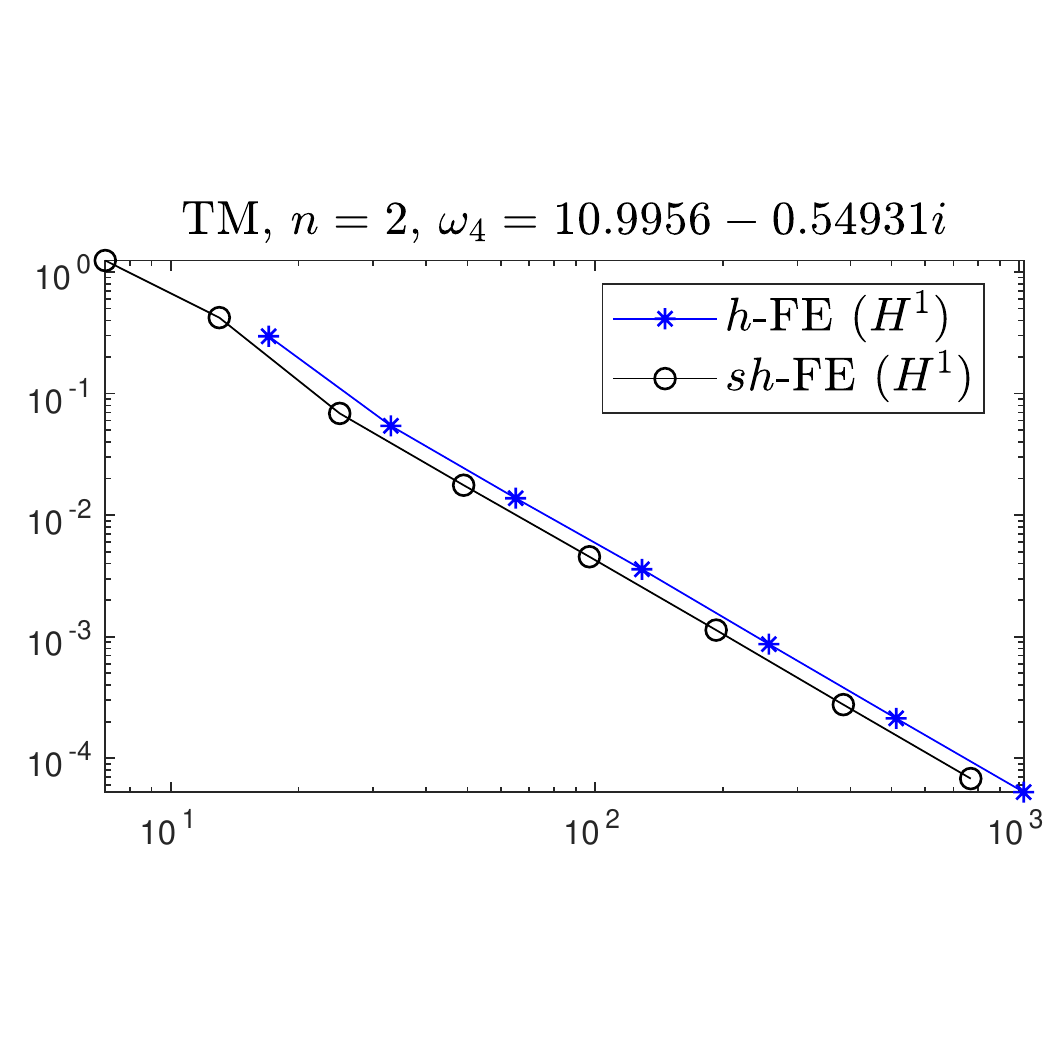}  };
		\draw(  5.50,10.35) node {\includegraphics[scale=0.53]{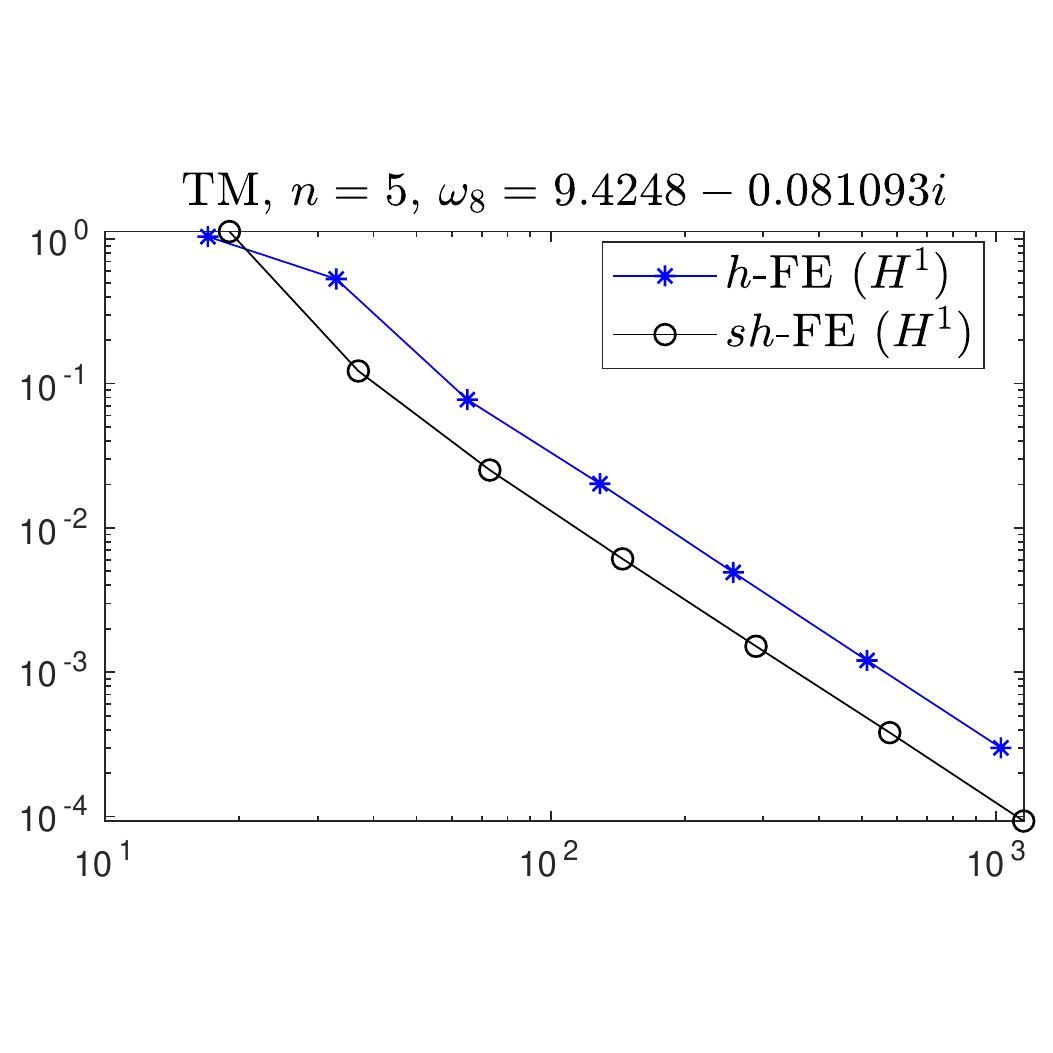}  };
		\draw( 11.00,10.35) node {\includegraphics[scale=0.53]{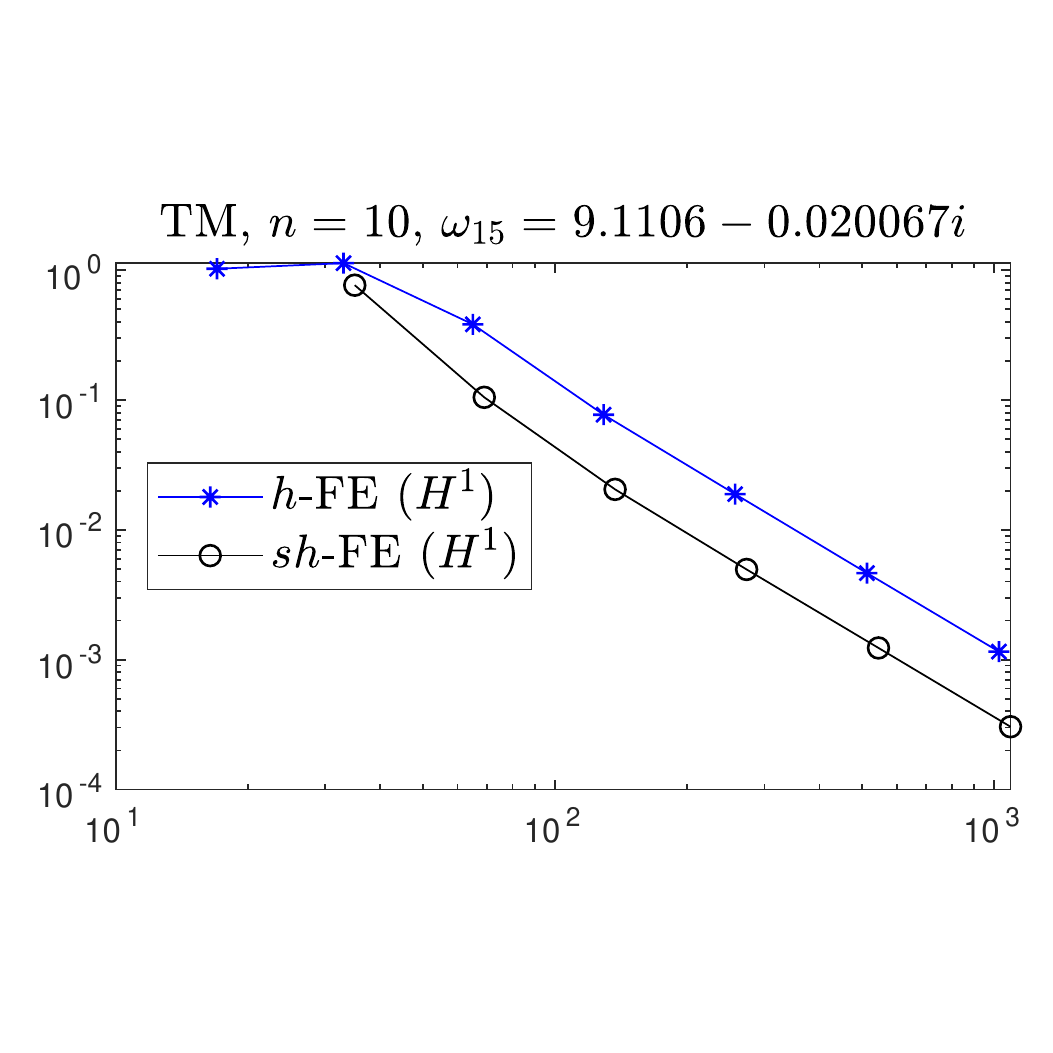}};
		
		\draw(  0.00, 6.90) node {\includegraphics[scale=0.53]{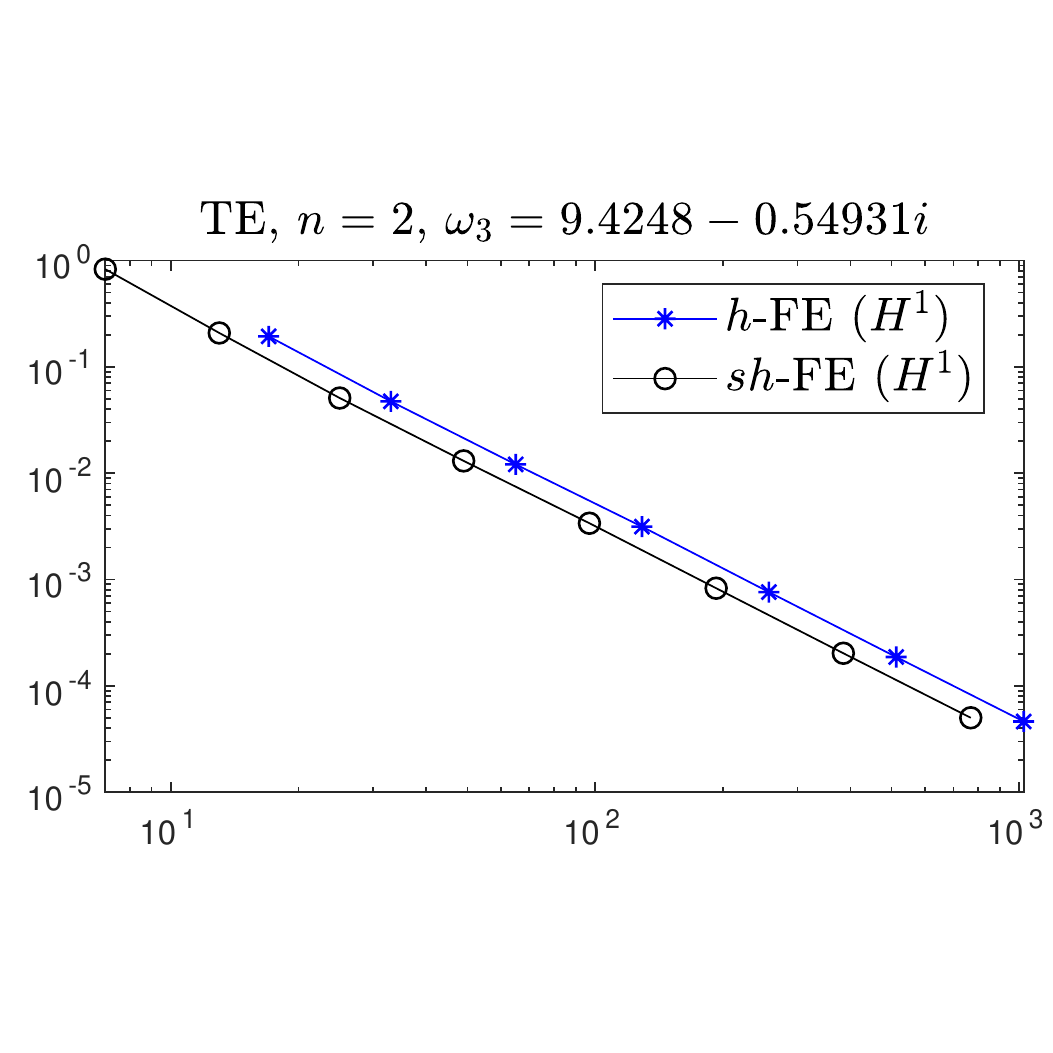}  };
		\draw(  5.50, 6.90) node {\includegraphics[scale=0.53]{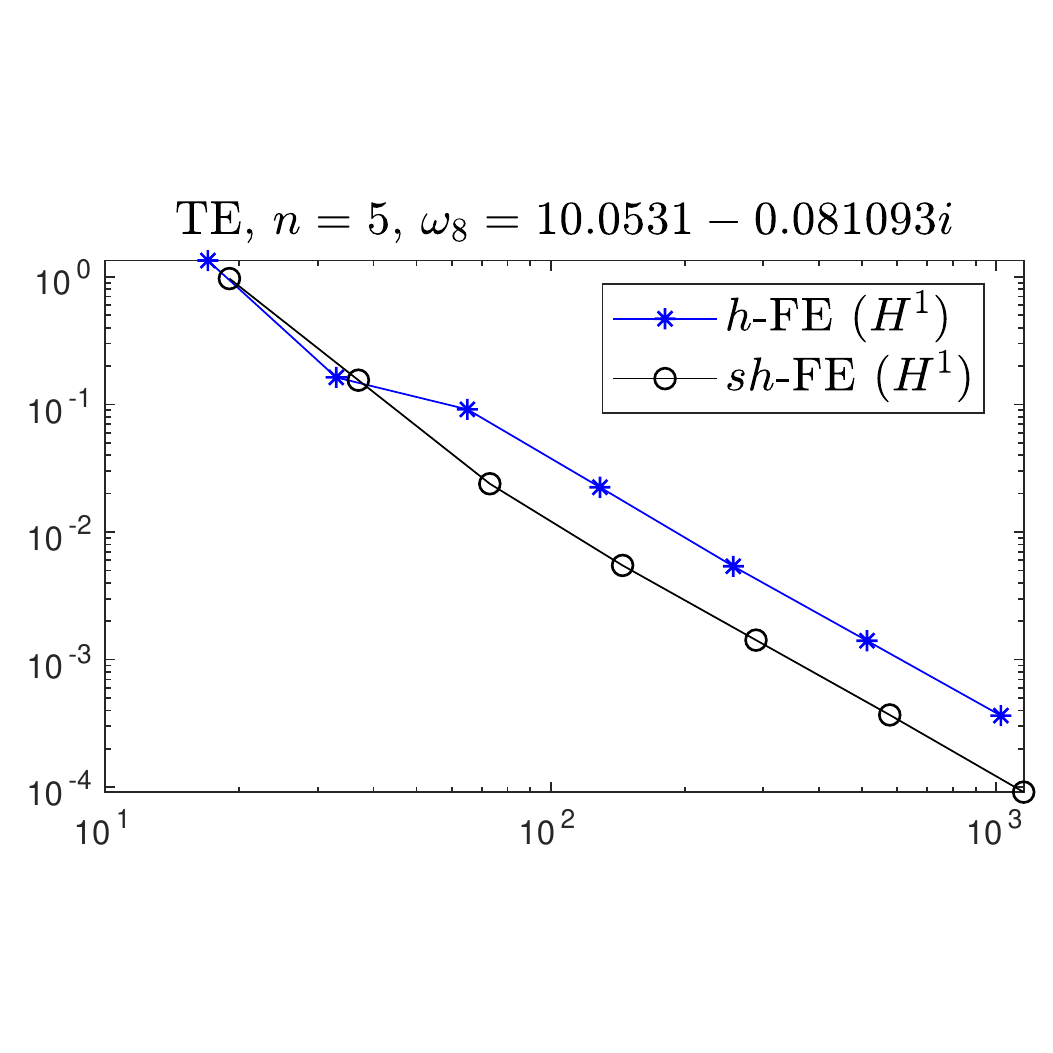}  };
		\draw( 11.00, 6.90) node {\includegraphics[scale=0.53]{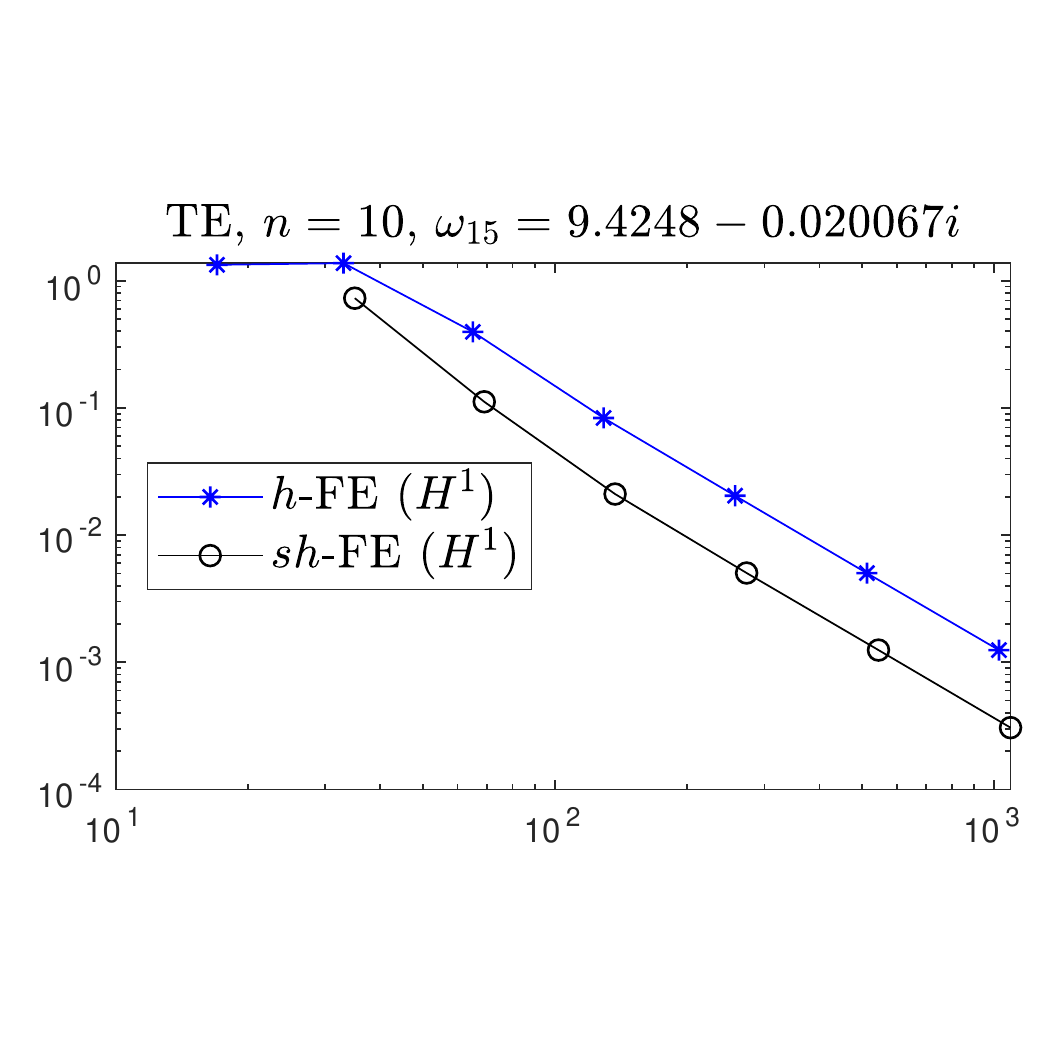}};
		
		\draw(  0.00, 3.45) node {\includegraphics[scale=0.53]{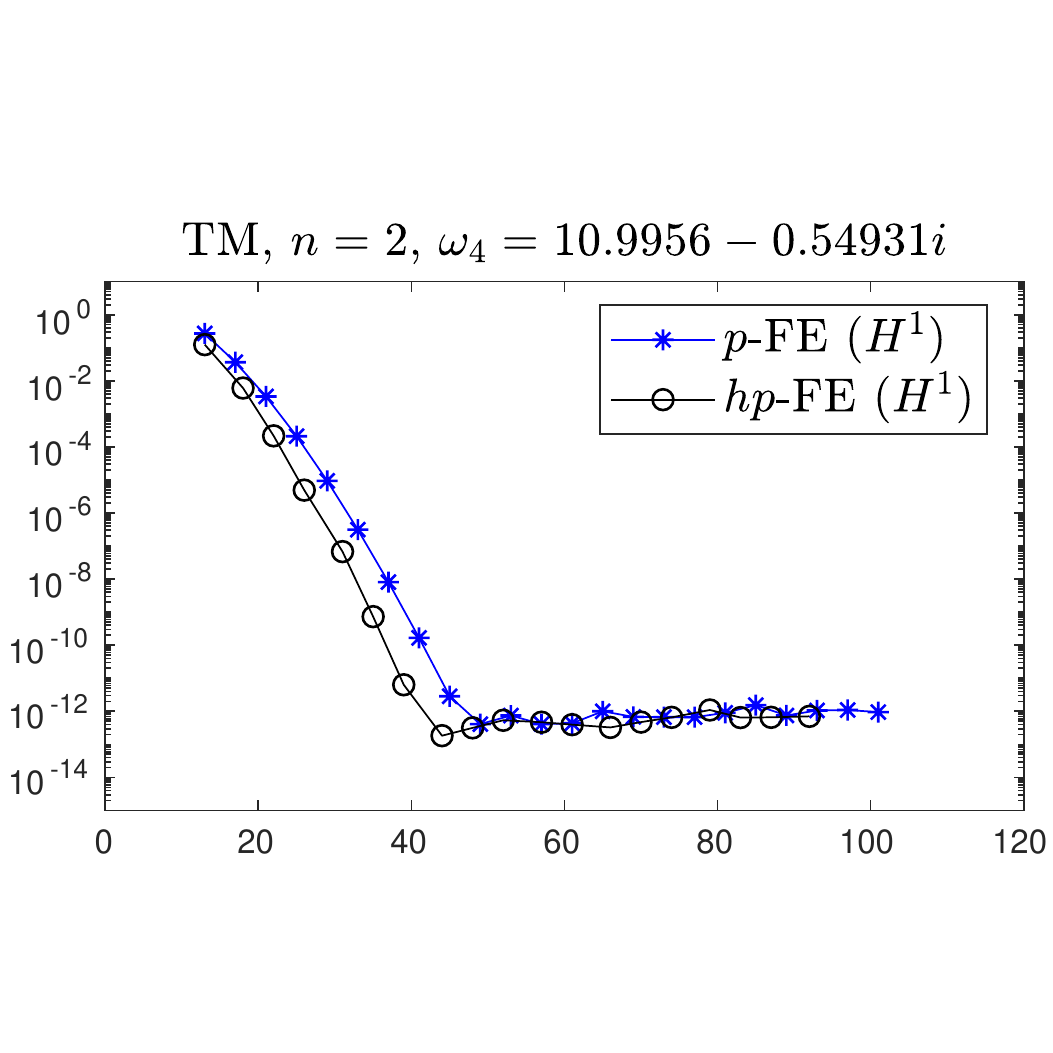}  };
		\draw(  5.50, 3.45) node {\includegraphics[scale=0.53]{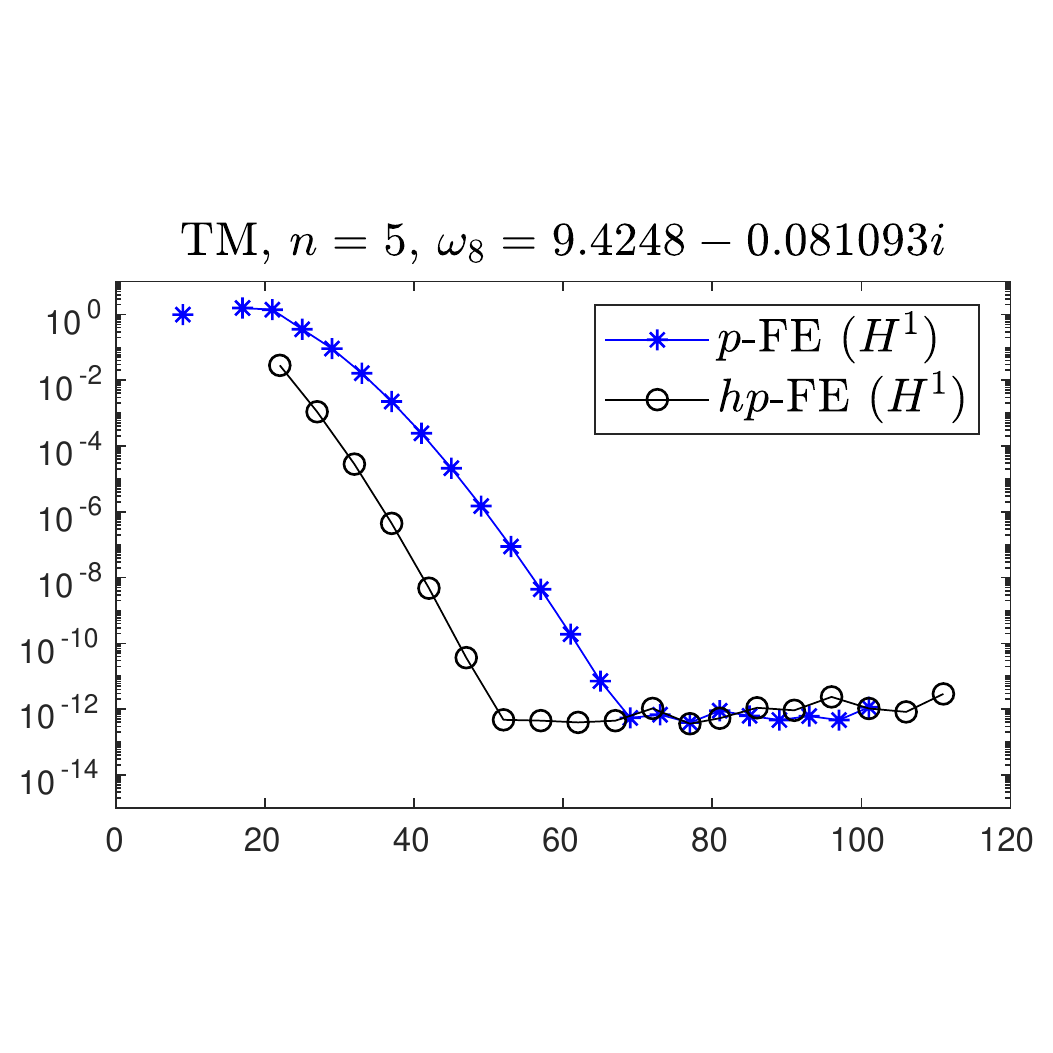}  };
		\draw( 11.00, 3.45) node {\includegraphics[scale=0.53]{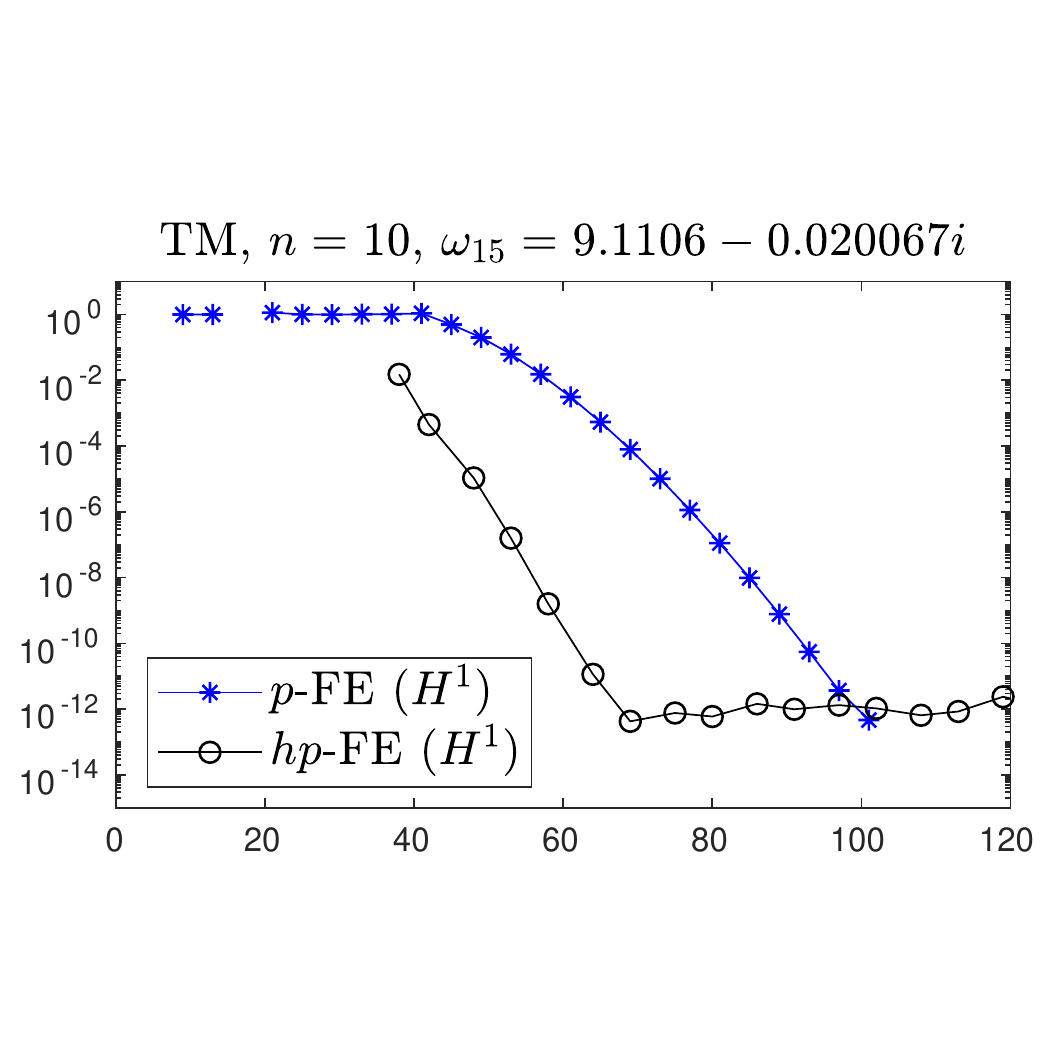}  };
		
		\draw(  0.00, 0.00) node {\includegraphics[scale=0.53]{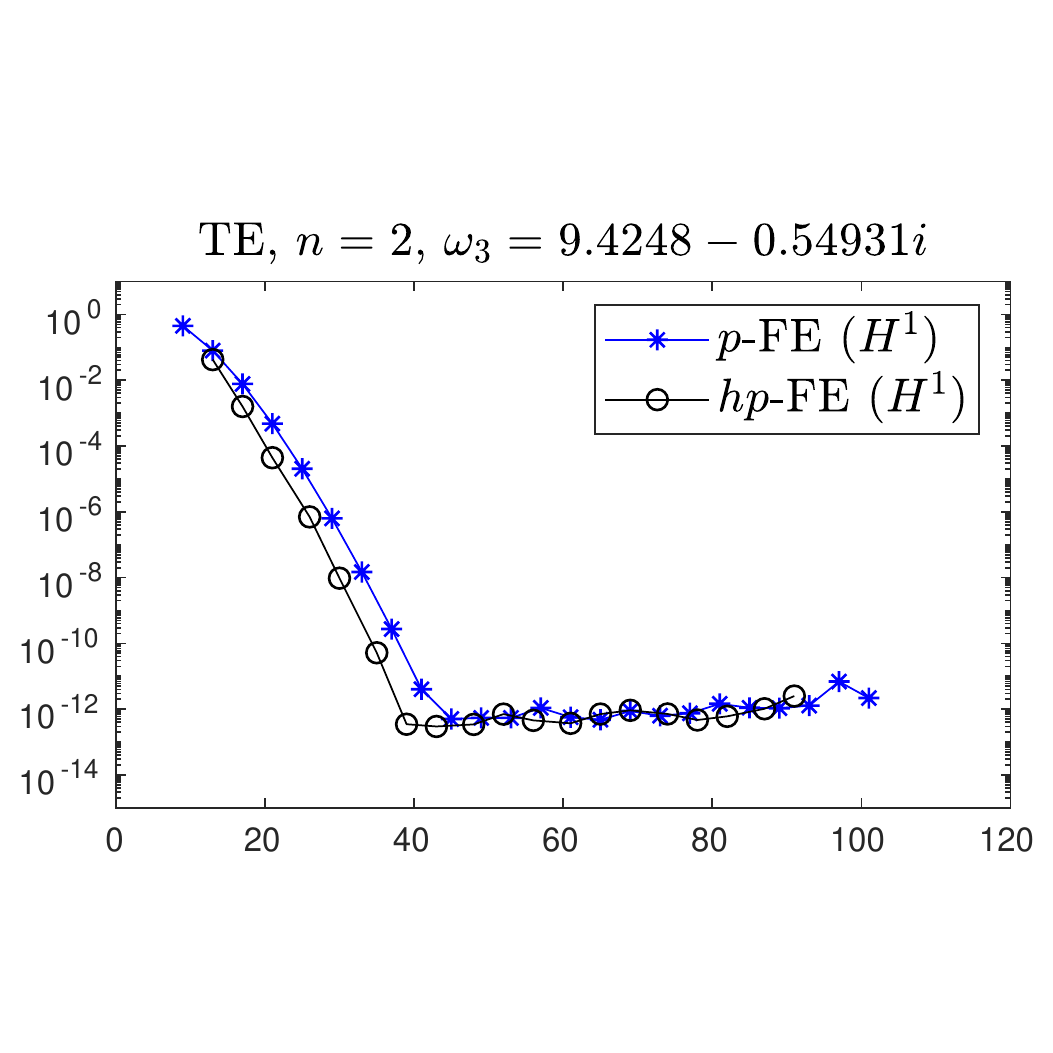}  };
		\draw(  5.50, 0.00) node {\includegraphics[scale=0.53]{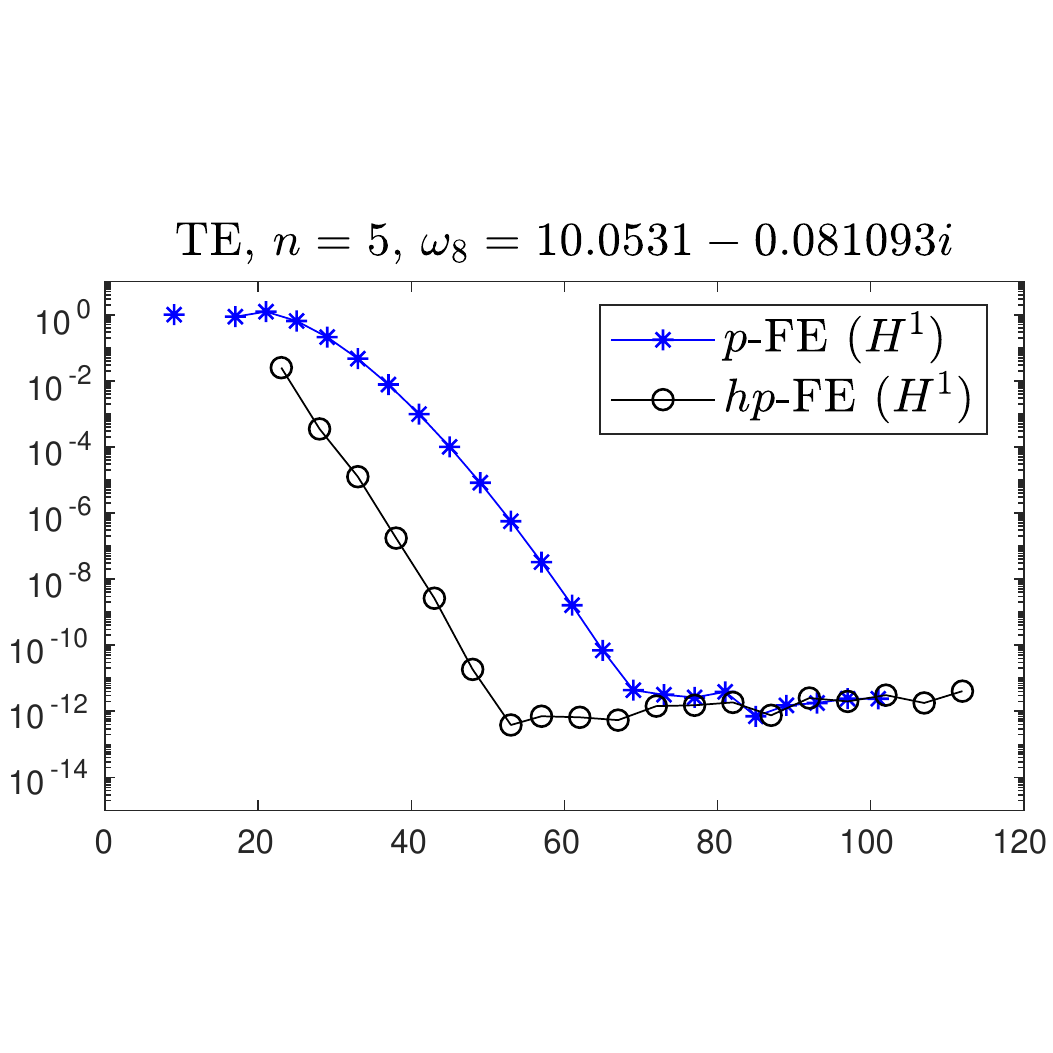}  };
		\draw( 11.00, 0.00) node {\includegraphics[scale=0.53]{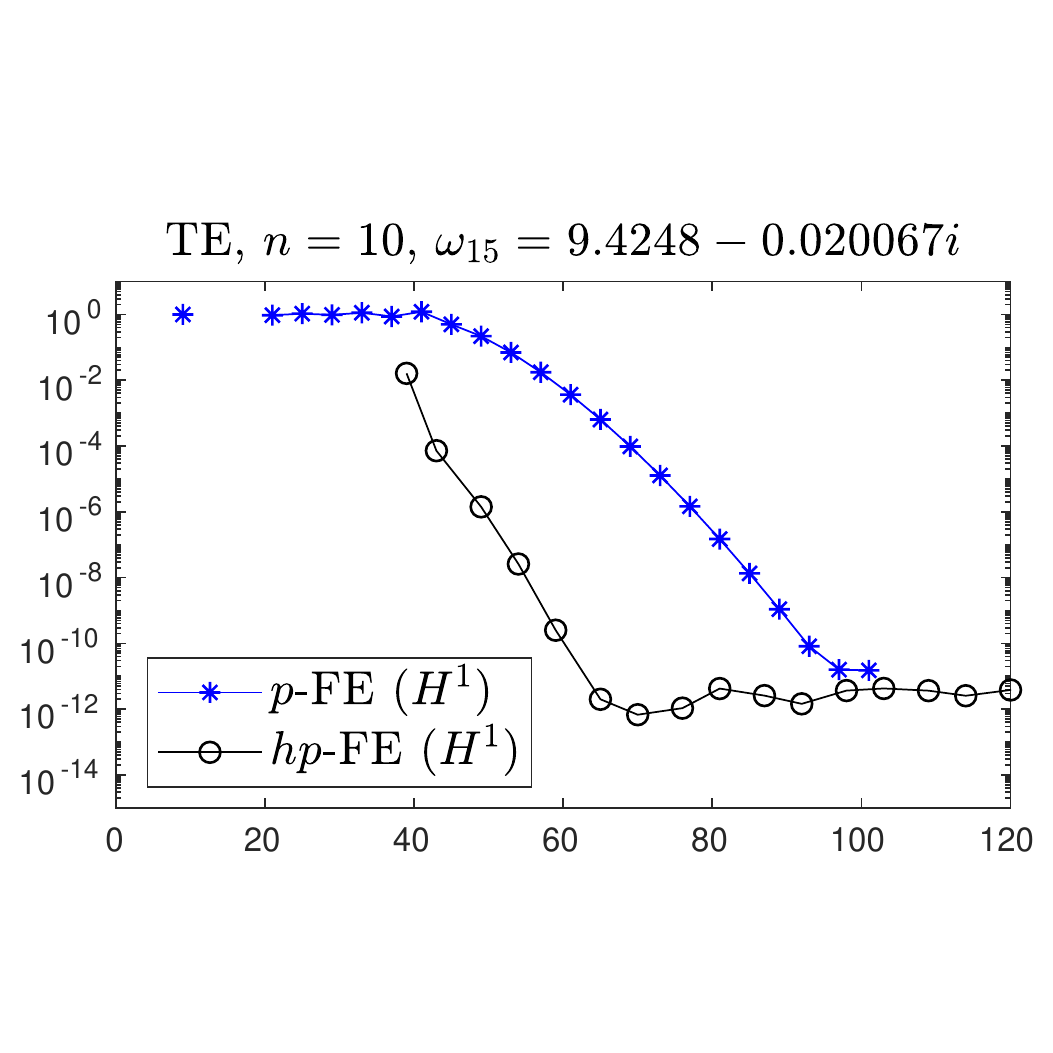}  };
	\end{tikzpicture}
	\vspace*{-10mm}
	\caption{\emph{Convergence plots (Relative errors vs. N) for the \emph{slab} problem \ref{sec:slab1d}
	in TM and TE polarizations. The upper horizontal stripe corresponds to classical $h$-FE error convergence for $n_1=2,5,10$ consecutively. Optimal convergence rates \eqref{eq:optimal_estimates} are indicated with solid, dashed, and dotted black lines. The following horizontal stripes correspond to classical $h$-FE and $p$-FE convergence marked with stars, and convergence with the a-priori strategies $sh$-FE (\ref{sec:h_str}), and $hp$-FE (\ref{sec:p_str}) are marked with circles.}}
	\label{fig:convergence_1D}
\end{figure}

\begin{figure}
	\begin{tikzpicture}[thick,scale=1.0, every node/.style={scale=0.9}]
		
		\draw(  0.00, 0.00) node {\includegraphics[scale=0.53]{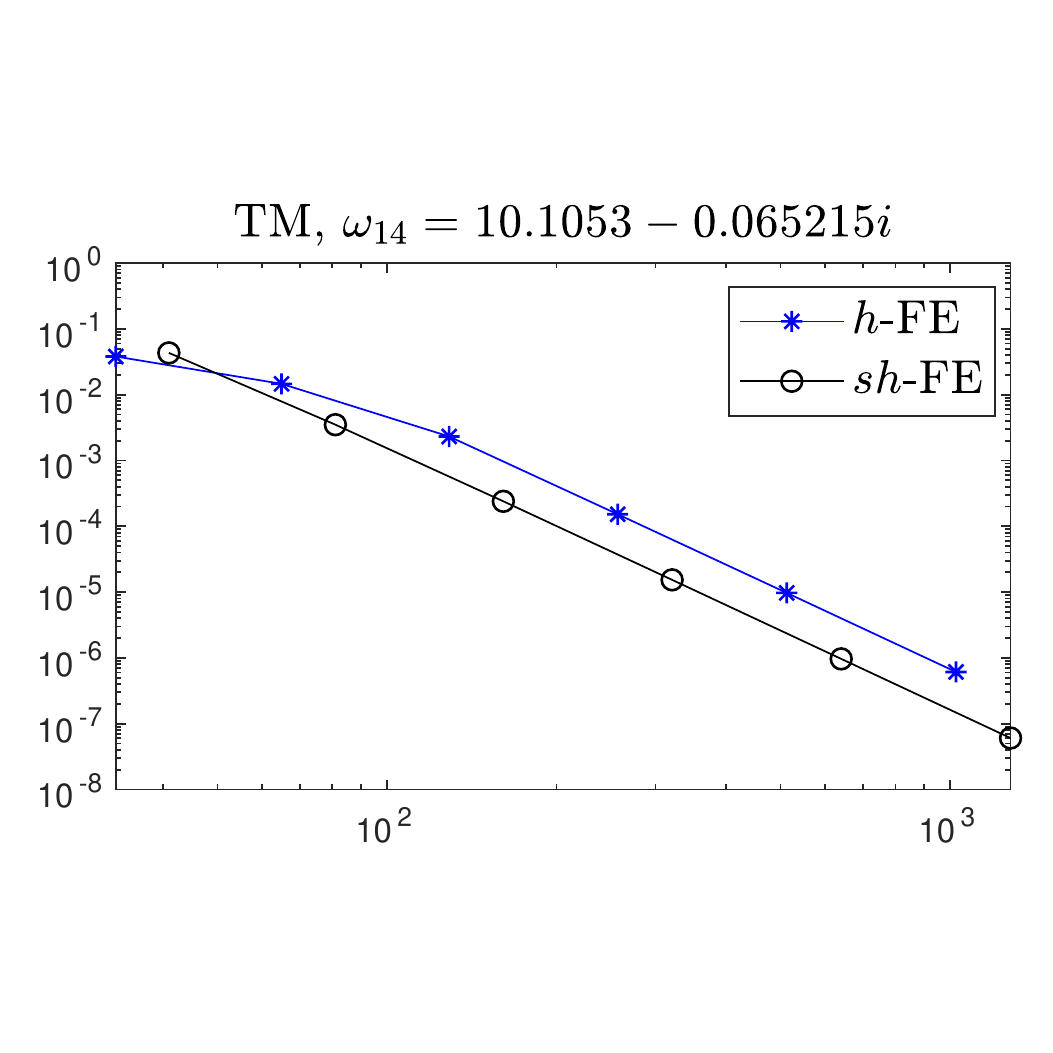}  };
		\draw(  5.50, 0.00) node {\includegraphics[scale=0.53]{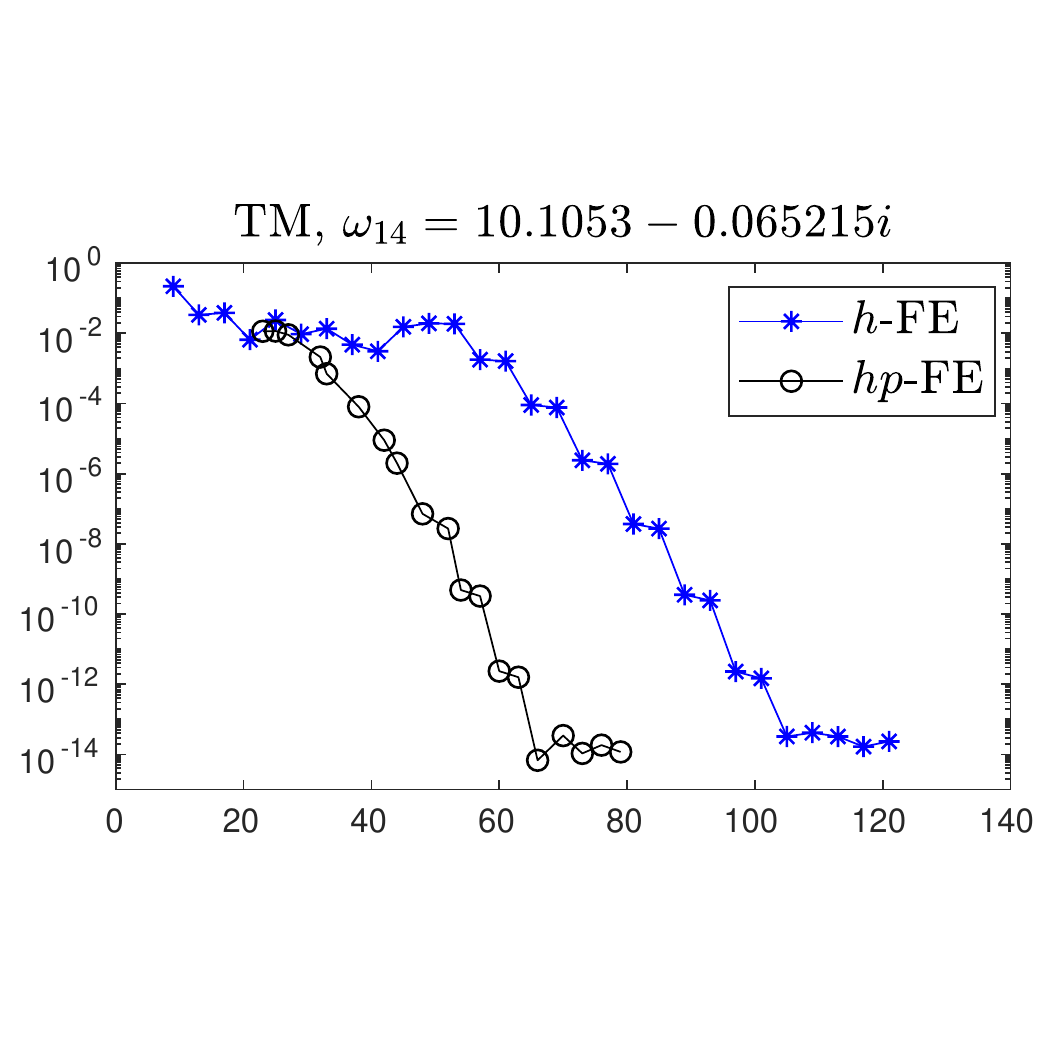}  };
		\draw( 11.00, 0.00) node {\includegraphics[scale=0.53]{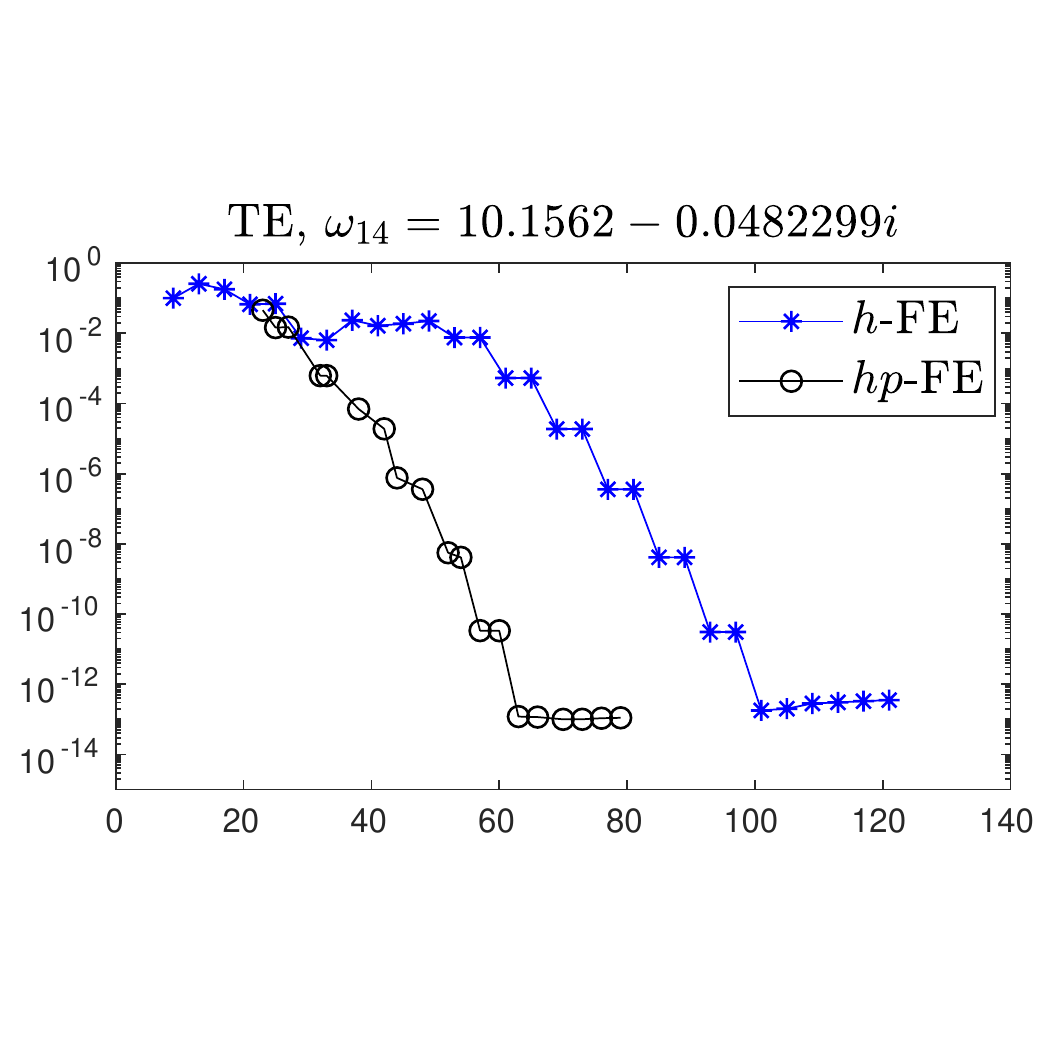}  };
	\end{tikzpicture}
	\vspace*{-10mm}
	\caption{\emph{Convergence plots (Relative errors vs. N) for the \emph{multiple slab} problem \ref{sec:multislab1d}
	in TM and TE polarizations. Classical $h$-FE and $p$-FE convergence are marked with stars, and convergence with the a-priori strategies $sh$-FE (\ref{sec:h_str}), and $hp$-FE (\ref{sec:p_str}) are marked with circles.} }
	\label{fig:convergence_multi1D}
\end{figure}

\subsubsection{Single disk problem}\label{sec:SD} 

Denote by $u=u_1,\,n=n_1$ the restrictions of $u,n$ to $\Om_1:=B(0,a)$, and set $n=n_2=1$ elsewhere.
The corresponding exact eigenfunctions to \eqref{eq:master_eq} and \eqref{eq:outgoing} read:
\begin{equation}
	u_1=
	N_m J_m(n_1 \om r)\left(
	\begin{array}{c}
		\cos m\theta \\
		\sin m\theta
	\end{array} \right),\,\, 
	u_2=
	H_m^{(1)}(\om r)\left(
	\begin{array}{c}
		\cos m\theta \\
		\sin m\theta
	\end{array} \right)
	,\,\, N_m:=\fr{H_m^{(1)}(a \om)}{J_m(a n_1 \om)}.
	\label{eq:exact}
\end{equation}
The eigenvalues $\om$ corresponding to $m=0$ are \emph{simple} and those corresponding to $m>0$ are \emph{degenerated} and have algebraic multiplicity $\alpha=2$.
The exact eigenvalue relationship for TM and TE can be written as
\begin{equation}
	J_m(a n_1 \om) H_m^{(1)\prime}(a \om)-g\,J'_m(a n_1 \om) H_m^{(1)}(a \om) =0,
	\label{eq:reson_SD}
\end{equation}
where $g=n_1,\,g=1/n_1$ corresponds to the TM polarization and TE polarization respectively.

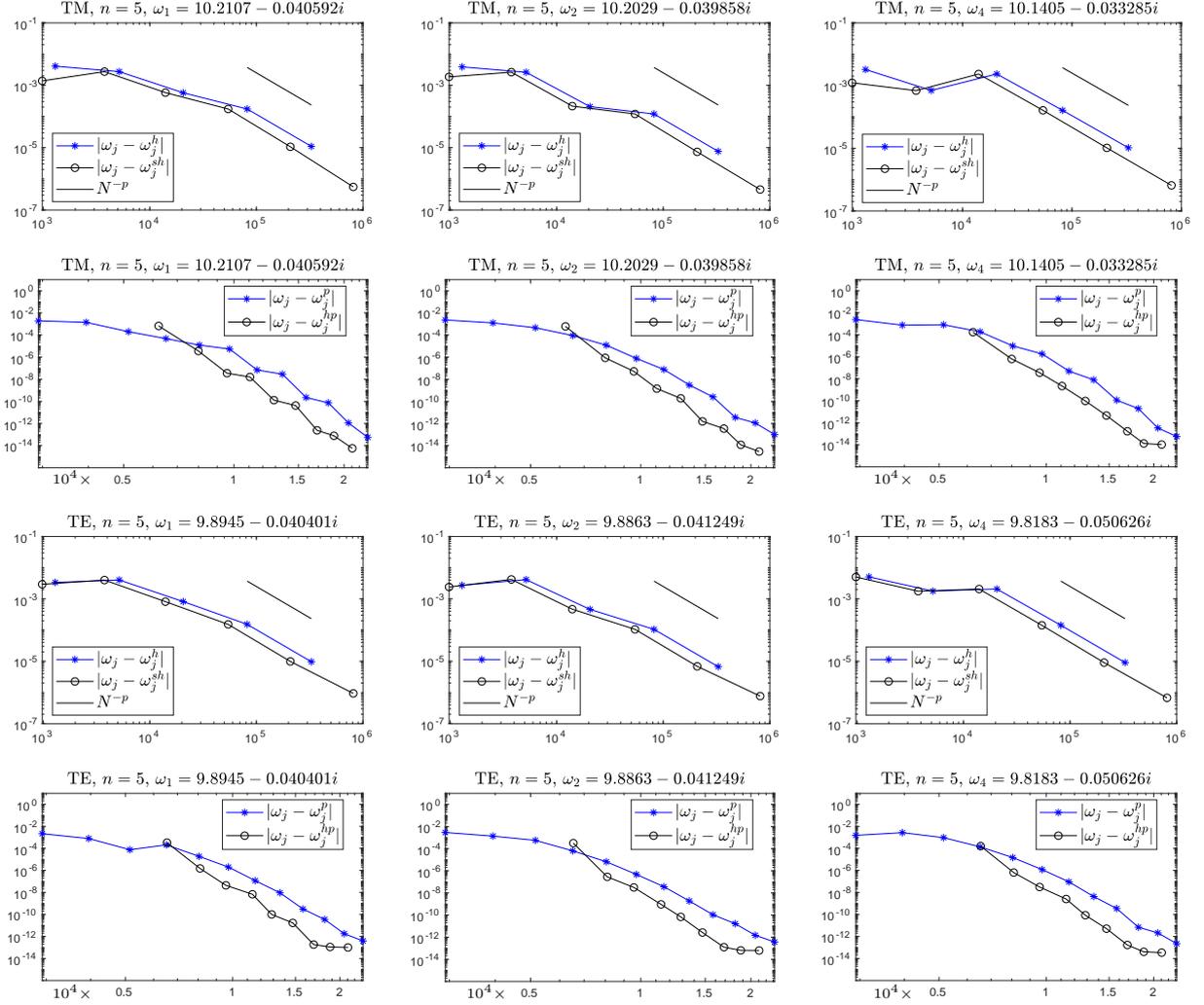
\begin{figure}
	\begin{tikzpicture}[thick,scale=1.0, every node/.style={scale=0.9}]
		
		\draw(  0.00, 10.5) node {\includegraphics[scale=0.53]{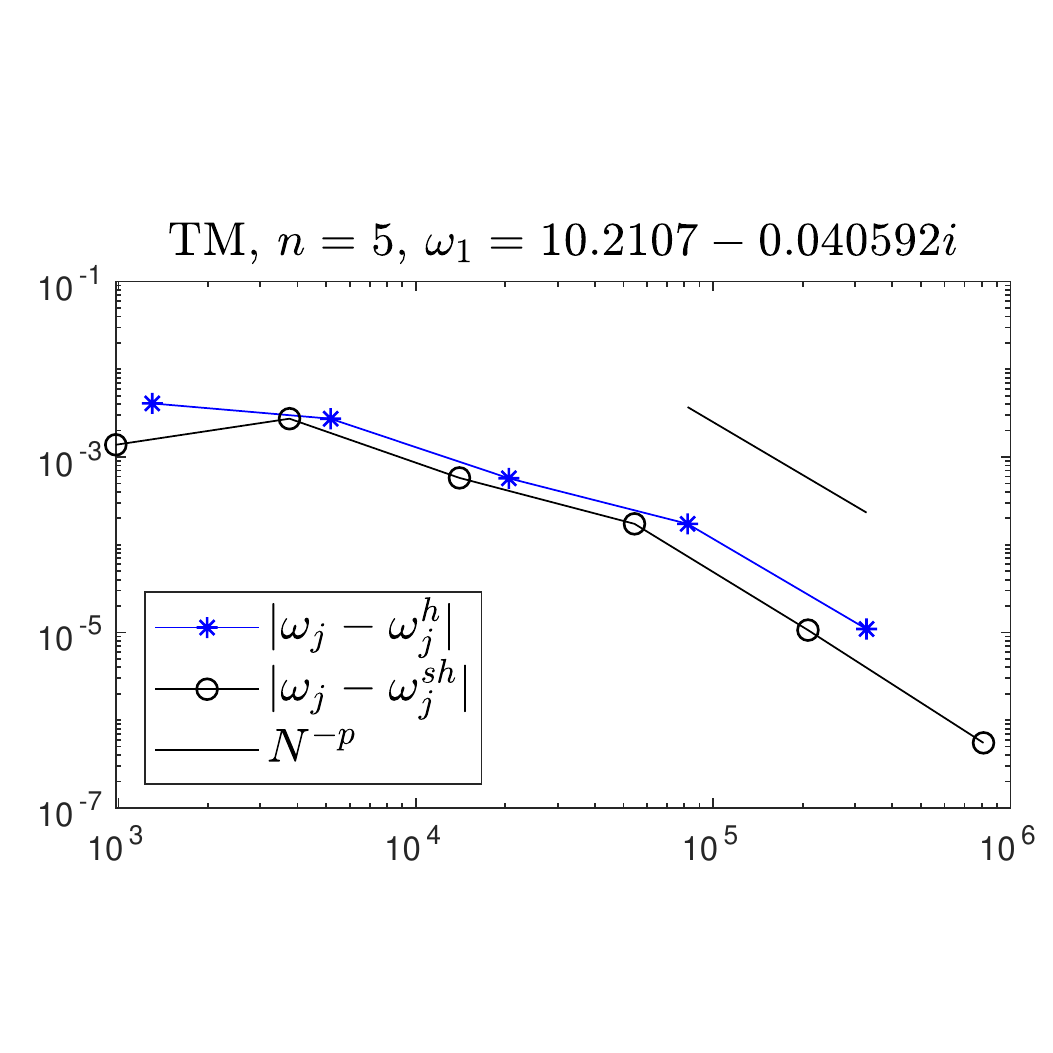}  };
		\draw(  5.50, 10.5) node {\includegraphics[scale=0.53]{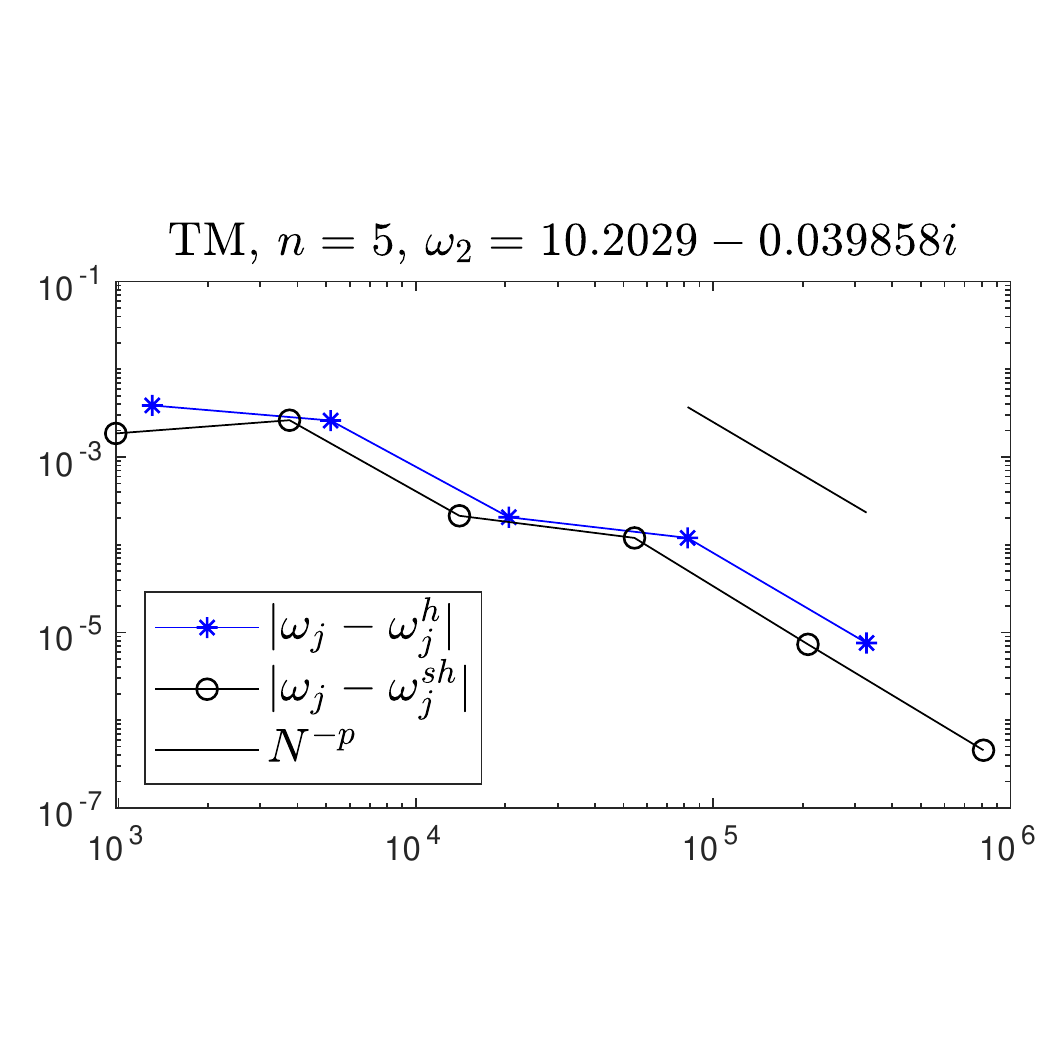}  };
		\draw( 11.00, 10.5) node {\includegraphics[scale=0.53]{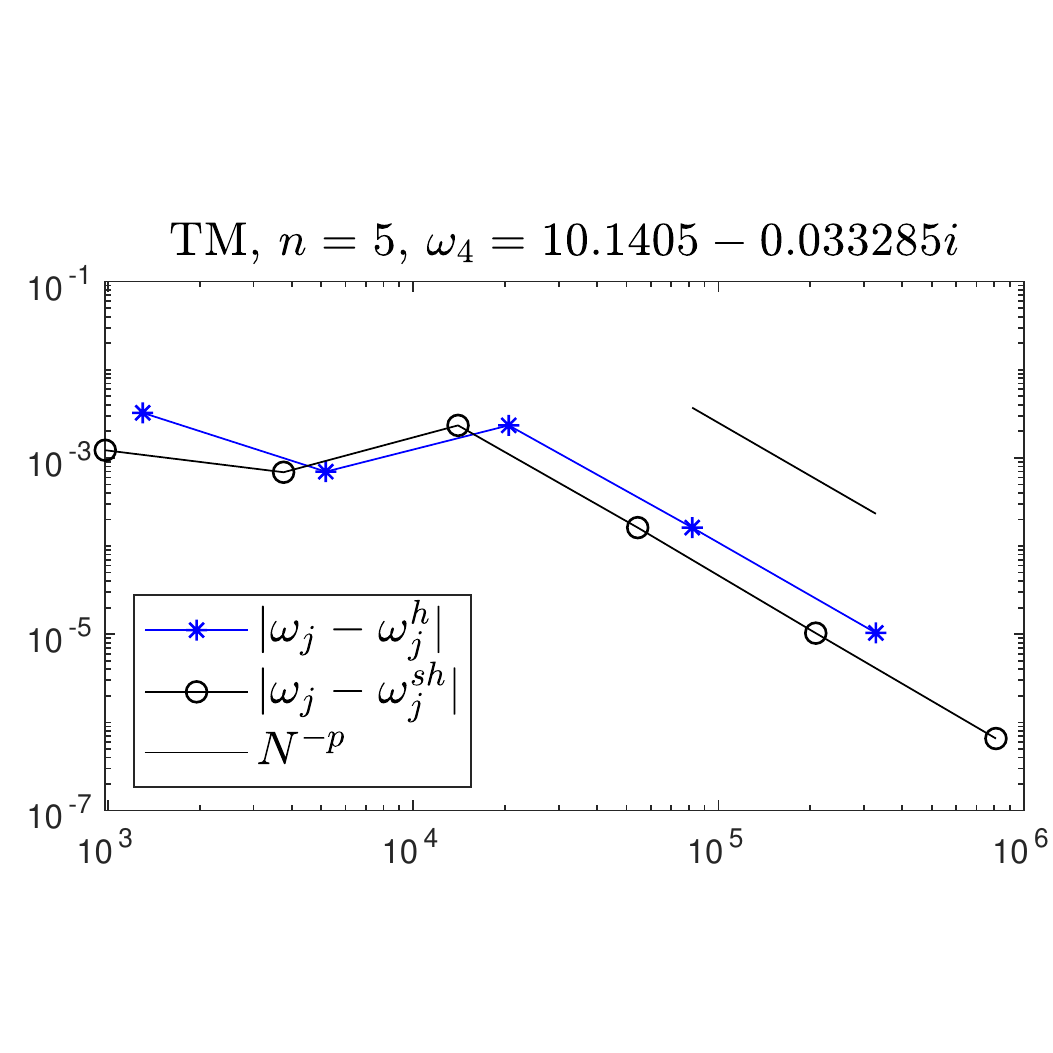}  };
		
		\draw(  0.00, 7.0) node {\includegraphics[scale=0.53]{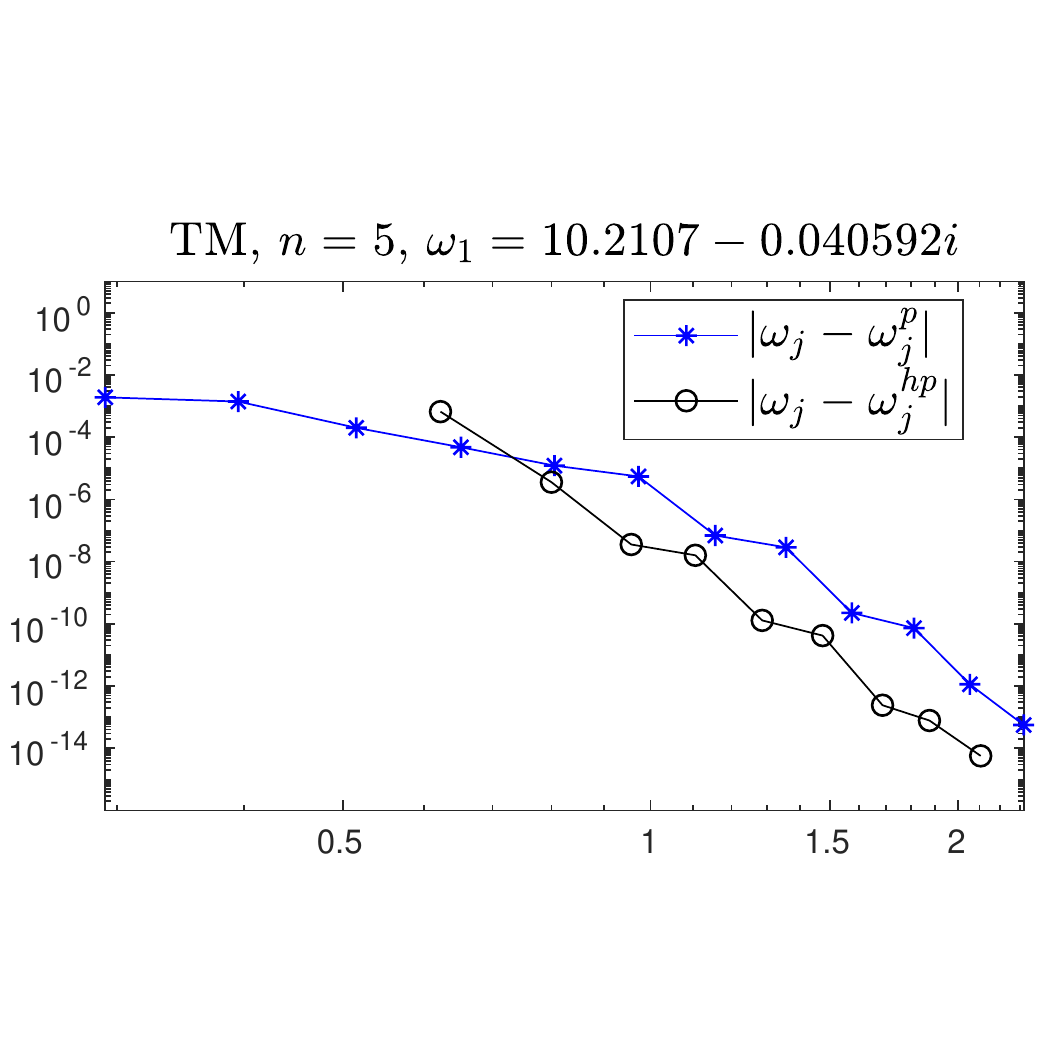}  };
		\draw(  5.50, 7.0) node {\includegraphics[scale=0.53]{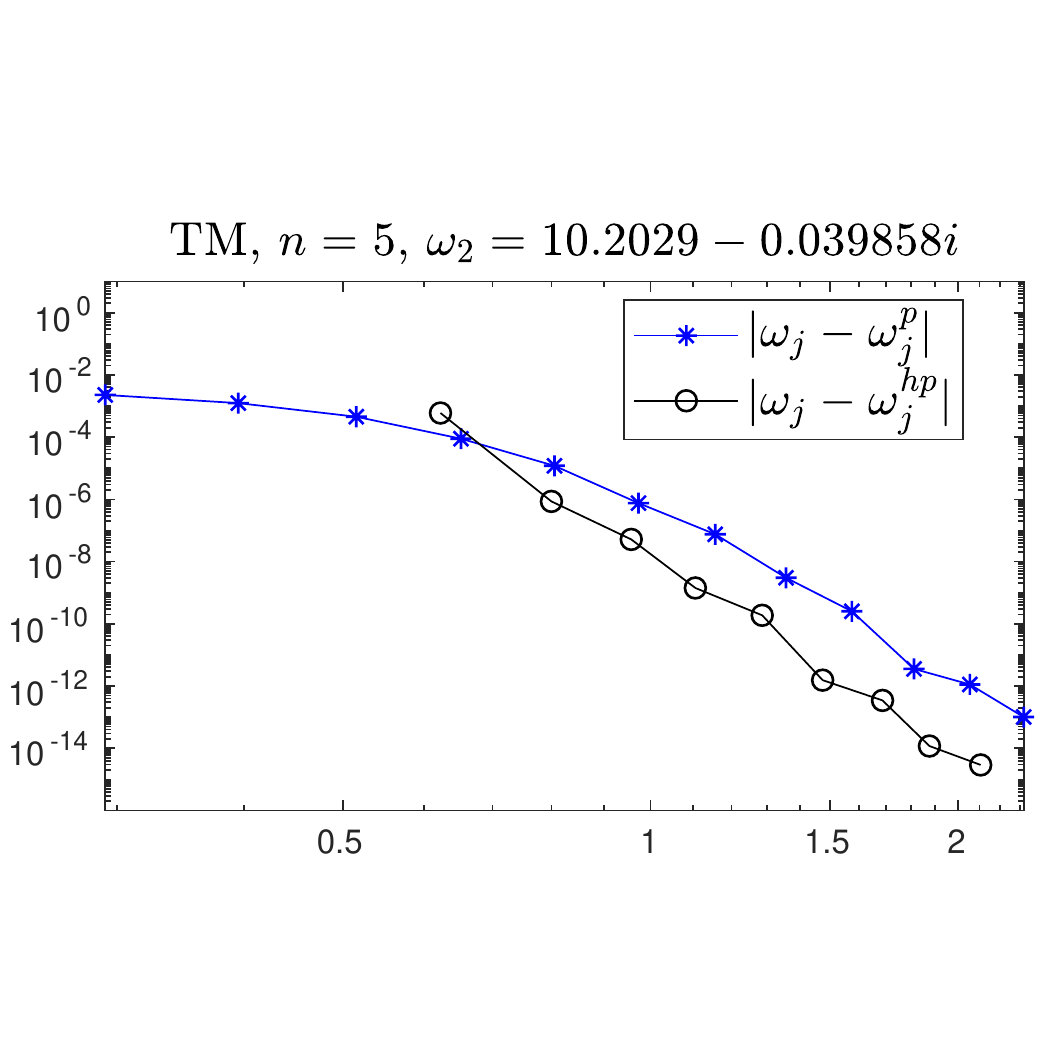}  };
		\draw( 11.00, 7.0) node {\includegraphics[scale=0.53]{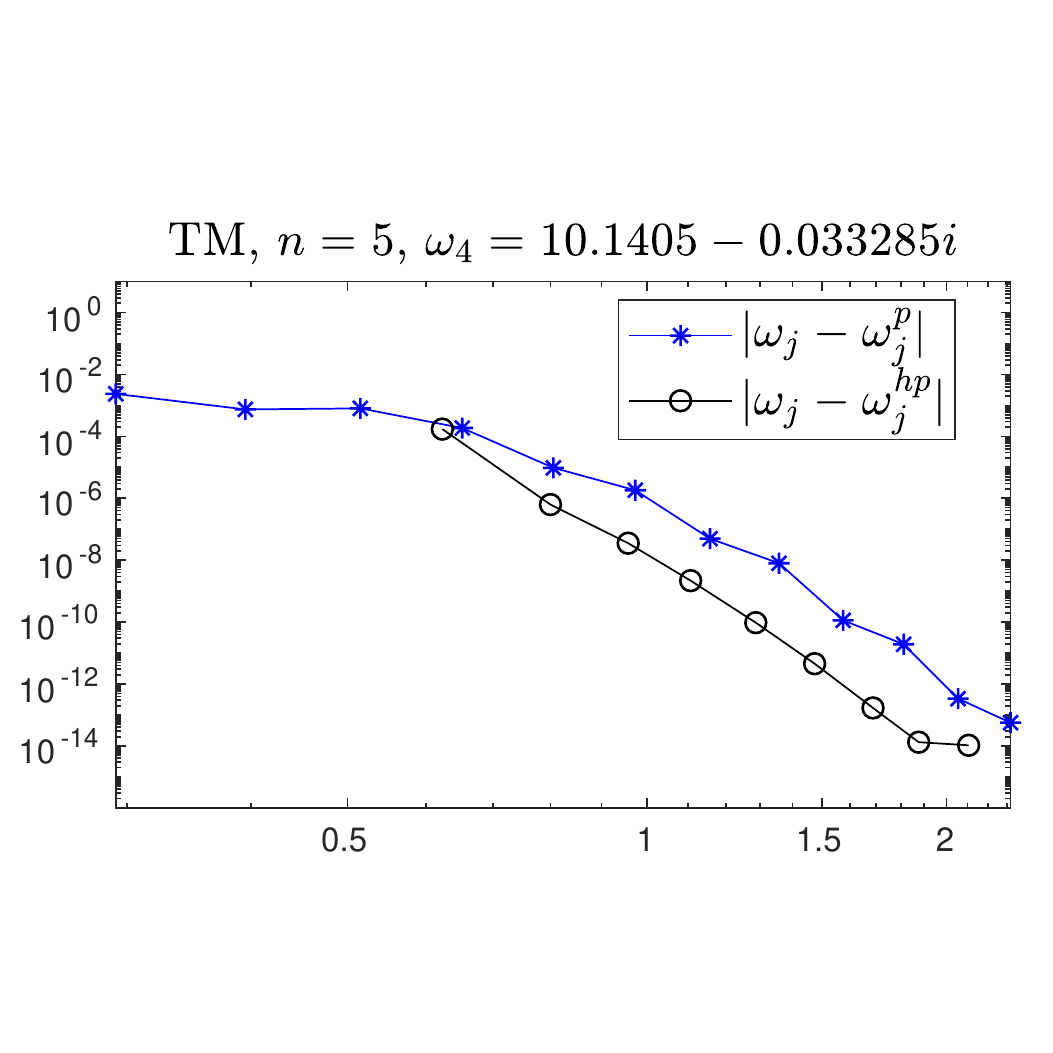}  };
		
		\draw(  0.00, 3.5) node {\includegraphics[scale=0.53]{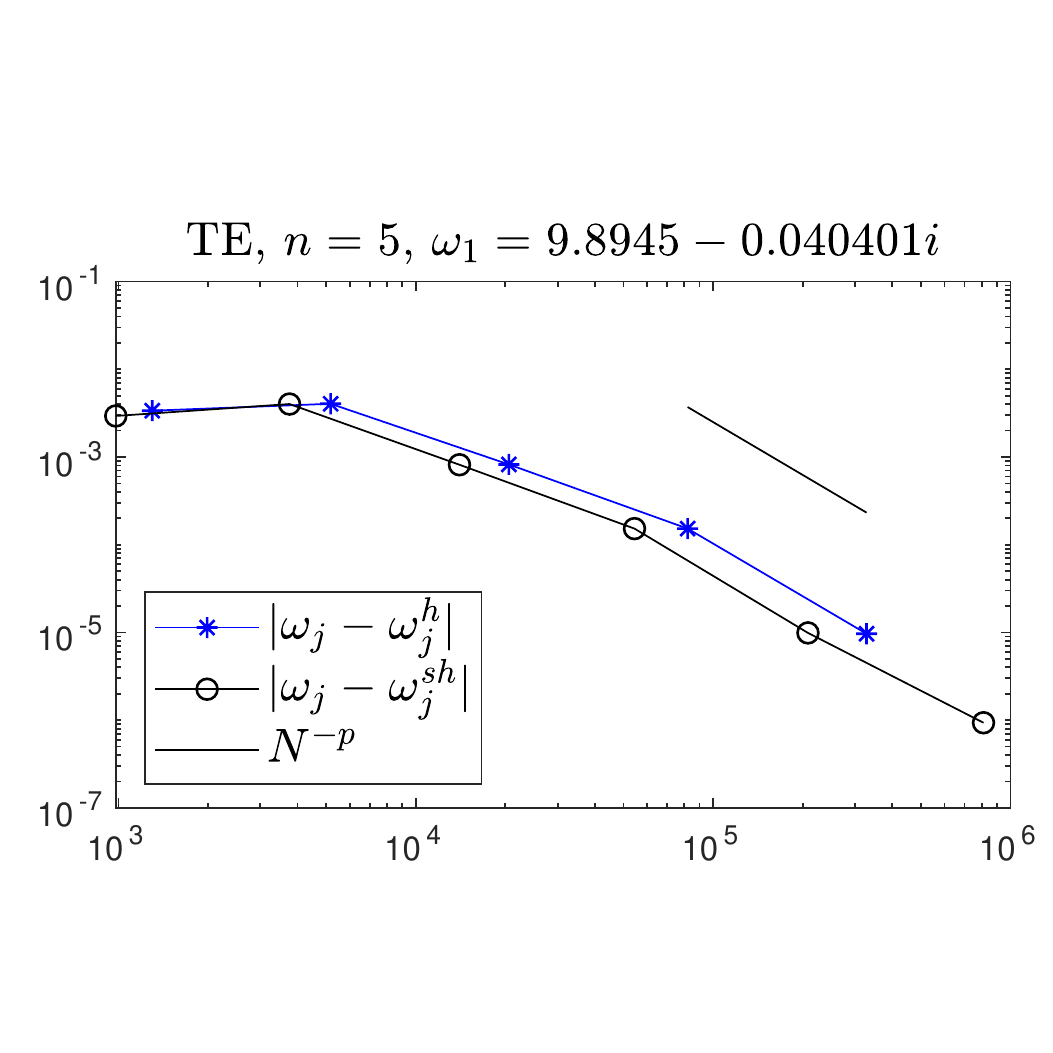}  };
		\draw(  5.50, 3.5) node {\includegraphics[scale=0.53]{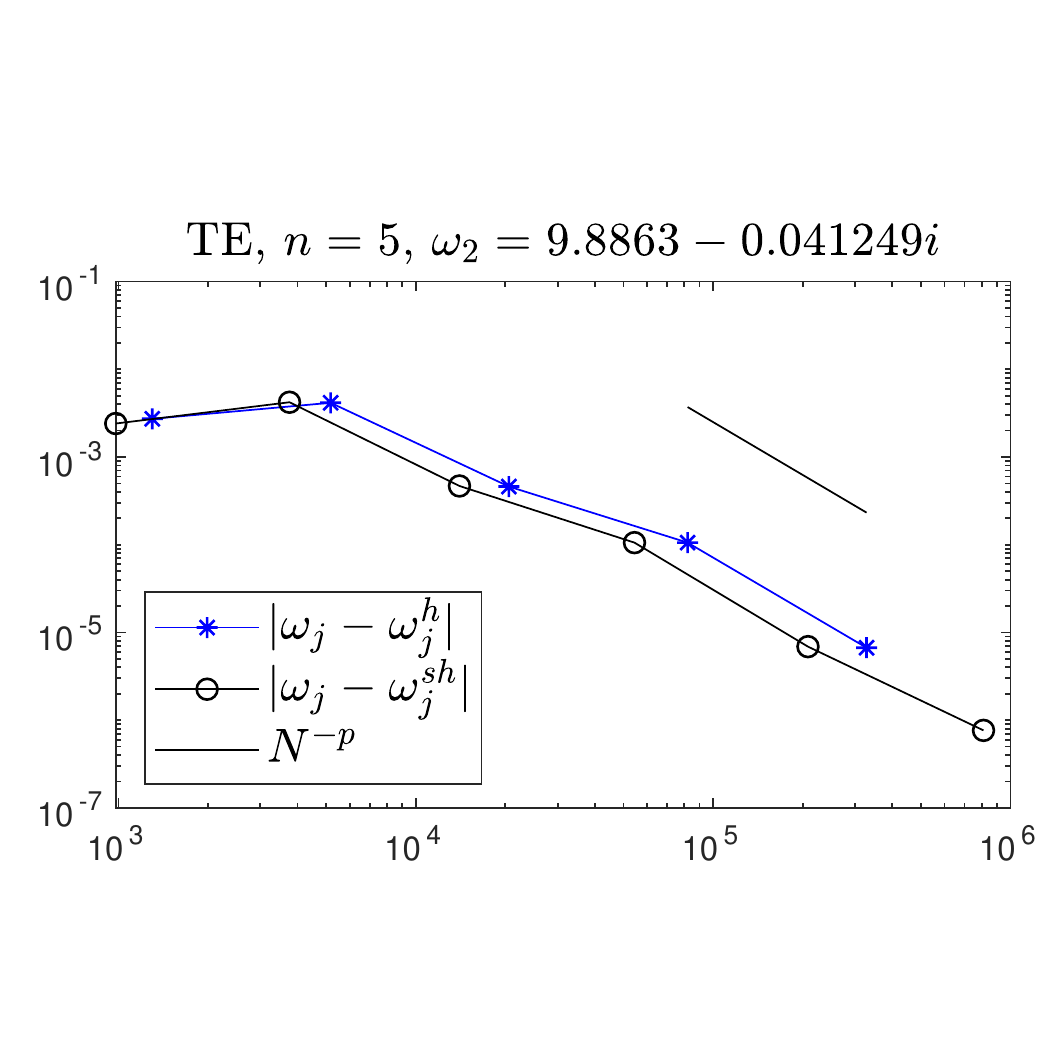}  };
		\draw( 11.00, 3.5) node {\includegraphics[scale=0.53]{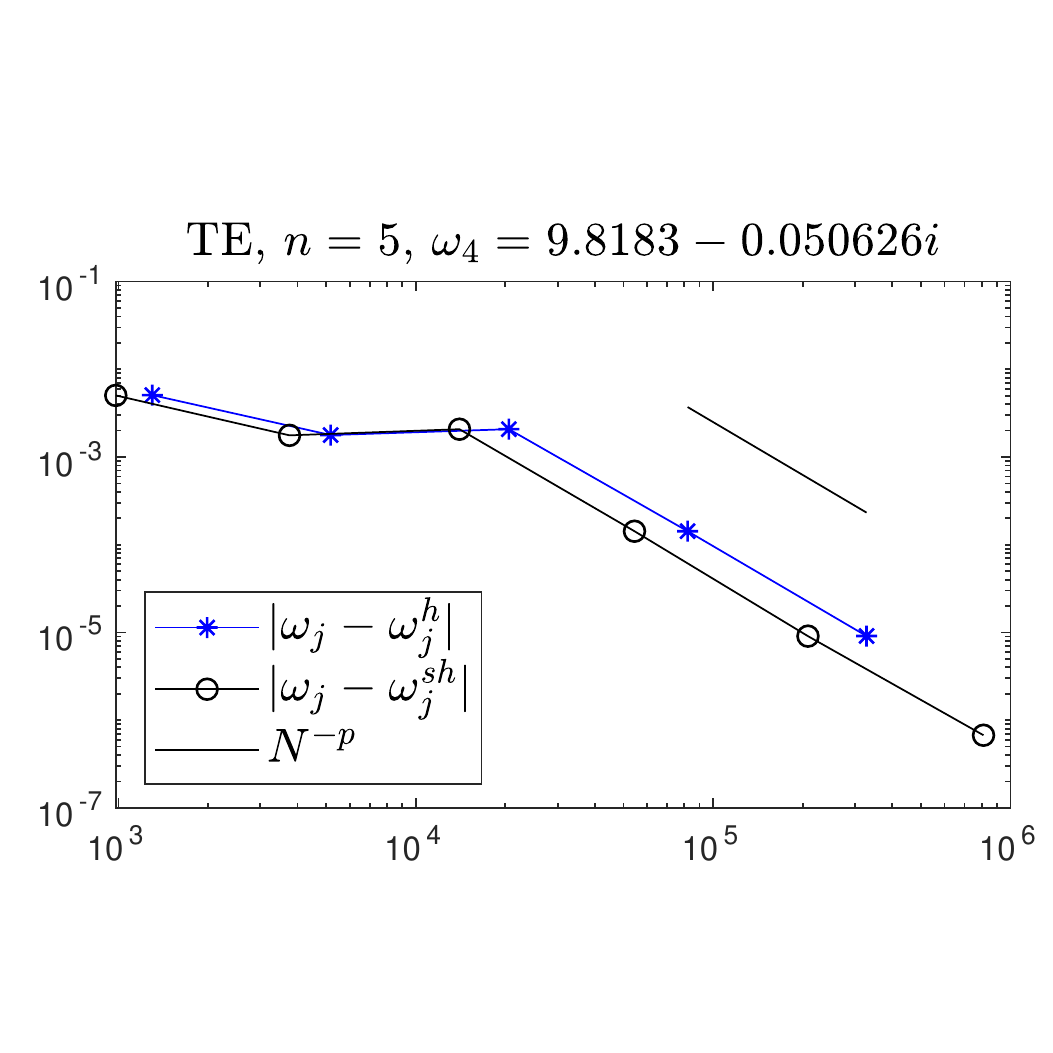}  };
		
		\draw(  0.00, 0.0) node {\includegraphics[scale=0.53]{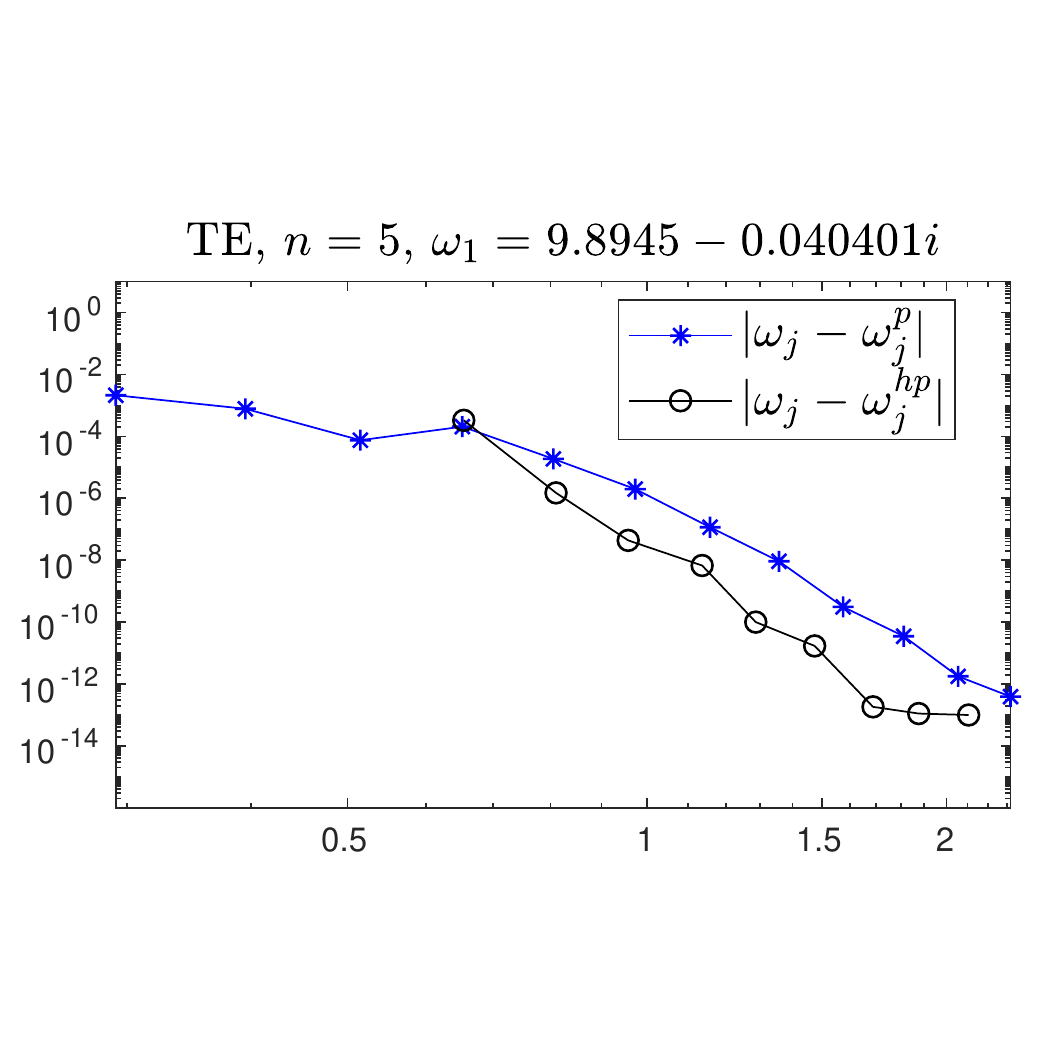}  };
		\draw(  5.50, 0.0) node {\includegraphics[scale=0.53]{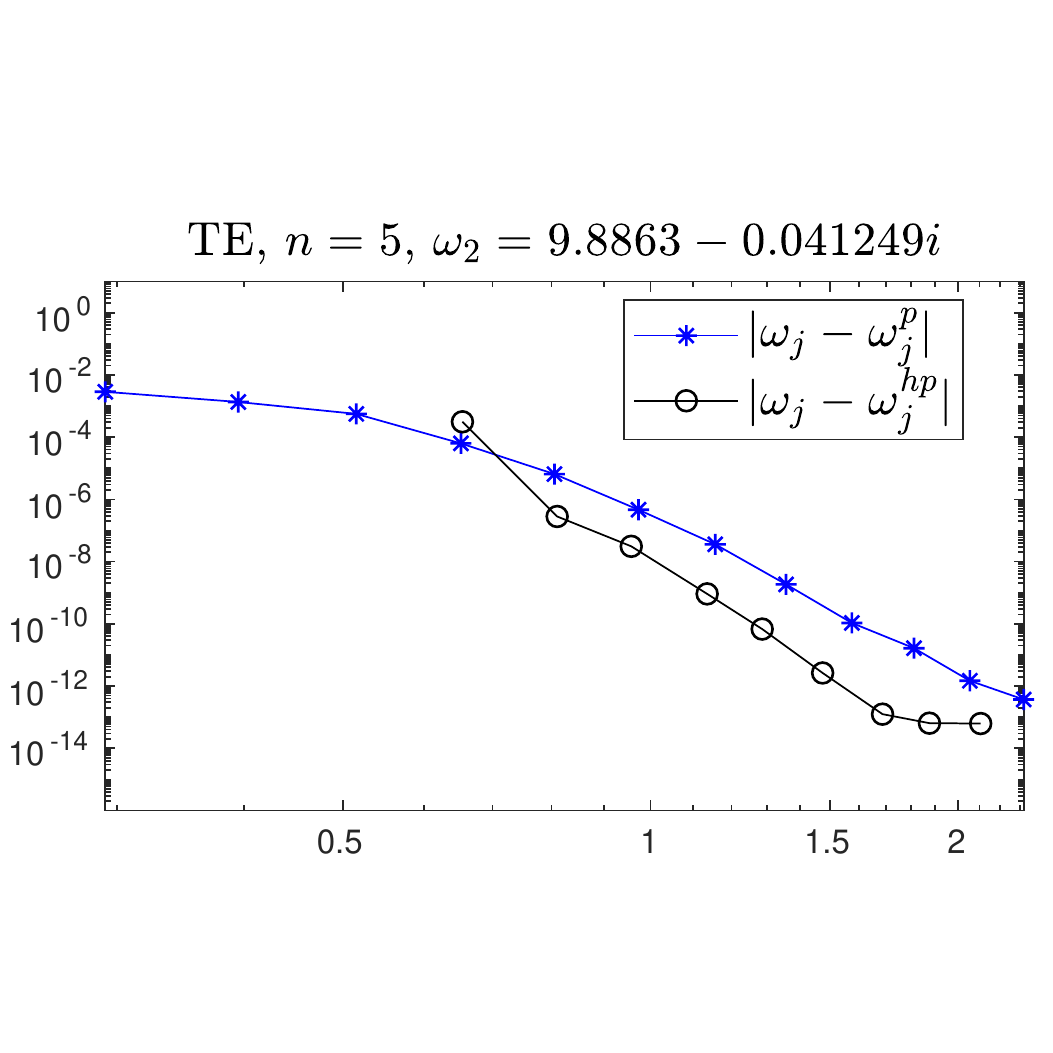}  };
		\draw( 11.00, 0.0) node {\includegraphics[scale=0.53]{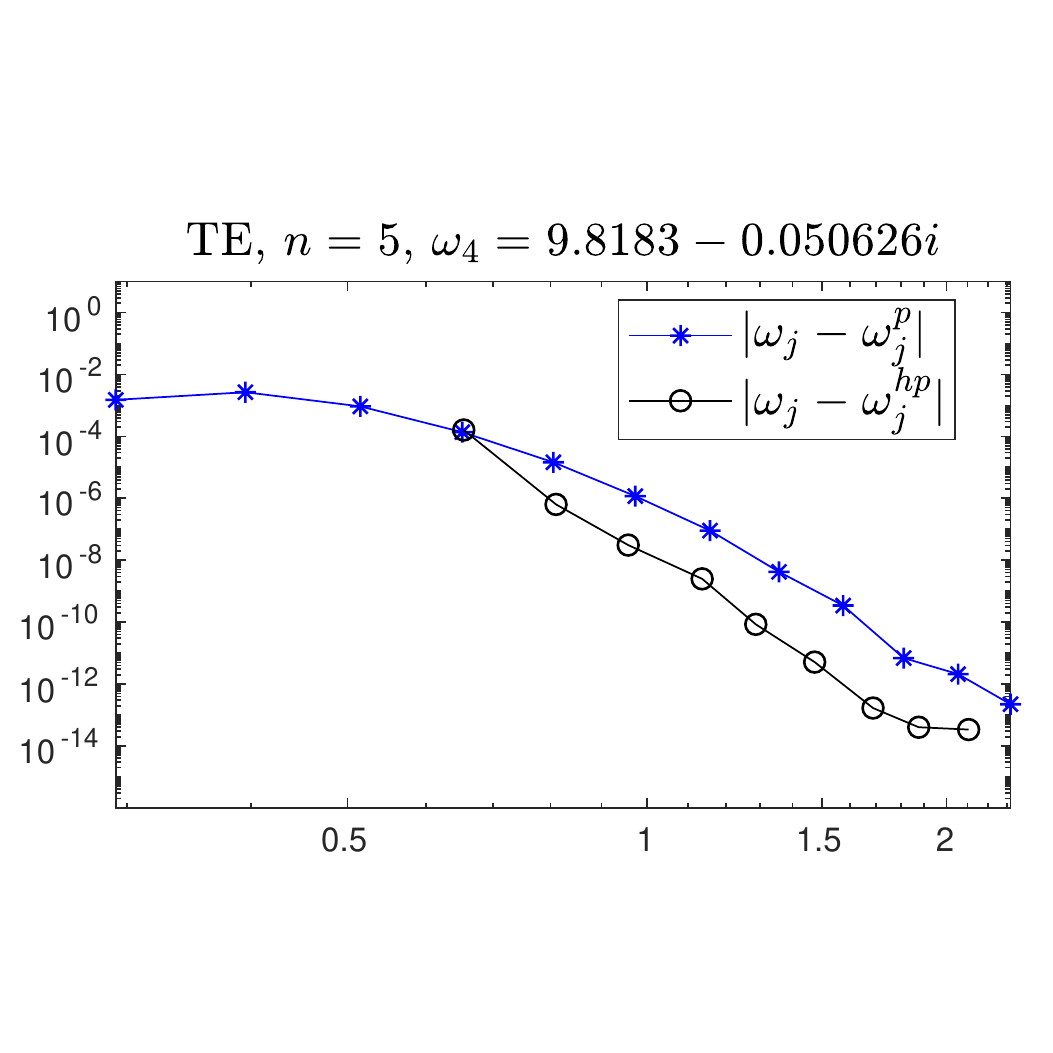}  };
		
		\node[text width=3cm] at ( -0.50, -1.52) {\tiny $10^4\times$};
		\node[text width=3cm] at (  5.10, -1.52) {\tiny $10^4\times$};
		\node[text width=3cm] at ( 10.60, -1.52) {\tiny $10^4\times$};
		
		\node[text width=3cm] at ( -0.50, 5.48) {\tiny $10^4\times$};
		\node[text width=3cm] at (  5.10, 5.48) {\tiny $10^4\times$};
		\node[text width=3cm] at ( 10.60, 5.48) {\tiny $10^4\times$};
		
	\end{tikzpicture}
	\vspace*{-10mm}
	\caption{\emph{Convergence plots (Relative errors vs. N) for TM and TE polarizations: \emph{single disk} problem \ref{sec:SD} and contrast $n_1=5$. Upper panels correspond to $h$-FE for $p=2$, and bottom panels to $p$-FE convergence. We mark with circles the a-priori strategies $sh$-FE and $hp$-FE, and with stars classical FE refinements. Each vertical strip shows different eigenpairs with $j=1,2,4$, featuring different angular numbers $m=0,2,6$.}}
	\label{fig:convergence_n5_SD}
\end{figure}

\subsubsection{Single coated disk problem}\label{sec:SDC} 

\begin{figure}[!h]
	\centering
	\begin{tikzpicture}[thick,scale=0.6, every node/.style={scale=1.0}]
		\tikzstyle{ann} = [fill=none,font=\large,inner sep=4pt]
		\draw( 13.00, 0.0) node { \includegraphics[scale=0.78]{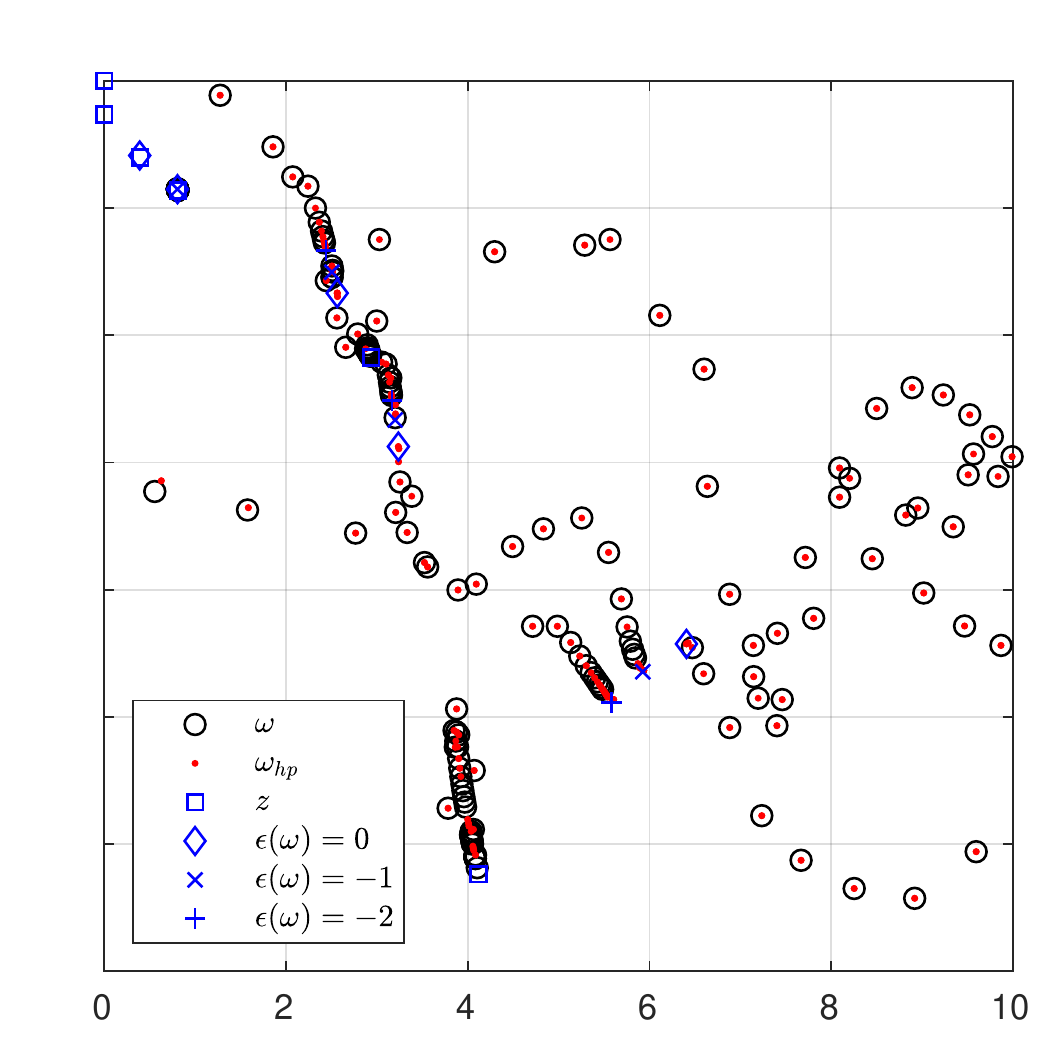} };
		\draw(  0.00, 0.0) node { \includegraphics[scale=0.78]{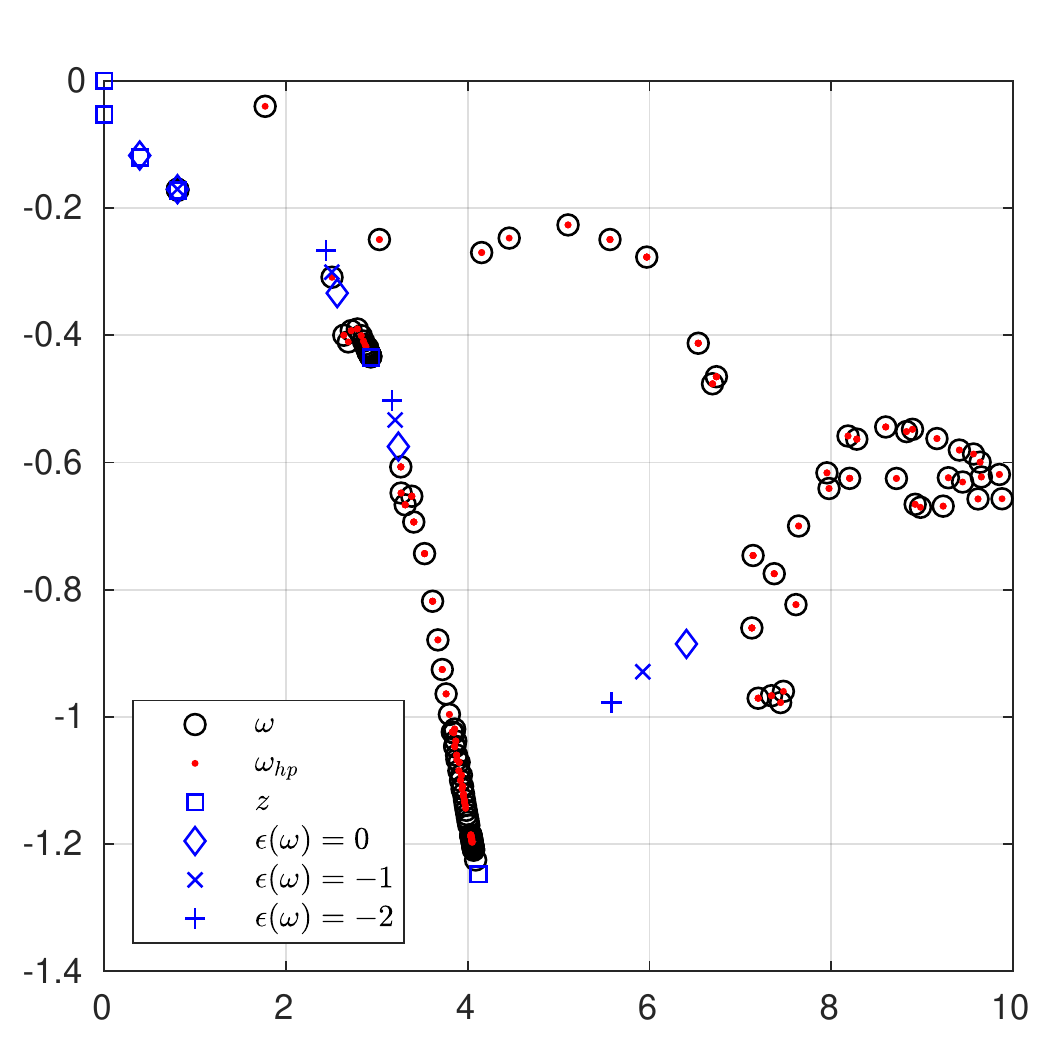} };
	\end{tikzpicture}
	\caption{\emph{Spectral window for the \emph{single coated disk} problem: polarizations TM (left) and TE (right). We mark with circles reference (Newton) eigenvalues $\om$ computed from \eqref{eq:reson_CSD}. FEM+NLEIGS eigenvalues $\om_\fem$ for a discretization $p=10,\,r=2$ are shown with dots, poles $z$ of $\eps(\om)$ with squares and its zeros with diamonds. The plasmonic branch points are marked with $\times$ and $+$.  }}
	\label{fig:eigs_SCD}
\end{figure}
In this configuration, we consider a resonator consisting of a dielectric disk with a uniform coating layer. The geometry is described by two concentric circumferences of radii $0<R_1<R_2$, with vacuum as surrounding medium. The inner disk has constant relative permittivity index, and is coated by a layer of gold. We set $n_1=\sqrt{\eps_s},$ and $n_2:=\sqrt{\eps_{metal}}$ is the value such that $\im n_2$ (absorption coefficient) is positive.  

The exact solutions satisfy \eqref{eq:master_eq}, and \eqref{eq:outgoing} with $R\geq R_2$. 
The resonance relationship reads
\begin{equation}
	\begin{array}{lcl}
	f^m_{1}(\om) &=& 	g_1 J^\prime_m(\om n_1 R_1) H^{(1)}_m(\om n_2  R_1) - 
										g_2 J_m(\om n_1 R_1) H^{(1)\prime}_m(\om n_2 R_1) , \\
	f^m_{2}(\om) &=& 	g_3 J_m(\om n_1 R_1) H^{(2)\prime}_m(\om n_2 R_1) -  
										g_4 J^\prime_m(\om n_1 R_1) H^{(2)}_m(\om n_2 R_1), \\
	f^m_{3}(\om) &=& 	g_5 H^{(1)}_m(\om n_2 R_2) H^{(1)\prime}_m(\om R_2) -   
										g_6 H^{(1)\prime}_m(\om n_2 R_2) H^{(1)}_m(\om R_2), \\
	f^m_{4}(\om) &=& 	g_7 H^{(1)}_m(\om R_2) H^{(2)\prime}_m(\om n_2 R_2) - 
										g_8 H^{(1)\prime}_m(\om R_2) H^{(2)}_m(\om n_2 R_2), \\[2mm]
	F_m(\om) &:=& (f^m_{1} f^m_{4} - f^m_{2} f^m_{3})(\om)=0,
	\end{array}
	\label{eq:reson_CSD}
\end{equation}
where for TM, $g:=(n_1,n_2,n_2,n_1,1,n_2,n_2,1)$, and for TE, $g:=(n_2,n_1,n_1,n_2,n_2,1,1,n_2)$. 
The parameters used for the computation are $R_1=0.8$, $R_2=1.0$ with scaling factor $L=1239.842\,nm$. 

A complex Newton root finder \cite{Yau98} is then used to compute very accurate approximations of the resonances. 
For each $m$ in equation \eqref{eq:reson_CSD}, we search numerically the resonances $\om_{m, 1},\om_{m, 2},\ldots$
with machine precision stopping criterion. In Table \ref{tab:CSD_reference}, we list a selection of resonances computed from \eqref{eq:reson_CSD}, which are used as a benchmark for studying 
the proposed $hp$-FE strategies \ref{sec:h_str} and \ref{sec:p_str} 
together with the proposed NEP strategy. 

\subsection{Coated disk dimer problem}\label{sec:Cdimer}
\begin{figure}
	\begin{tikzpicture}[thick,scale=1.0, every node/.style={scale=0.9}]
		\draw(  0.00, 0.0) node {\includegraphics[scale=0.92]{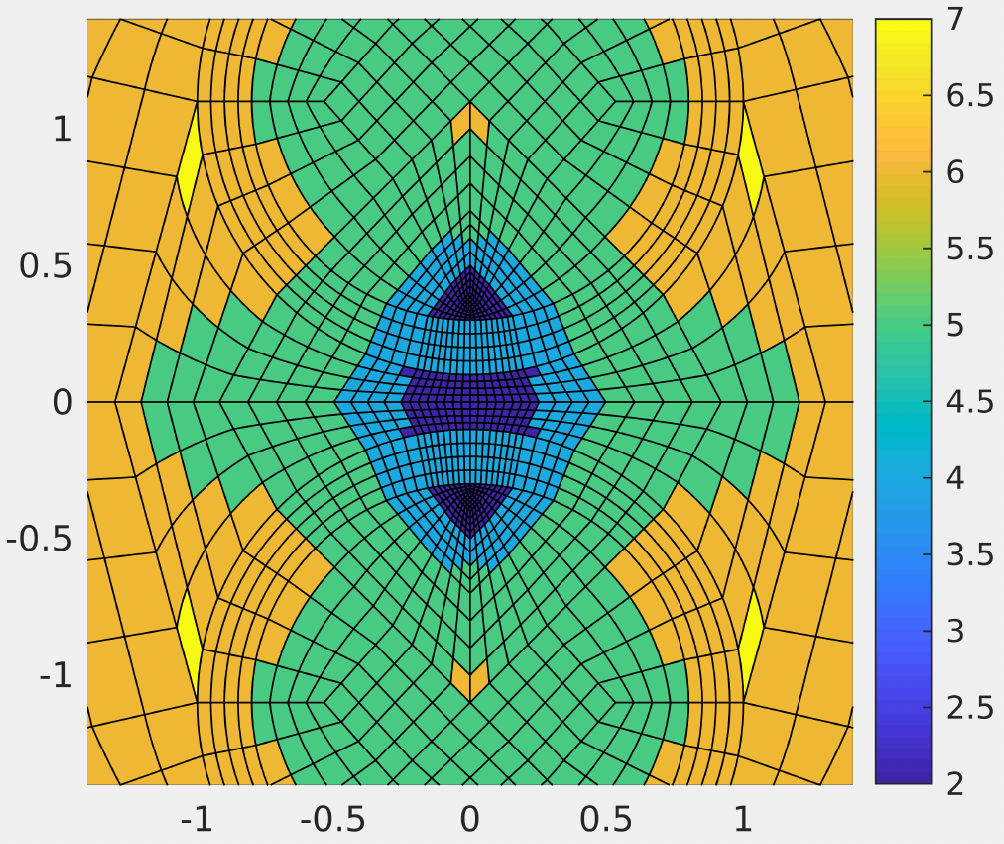} };
		\draw(  8.10, 0.0) node {\includegraphics[scale=0.92,trim=7.5mm 0mm 0mm 0mm,clip]{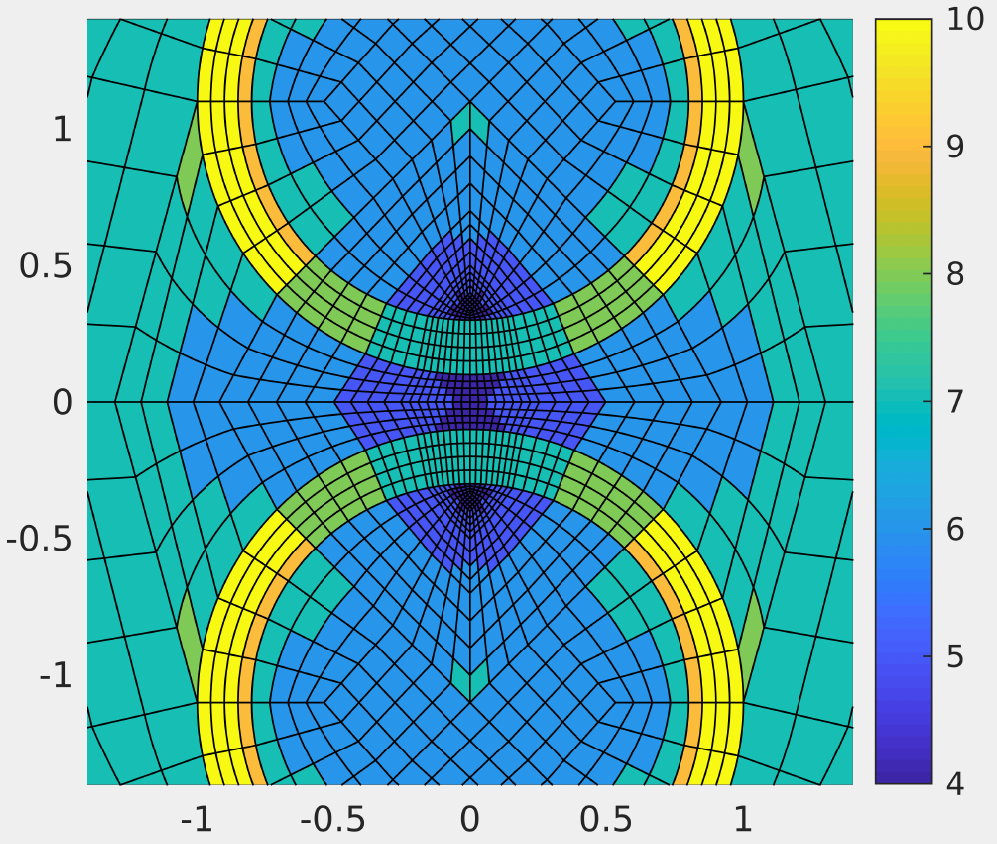} };
	\end{tikzpicture}
	\vspace*{-5mm}
	\caption{\emph{Resulting polynomial degree distribution $p_j$ from strategy in section \ref{sec:p_str}
	for shifts $\mu=4.162-0.2648i$ (left), $\mu=2.9-0.422i$ (right), and corresponding start values $p_0=7$ and $p_0=10$, respectively. In colors we give the computed $p_j$. }}
	\label{fig:grid_cdimer}
\end{figure}

The final configuration consists of two coated disks, each with equal dimension as the one presented in Sec. \ref{sec:SD}. The coated disks are surrounded by vacuum, and are separated vertically by a distance $s=0.2$.
For this problem 
we compute reference solutions by solving the problem on a very fine mesh.


\begin{table}
\robustify\bfseries
\centering
\small
\begin{tabular}{rS[table-format=2.9, detect-weight]
                 S[table-format=2.9, detect-weight]
                 S[table-format=2.9, detect-weight]
                 S[table-format=2.9, detect-weight] }
\toprule
{$j$} & {$\re \om_{T\!M}$} & {$\im \om_{T\!M}$} & {$\re \om_{T\!E}$} & {$\im \om_{T\!E}$} \\
\midrule
 1  &  0.391206696  &  -0.117682733    &  1.275203310  &  -0.017729356 \\
 2  &  0.392635042  &  -0.118062545    &  1.407446763  &  -0.351566607 \\
 3  &  0.809151314  &  -0.171363257    &  1.518834290  &  -0.459714126 \\
 4  &  1.775357827  &  -0.032801891    &  1.833732651  &  -0.071556523 \\
 5  &  2.553994710  &  -0.278675516    &  2.122066617  &  -0.294491051 \\
 6  &  2.654205934  &  -0.403812208    &  2.212801536  &  -0.162351126 \\
 7  &  2.889635797  &  -0.420955865    &  2.904693880  &  -0.427303148 \\
 8  &  2.907250975  &  -0.426266060    &  2.905243670  &  -0.423913316 \\
 9  &  3.613338577  &  -0.811314585    &  2.905244289  &  -0.423897437 \\
10  &  3.668597318  &  -0.875944785    &  3.034308619  &  -0.251533276 \\
11  &  4.152064265  &  -0.280206808    &  3.104957066  &  -0.446726774 \\
12  &  4.459080468  &  -0.244143721    &  4.087887349  &  -0.794352424 \\
13  &  5.565335954  &  -0.248309435    &  4.297453212  &  -0.269106228 \\
14  &  5.952378524  &  -0.278171454    &  4.491337149  &  -0.744646034 \\
15  &  6.175153470  &  -0.890257000    &  5.248059111  &  -0.596147639 \\
16  &  6.636767136  &  -0.423048913    &  5.560206560  &  -0.251973018 \\
17  &  6.672655322  &  -0.452937916    &  5.809753794  &  -0.884187579 \\
18  &  7.251741950  &  -0.616774531    &  6.078619602  &  -0.730690710 \\
19  &  8.155796969  &  -0.628012814    &  7.196957492  &  -0.638546672 \\
20  &  8.855927311  &  -0.444228537    &  8.961786233  &  -0.520547479 \\
\bottomrule
\end{tabular}


%



\caption{\emph{Reference eigenvalues for the \emph{coated disk dimer} problem described in Section \ref{sec:Cdimer}.}}
\label{tab:CD_reference}
\end{table}

\begin{figure}[!h]
	\centering
	\begin{tikzpicture}[thick,scale=0.6, every node/.style={scale=1.0}]
		\tikzstyle{ann} = [fill=none,font=\large,inner sep=4pt]
		
		\draw(  -0.20, 0.0) node { \includegraphics[scale=0.78]{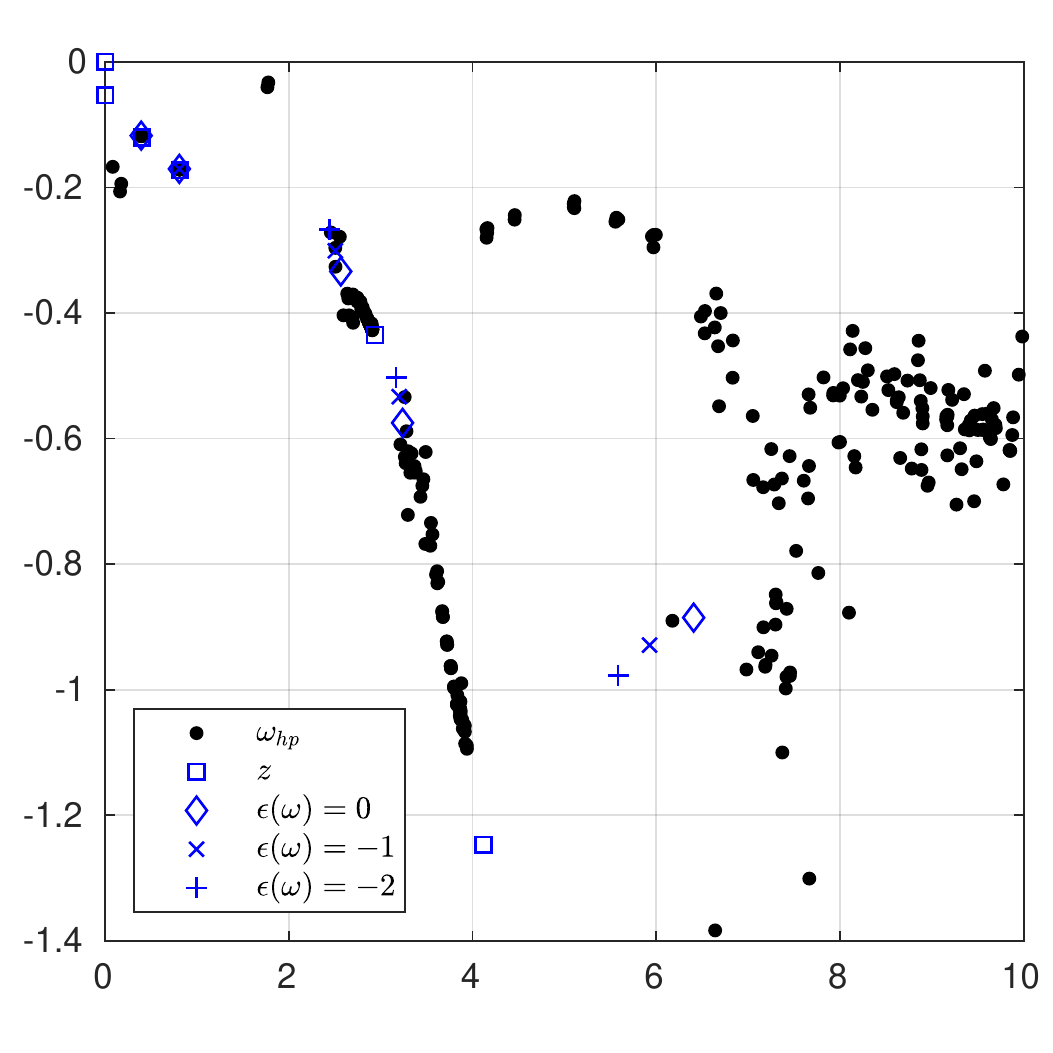} };
		
		\node[scale=0.8] at ( 8.50-1.5, 5.0+1.3 ) {$\om_{1}$}; 
		\draw( 8.50, 5.0) node { \includegraphics[scale=0.1]{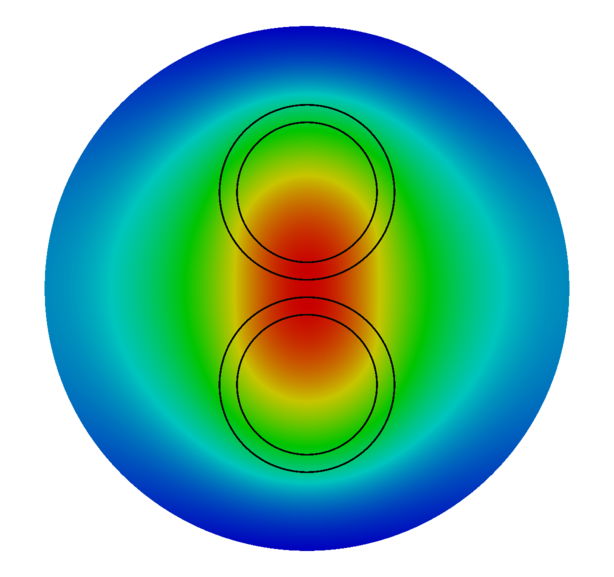} }; %
		\node[scale=0.8] at ( 8.50-1.5, 2.5+1.3 ) {$\om_{5}$};
		\draw( 8.50, 2.5) node { \includegraphics[scale=0.1]{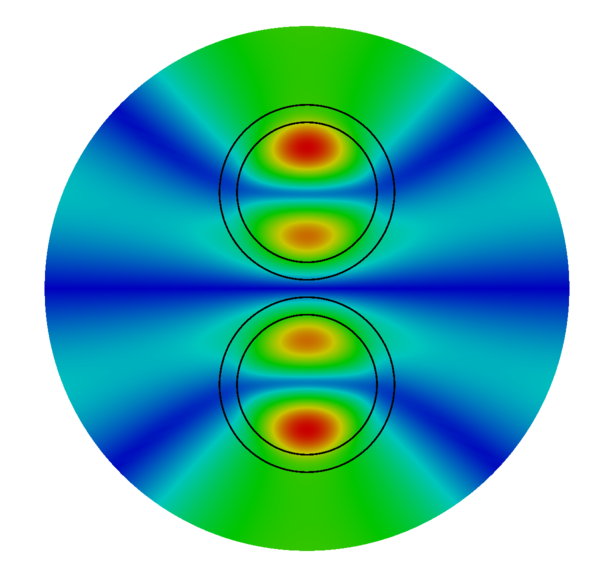} }; 
		\node[scale=0.8] at ( 8.50-1.5, 0.0+1.3 ) {$\om_{9}$};
		\draw( 8.50, 0.0) node { \includegraphics[scale=0.1]{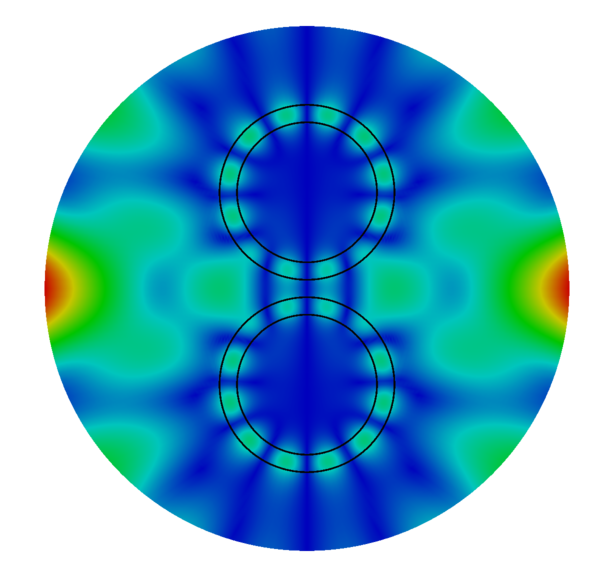} }; 
		\node[scale=0.8] at ( 8.50-1.5,-2.5+1.3 ) {$\om_{13}$};
		\draw( 8.50,-2.5) node { \includegraphics[scale=0.1]{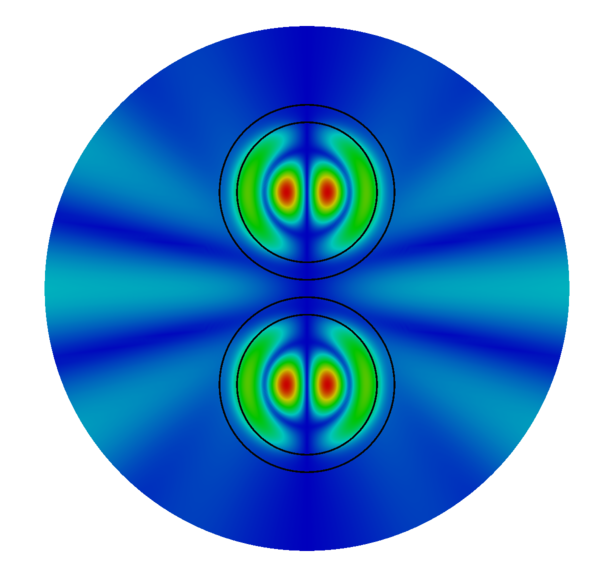} }; 
		\node[scale=0.8] at ( 8.50-1.5,-5.0+1.3 ) {$\om_{17}$};
		\draw( 8.50,-5.0) node { \includegraphics[scale=0.1]{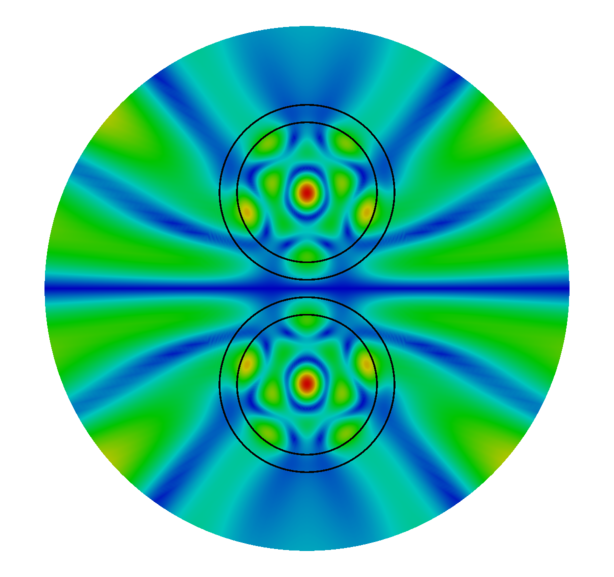} }; 
		
		\node[scale=0.8] at ( 11.80-1.5, 5.0+1.3 ) {$\om_{2}$};
		\draw( 11.80, 5.0) node { \includegraphics[scale=0.1]{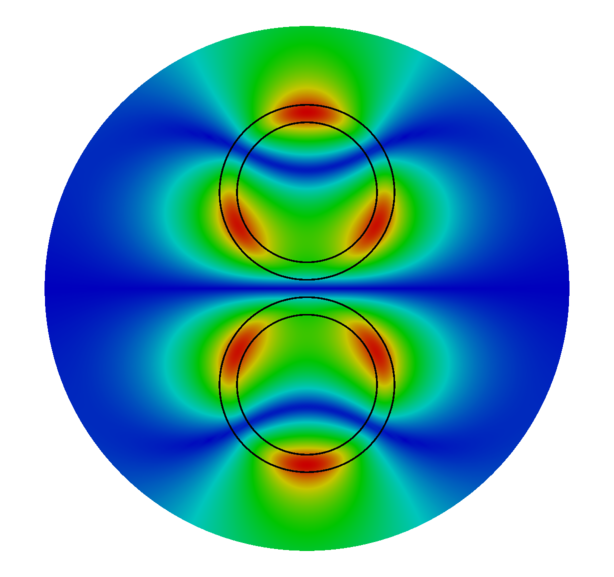} }; 
		\node[scale=0.8] at ( 11.80-1.5, 2.5+1.3 ) {$\om_{6}$};
		\draw( 11.80, 2.5) node { \includegraphics[scale=0.1]{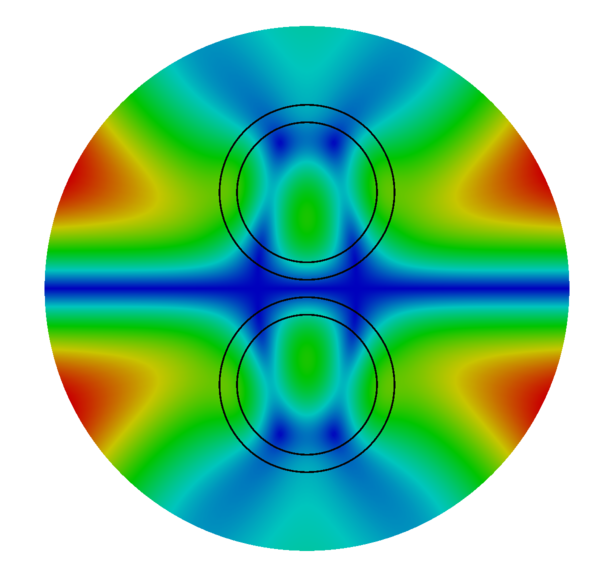} }; 
		\node[scale=0.8] at ( 11.80-1.5, 0.0+1.3 ) {$\om_{10}$};
		\draw( 11.80, 0.0) node { \includegraphics[scale=0.1]{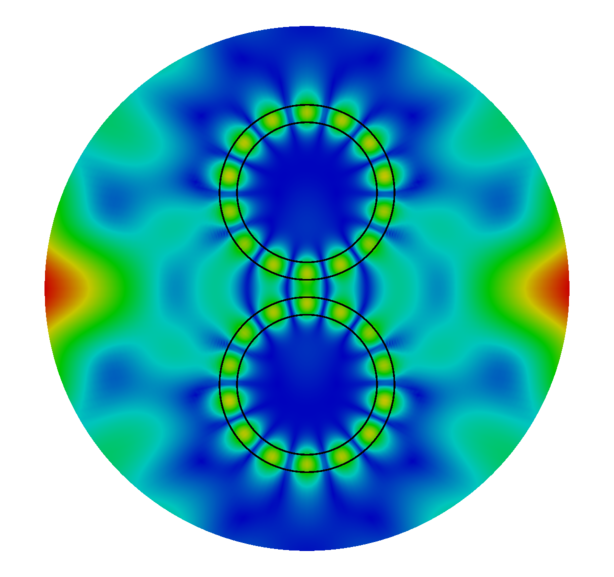} }; 
		\node[scale=0.8] at ( 11.80-1.5,-2.5+1.3 ) {$\om_{14}$};
		\draw( 11.80,-2.5) node { \includegraphics[scale=0.1]{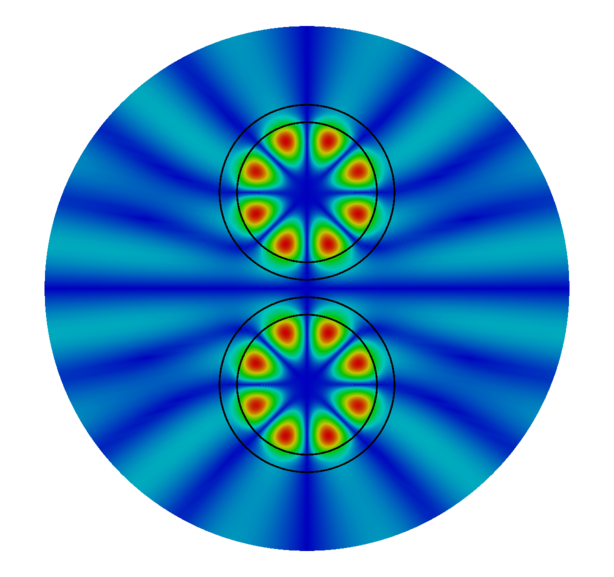} }; 
		\node[scale=0.8] at ( 11.80-1.5,-5.0+1.3 ) {$\om_{18}$};
		\draw( 11.80,-5.0) node { \includegraphics[scale=0.1]{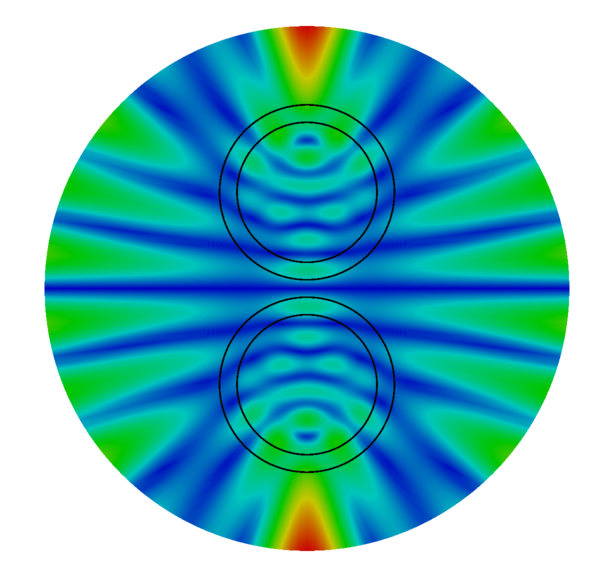} }; 
		
		\node[scale=0.8] at ( 15.10-1.5, 5.0+1.3 ) {$\om_{3}$};
		\draw( 15.10, 5.0) node { \includegraphics[scale=0.1]{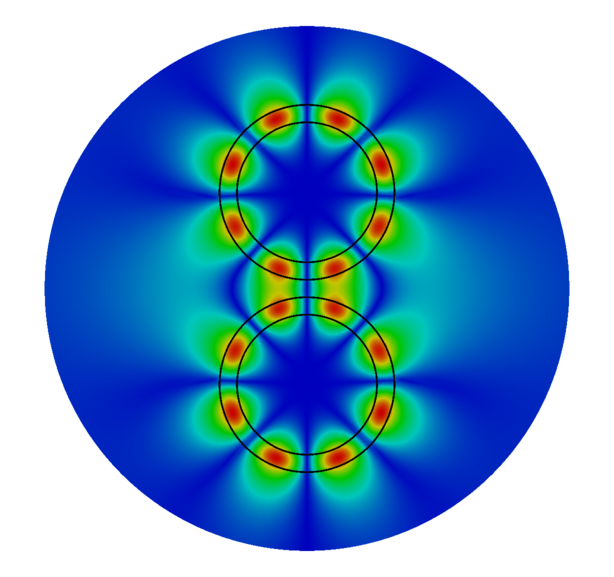} }; 
		\node[scale=0.8] at ( 15.10-1.5, 2.5+1.3 ) {$\om_{7}$};
		\draw( 15.10, 2.5) node { \includegraphics[scale=0.1]{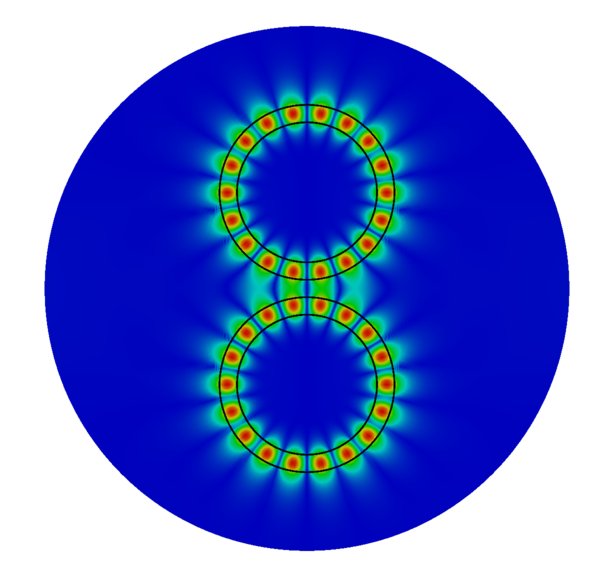} }; %
		\node[scale=0.8] at ( 15.10-1.5, 0.0+1.3 ) {$\om_{11}$};
		\draw( 15.10, 0.0) node { \includegraphics[scale=0.1]{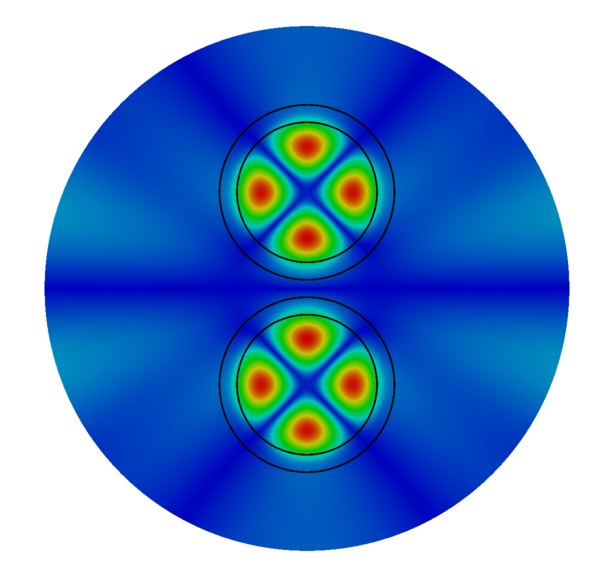} }; 
		\node[scale=0.8] at ( 15.10-1.5,-2.5+1.3 ) {$\om_{15}$};
		\draw( 15.10,-2.5) node { \includegraphics[scale=0.1]{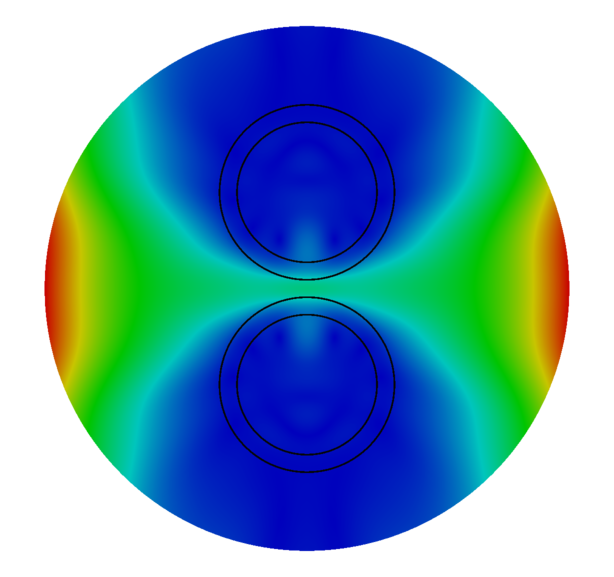} }; %
		\node[scale=0.8] at ( 15.10-1.5,-5.0+1.3 ) {$\om_{19}$};
		\draw( 15.10,-5.0) node { \includegraphics[scale=0.1]{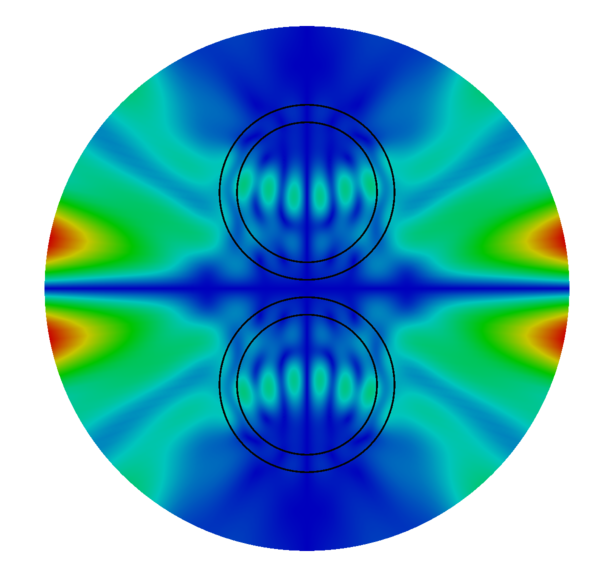} }; 
		
		\node[scale=0.8] at ( 18.40-1.5, 5.0+1.3 ) {$\om_{4}$};
		\draw( 18.40, 5.0) node { \includegraphics[scale=0.1]{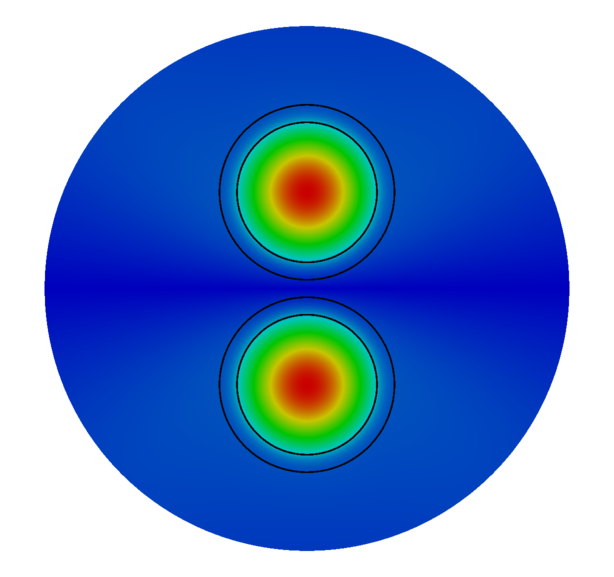} }; 
		\node[scale=0.8] at ( 18.40-1.5, 2.5+1.3 ) {$\om_{8}$};
		\draw( 18.40, 2.5) node { \includegraphics[scale=0.1]{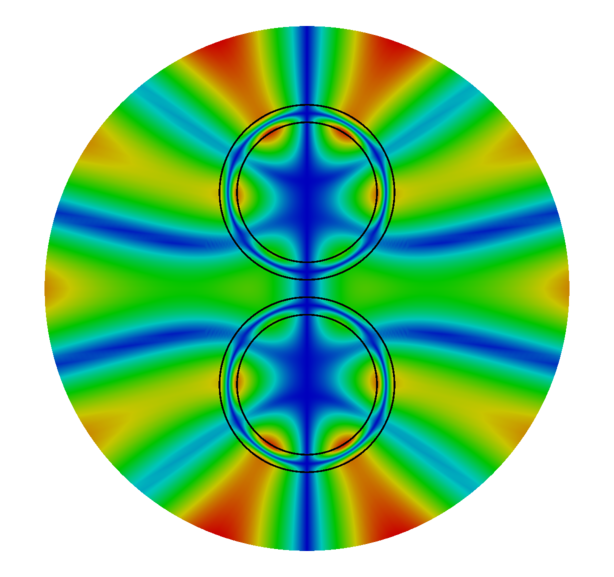} }; 
		\node[scale=0.8] at ( 18.40-1.5, 0.0+1.3 ) {$\om_{12}$};
		\draw( 18.40, 0.0) node { \includegraphics[scale=0.1]{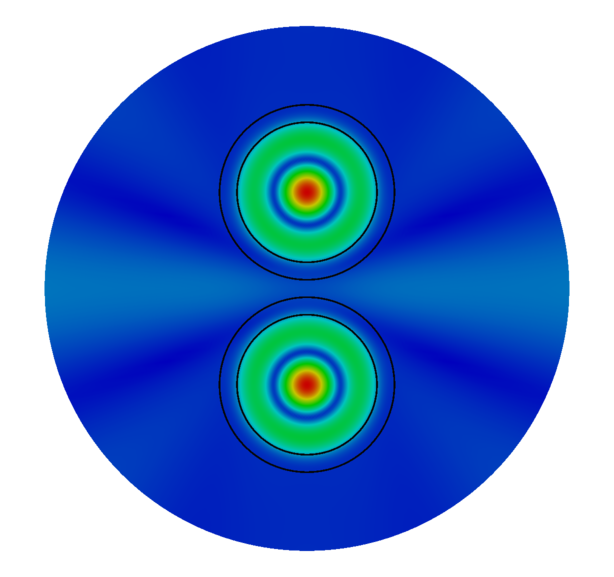} }; 
		\node[scale=0.8] at ( 18.40-1.5,-2.5+1.3 ) {$\om_{16}$};
		\draw( 18.40,-2.5) node { \includegraphics[scale=0.1]{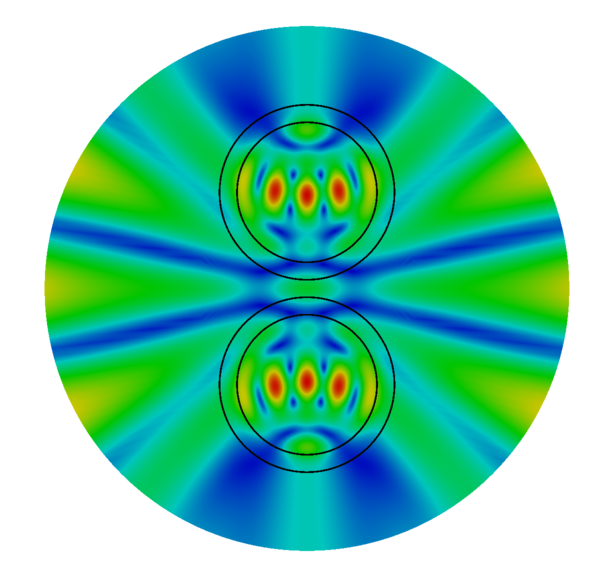} }; 
		\node[scale=0.8] at ( 18.40-1.5,-5.0+1.3 ) {$\om_{20}$};
		\draw( 18.40,-5.0) node { \includegraphics[scale=0.1]{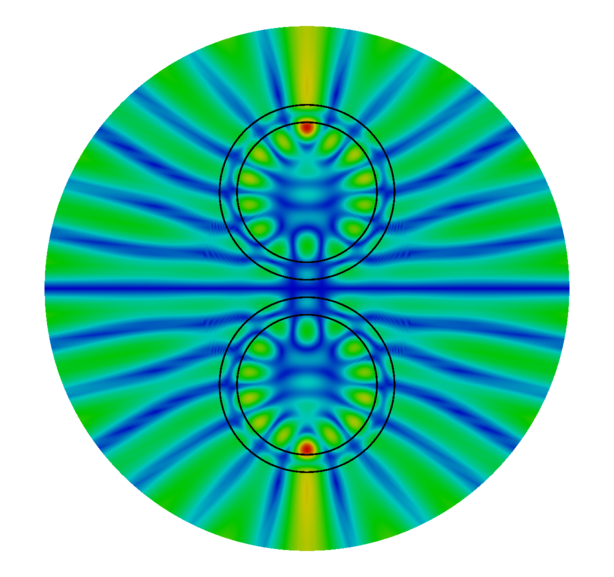} }; 
		
	\end{tikzpicture}
	\vspace*{-4mm}
	\caption{\emph{Left) Spectral window for the \emph{coated dimer} problem \ref{sec:Cdimer} with TM polarization. FEM+NLEIGS eigenvalues $\om_\fem$ are shown with dots, poles $z$ of $\eps(\om)$ with squares and its zeros with diamonds. The plasmonic branch points are marked with $\times$ and $+$. Right) In colors we plot $\|E_j\|$ from the resonant mode corresponding to $\om_j$ listed in table \ref{tab:CD_reference}. }}
	\label{fig:eigs_CD_TM}
\end{figure}


\begin{figure}[!h]
	\centering
	\begin{tikzpicture}[thick,scale=0.6, every node/.style={scale=1.0}]
		\tikzstyle{ann} = [fill=none,font=\large,inner sep=4pt]
		
		\draw( -0.20, 0.0) node { \includegraphics[scale=0.78]{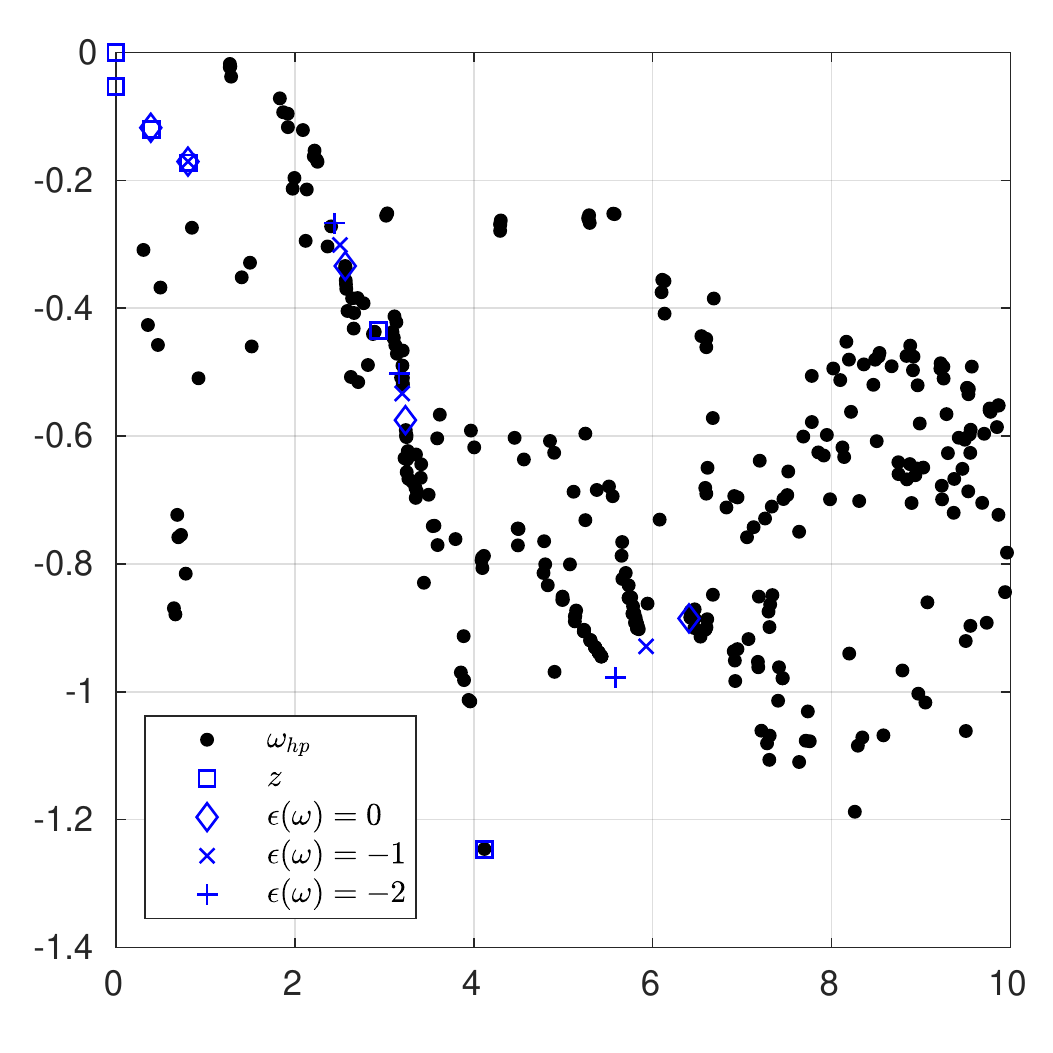} };
		
		\node[scale=0.8] at ( 8.50-1.5, 5.0+1.3 ) {$\om_{1}$}; 
		\draw( 8.50, 5.0) node { \includegraphics[scale=0.1]{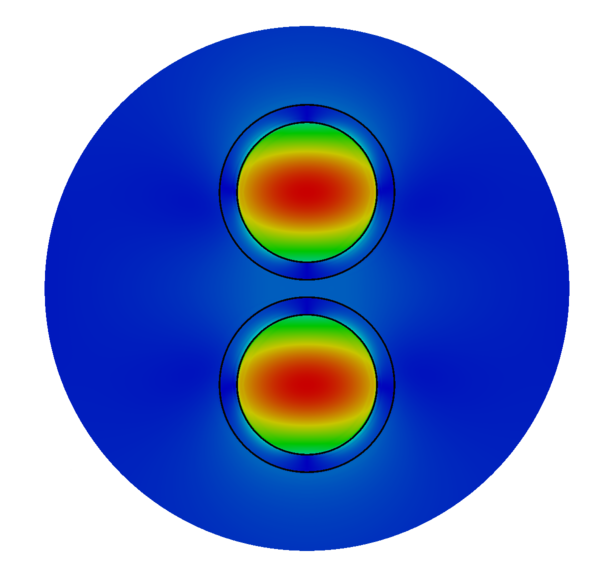} };
		\node[scale=0.8] at ( 8.50-1.5, 2.5+1.3 ) {$\om_{5}$};
		\draw( 8.50, 2.5) node { \includegraphics[scale=0.1]{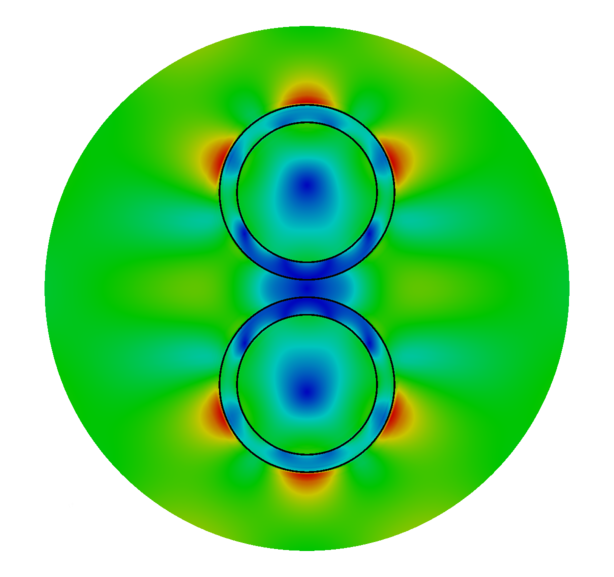} };
		\node[scale=0.8] at ( 8.50-1.5, 0.0+1.3 ) {$\om_{9}$};
		\draw( 8.50, 0.0) node { \includegraphics[scale=0.1]{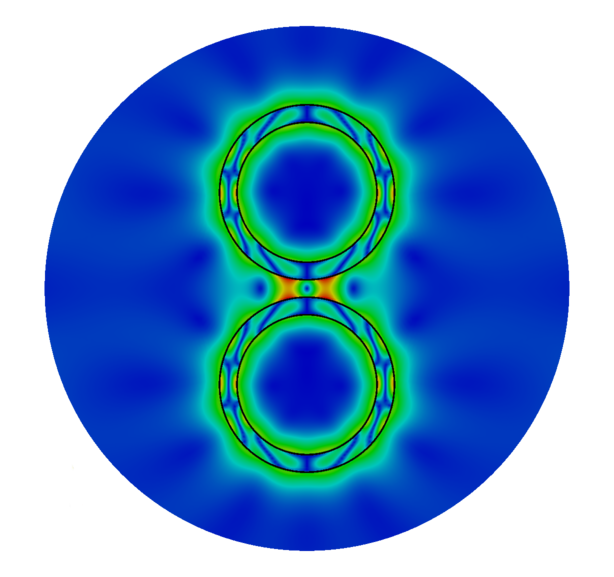} };
		\node[scale=0.8] at ( 8.50-1.5,-2.5+1.3 ) {$\om_{13}$};
		\draw( 8.50,-2.5) node { \includegraphics[scale=0.1]{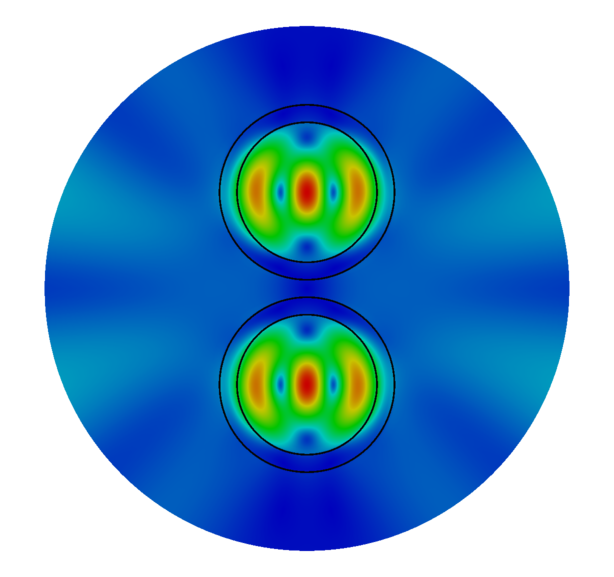} };
		\node[scale=0.8] at ( 8.50-1.5,-5.0+1.3 ) {$\om_{17}$};
		\draw( 8.50,-5.0) node { \includegraphics[scale=0.1]{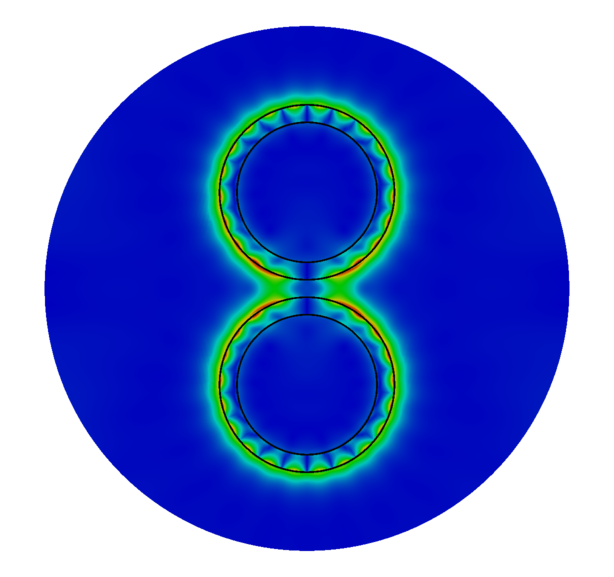} };
		
		\node[scale=0.8] at ( 11.80-1.5, 5.0+1.3 ) {$\om_{2}$};
		\draw( 11.80, 5.0) node { \includegraphics[scale=0.1]{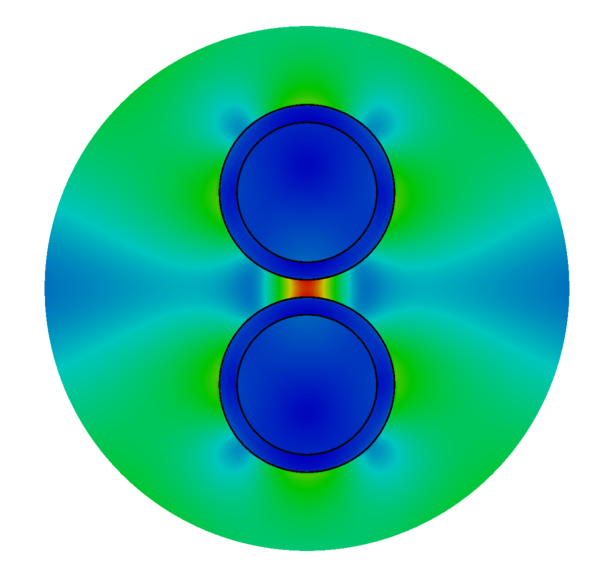} };
		\node[scale=0.8] at ( 11.80-1.5, 2.5+1.3 ) {$\om_{6}$};
		\draw( 11.80, 2.5) node { \includegraphics[scale=0.1]{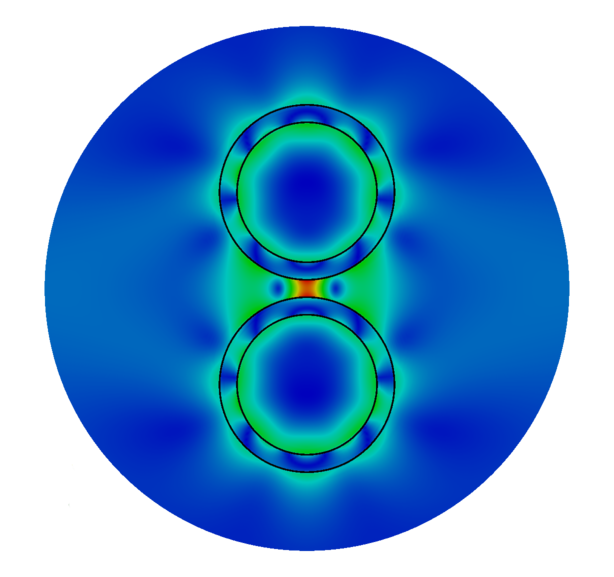} };
		\node[scale=0.8] at ( 11.80-1.5, 0.0+1.3 ) {$\om_{10}$};
		\draw( 11.80, 0.0) node { \includegraphics[scale=0.1]{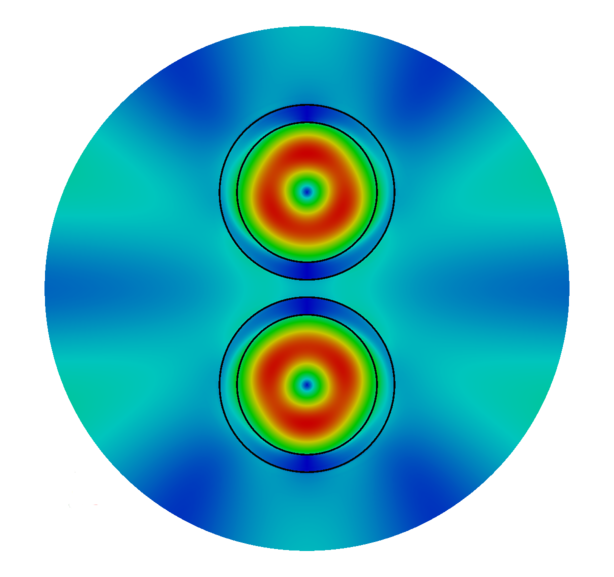} };
		\node[scale=0.8] at ( 11.80-1.5,-2.5+1.3 ) {$\om_{14}$};
		\draw( 11.80,-2.5) node { \includegraphics[scale=0.1]{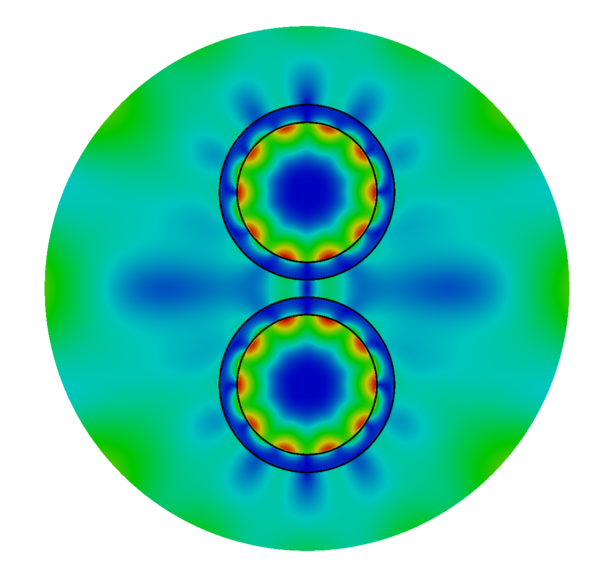} };
		\node[scale=0.8] at ( 11.80-1.5,-5.0+1.3 ) {$\om_{18}$};
		\draw( 11.80,-5.0) node { \includegraphics[scale=0.1]{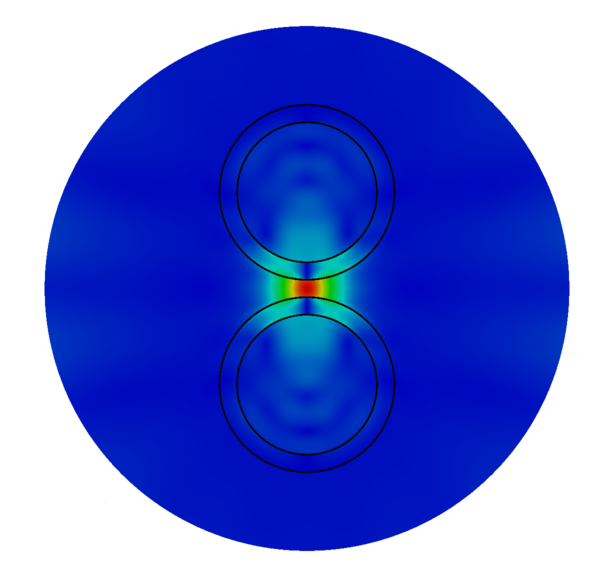} };
		
		\node[scale=0.8] at ( 15.10-1.5, 5.0+1.3 ) {$\om_{3}$};
		\draw( 15.10, 5.0) node { \includegraphics[scale=0.1]{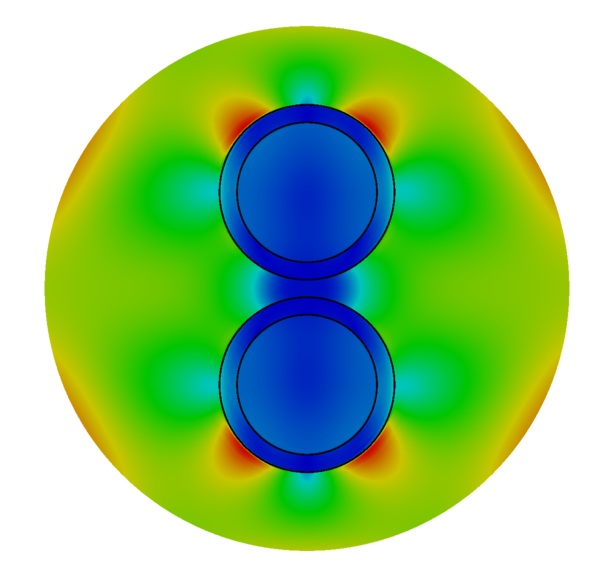} };
		\node[scale=0.8] at ( 15.10-1.5, 2.5+1.3 ) {$\om_{7}$};
		\draw( 15.10, 2.5) node { \includegraphics[scale=0.1]{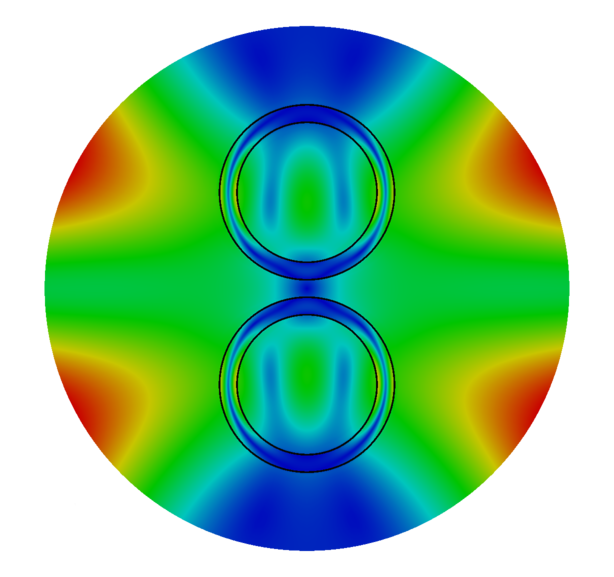} }; 
		\node[scale=0.8] at ( 15.10-1.5, 0.0+1.3 ) {$\om_{11}$};
		\draw( 15.10, 0.0) node { \includegraphics[scale=0.1]{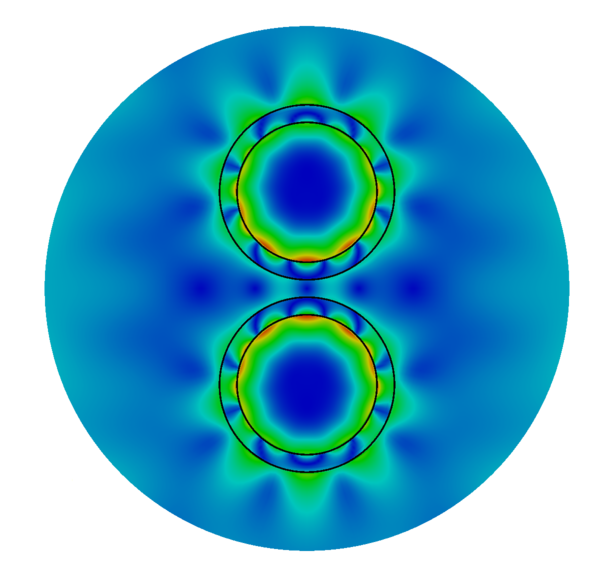} };
		\node[scale=0.8] at ( 15.10-1.5,-2.5+1.3 ) {$\om_{15}$};
		\draw( 15.10,-2.5) node { \includegraphics[scale=0.1]{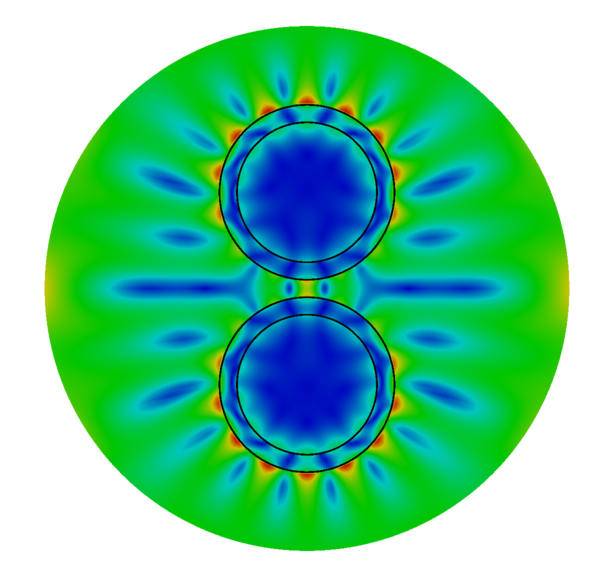} };
		\node[scale=0.8] at ( 15.10-1.5,-5.0+1.3 ) {$\om_{19}$};
		\draw( 15.10,-5.0) node { \includegraphics[scale=0.1]{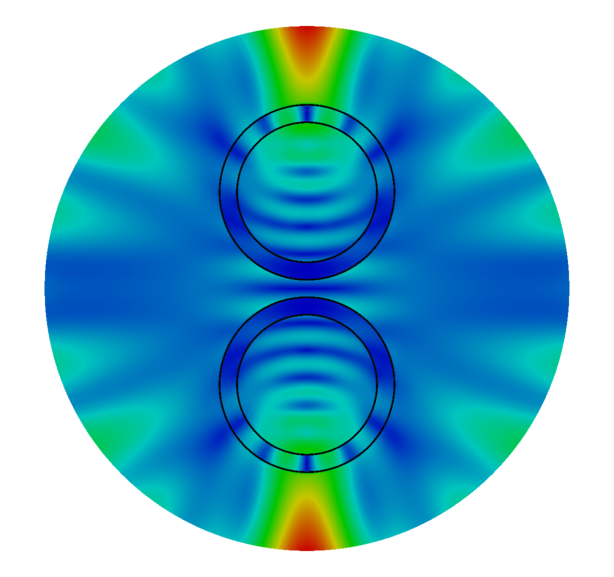} };
		
		\node[scale=0.8] at ( 18.40-1.5, 5.0+1.3 ) {$\om_{4}$};
		\draw( 18.40, 5.0) node { \includegraphics[scale=0.1]{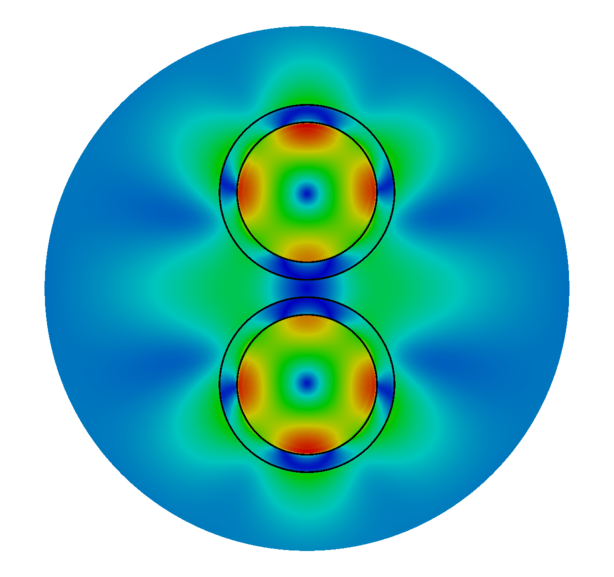} };
		\node[scale=0.8] at ( 18.40-1.5, 2.5+1.3 ) {$\om_{8}$};
		\draw( 18.40, 2.5) node { \includegraphics[scale=0.1]{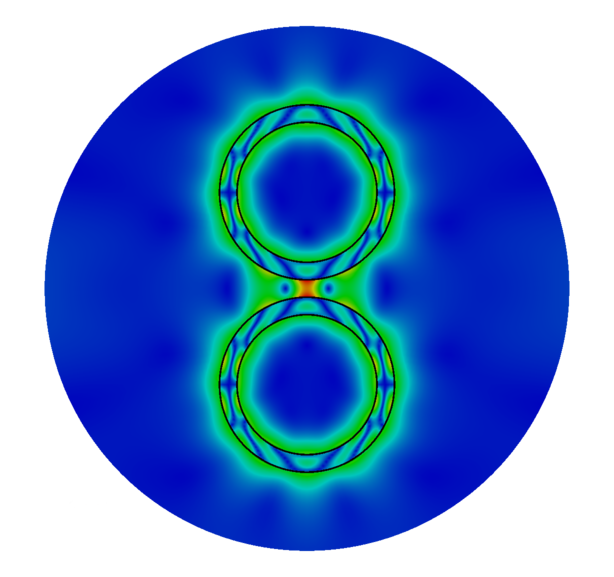} };	
		\node[scale=0.8] at ( 18.40-1.5, 0.0+1.3 ) {$\om_{12}$};
		\draw( 18.40, 0.0) node { \includegraphics[scale=0.1]{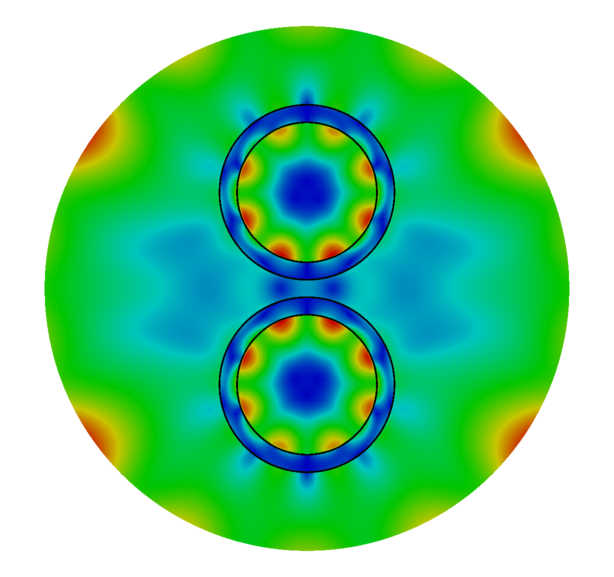} };
		\node[scale=0.8] at ( 18.40-1.5,-2.5+1.3 ) {$\om_{16}$};
		\draw( 18.40,-2.5) node { \includegraphics[scale=0.1]{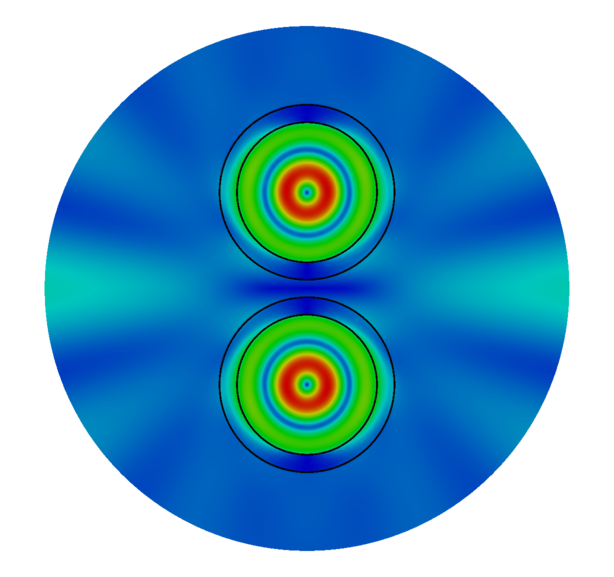} };
		\node[scale=0.8] at ( 18.40-1.5,-5.0+1.3 ) {$\om_{20}$};
		\draw( 18.40,-5.0) node { \includegraphics[scale=0.1]{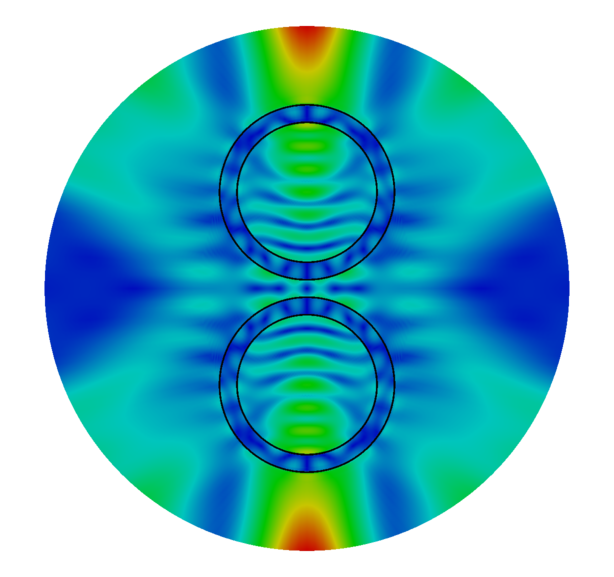} };
		
	\end{tikzpicture}
	\vspace*{-4mm}
	\caption{\emph{Spectral window for the \emph{coated dimer} problem \ref{sec:Cdimer} with TE polarization. FEM+NLEIGS eigenvalues $\om_\fem$ are shown with dots, poles $z$ of $\eps(\om)$ with squares and its zeros with diamonds. The plasmonic branch points are marked with $\times$ and $+$. Right) In colors we plot $\|E_j\|$ from the resonant mode corresponding to $\om_j$ listed in table \ref{tab:CD_reference}. } }
	\label{fig:eigs_CD_TE}
\end{figure}


\section{Numerical experiments and results}\label{sec:results}

In this section we perform numerical computations to test the reliability and performance of the proposed solution strategy. Particularly, we present a comparison of classical finite element error convergence against convergence of the a-priori strategies presented in Section \ref{sec:strategies}. 
For the comparison, we define the \emph{gain} 
as the percentage of reduction in degrees of freedom compared to using classical FE refinement strategies at a fixed relative error.
From a conforming coarse triangulation $\mc T(\Om)$ with no ghost nodes, the classical $h$ refinement strategy consists in keeping $p$ fixed, and performing consecutive refinements by splitting each element in $2^d$ new elements. The classical $p$ refinement strategy consists in keeping the number of elements constant and increasing $p$ uniformly in each cell. 

\subsection{Results for non-dispersive problems} 

The studies are performed on the problems described in Section \ref{sec:benchmarks}, where expressions for the reference solutions are given for most problems. Furthermore, all given study cases have piecewise analytic coefficients and the domains have no corners. Hence, the expected optimal asymptotic error estimates are:
\begin{equation}
	\begin{array}{ll}
		\|u-u_h\|_l \!\!\!&\leq C(\omega)\,h^{p-l+2}\|u\|_l,\,\,\text{for }h\hbox{-FE} \text{ with } l=0,1,\\[2mm]
		\|u-u_h\|_l \!\!\!&\leq C(\omega)\,e^{-\alpha_l N^{1/d}}\|u\|_l,\,\,\text{for }p\hbox{-FE} \text{ with } l=0,1,
	\end{array}
	\label{eq:optimal_estimates}
\end{equation}
where we denote by $\|\cdot\|_l$ the standard $H^l(\Om)$ norm. For the convergence studies, we use the relationship $N\leq ch^{-d},\,c>0$ valid for \emph{shape regular} meshes \cite[Sec. 4.3]{Schwab1998}.  

In the following sections we discuss the results of the convergence study.

\subsubsection{Results for 1D problems} \label{sec:res1d}

We start by describing general observations resulting from computations on the 1D problems described in Section \ref{sec:models_1D}, where the discretization is based on the DtN formulation described in Sec. \ref{sec:resonances1D}. 
First, we gather results for the single \emph{slab} problem \ref{sec:slab1d} in Figure \ref{fig:convergence_1D}. 
In the upper strip we present classical $h$-FE relative errors corresponding to TM eigenvalues and eigenfunctions measured in $L_2$, and $H^1$ norms. Plots corresponding to $n_1=2,5,10$ are given for an eigenvalue close to the shift $\mu=10$. The results indicate that the optimal convergence rates \eqref{eq:optimal_estimates} are reached, and the same was observed for the TE polarization. The following horizontal strips depict $H^1$ errors for classical FE and for strategies \ref{sec:h_str} and \ref{sec:p_str}, from where it is observed that both proposed strategies 
effectively reduce the convergence's pre-asymptotic phase compared with classical FE refinements. 
However, the asymptotic rate of convergence remains naturally unaltered. Additionally, optimal convergence rates \eqref{eq:optimal_estimates} are reached for both TM and TE polarizations, experiencing the same gain independently of polarization. The second and third horizontal strips in Figure \ref{fig:convergence_1D} show a gain from $h$-FE of $25\%,\,44\%$, and $47\%$, corresponding to $n_1=2,\,5$, and $10$ for a relative error around $10^{-3}$.
Similarly, the fourth and fifth horizontal strips in Figure \ref{fig:convergence_1D} show a gain from $p$-FE of $16\%,\,21\%$, and $35\%$, corresponding to $n_1=2,\,5$, and $10$ for a relative error around $10^{-6}$.
Additional numerical computations confirm that similar gains were observed for other shifts. Additionally, we observe from the plots that the gain by using the strategy  \ref{sec:p_str} increases for higher accuracies. This is expected as the proposed strategies are designed from Theorem \ref{thm:low_frequency}, and Theorem \ref{thm:high_frequency}.
The results confirm that the a-priori strategies \ref{sec:h_str} and \ref{sec:p_str}, achieve convergence with a shorter pre-asymptotic phase than classical FE-methods, 
with a gain proportional to the refractive index contrast. 
The results from computations on problem \ref{sec:multislab1d} are gathered in Figure \ref{fig:convergence_multi1D}, where it becomes evident that
the use of the proposed strategies also work well for problems with piecewise constant coefficients. Computations feature a gain of $37\%$ for a relative error of order $10^{-3}$ in $h$-FE, whereas $36\%$ gain for a relative error of order $10^{-6}$ in $p$-FE.

\subsubsection{Results for the \emph{single disk} problem}\label{sec:res_SD} 

In this section the a-priori strategies \ref{sec:h_str}, and \ref{sec:p_str} are tested for configurations in 2D.
In particular, the results for the problem \ref{sec:SD}, with $n_1=5$, are gathered in Figure \ref{fig:convergence_n5_SD}. In this problem, eigenpairs are numbered using the angular integer $m$ as suggested by \eqref{eq:exact}. 
We compute pairs for TM with $\mu=10.2-0.04i$, and for TE with $\mu=9.85-0.04i$. 
The results for both
polarizations are very similar to those discussed in Section \ref{sec:res1d} for 1D.
Particularly, the eigenvalue error for this problem converges following the optimal rates \eqref{eq:optimal_estimates}, and both a-priori strategies achieve convergence with a shorter pre asymptotic phase compared to classical FE-methods.
Particularly we achieve a gain of $36\%$ for the $h$-strategy with relative error of order $10^{-6}$, 
whereas up to $17\%$ in the $p$-strategy with relative error of order $10^{-8}$. 
Furthermore, we see that for the chosen $\mu$, eigenfunctions with different $m$ exhibit the same gain.

\begin{remark}\label{rem_wgm}
For large angular values $m$, eigenfunctions in this problem are expected to exhibit localized oscillations around the boundary of the dielectric disk (juncture with air) that extend to air. These are known as whispering-Gallery-modes (WGM) \cite{johnson93}. It is observed that the strategies \ref{sec:h_str}, and \ref{sec:p_str} underestimate the FE requirements for correct approximation of these modes, as we refine cells according to bulk estimators/goals, and contributions from edges are not considered. However, since we know in advance where to perform mesh refinements it is straightforward to setup a-priori strategies for accurate computation of these modes. From now on, we exclude these type of modes from our discussions.
\end{remark}


\subsection{Results for dispersive problems}

In the remainder of the section we gather results from problems described in sections \ref{sec:SDC}, and \ref{sec:Cdimer}, which feature dispersive material properties. The positive results from last sections indicate that the a-priori strategies \ref{sec:h_str}, and \ref{sec:p_str} applied to non-dispersive problems perform best when there is a high contrast in the refractive index. Similarly, we expect to obtain greater gains when $|n(\om)|$ is large.
We start by testing the reliability and performance of the NEP solution strategy described in \ref{sec:NLEIGS}. 
Particularly, we check that the strategy can be used to obtain good approximations to the exact resonances even close to the poles and zeros of $\eps_{metal}(\om)$ given in \eqref{eq:drude_lorentz}.
Finally, we consider the error convergence for the problem presented in \ref{sec:Cdimer}, 
which is computationally more demanding.
\subsubsection{Results for the \emph{single coated disk} problem}\label{sec:res_CSD}

In order to test the reliability of the proposed NEP solver strategy, we use the Benchmark presented in 
Section \ref{sec:SDC}. From \eqref{eq:eig_matrix}, we compute approximations to the resonances given 
by \eqref{eq:reson_CSD}. In Figure \ref{fig:eigs_SCD}, we present the result after taking multiple shifts inside a relatively large spectral window, from where we observe an excellent agreement between approximations and exact resonances. We conclude that the proposed a-priori strategies together with SLEPc's implementation of NLEIGS result in excellent approximations of the exact pairs even close to the poles and zeros of model \eqref{eq:drude_lorentz}.
 
Moreover, computations corresponding to TE polarization feature a sequence of resonances accumulating around the so-called plasmonic branch points of the model, which are the values of $\om$ such that $\eps_{metal}=-1,\,\eps_{metal}=-2$. For reference, we mark them with $\times$, and $+$ respectively. 
From Figure \ref{fig:eigs_SCD} and TE polarization (right), we observe that the approximation $\om^\fem=0.6288-0.6288i$ converged to an eigenvalue of the modified PML problem, which differs considerably from the exact value $\om=0.5569-0.6457i$. The reason is that $\om^\fem$ is close to the critical line of the PML \cite{araujo+engstrom+2017}. 

The performance of the solver is evaluated in Fig.~\ref{fig:performance}, where
the plots illustrate the strong scaling of the parallel code, that is, how the execution time varies for increasing number of processes with a fixed problem size. 
Since the problem size is constant, for large number of processes the performance degrades, because the amount of work assigned to each process is too small.
We can see that the run time for 128 processes grows with respect to 64 processes; if the test problems were bigger this performance degradation would occur later for larger number of processes.
Still, we cannot expect to scale to many more processes since the solver employs a direct linear solver (MUMPS in our case) for one step of the algorithm, which has limited scalability. 
The figure also shows that the total execution time in the case of higher polynomial degree (right plots) is significantly smaller than for the higher refinement level (left plots). This is due to a much smaller problem size, see Table \ref{tab:matsize}, even though the generated matrices are much less sparse. A shorter time and a higher percentage of nonzero elements also implies a worse scalability, as it can also be seen in the right plots.
\begin{figure}[t]
\centering
\begin{tikzpicture}[scale=0.75]
  \begin{loglogaxis}[
    title={\textsf{TM, $r=8$, $p_0=2$}},
    ylabel={Time [s]},
    grid=major,
    log basis x=2,
    xtick={1,2,4,8,16,32,64,128},
    xticklabels={1,2,4,8,16,32,64,128},
    ticklabel style={font=\small},
    legend style={at={(.98,.98)},anchor=north east,cells={anchor=west},font=\small}
    ]
    \addplot coordinates { 
          (  1,  3.9590e+03)
          (  2,  9.8398e+02)
          (  4,  6.5973e+02)
          (  8,  3.7497e+02)
          ( 16,  2.3655e+02)
          ( 32,  1.7228e+02)
          ( 64,  1.4078e+02)
          (128,  1.4852e+02)
    };
  \end{loglogaxis}
\end{tikzpicture}
\hspace{3mm}
\begin{tikzpicture}[scale=0.75]
  \begin{loglogaxis}[
    title={\textsf{TM, $r=3$, $p_0=10$}},
    ylabel={Time [s]},
    grid=major,
    log basis x=2,
    xtick={1,2,4,8,16,32,64,128},
    xticklabels={1,2,4,8,16,32,64,128},
    ticklabel style={font=\small},
    legend style={at={(.98,.98)},anchor=north east,cells={anchor=west},font=\small}
    ]
    \addplot coordinates { 
          (  1, 7.2342e+01 )
          (  2, 4.3065e+01 )
          (  4, 2.6399e+01 )
          (  8, 1.4087e+01 )
          ( 16, 8.7583e+00 )
          ( 32, 5.7665e+00 )
          ( 64, 5.2629e+00 )
          (128, 8.1390e+00 )
    };
  \end{loglogaxis}
\end{tikzpicture}
\\[2mm]
\hspace{0.3mm}
\begin{tikzpicture}[scale=0.75]
  \begin{loglogaxis}[
    title={\textsf{TE, $r=8$, $p_0=2$}},
    ylabel={Time [s]},
    grid=major,
    log basis x=2,
    xtick={1,2,4,8,16,32,64,128},
    xticklabels={1,2,4,8,16,32,64,128},
    ticklabel style={font=\small},
    legend style={at={(.98,.98)},anchor=north east,cells={anchor=west},font=\small}
    ]
    \addplot coordinates { 
          (  2, 1.9624e+03 )
          (  4, 1.3223e+03 )
          (  8, 6.4399e+02 )
          ( 16, 4.0866e+02 )
          ( 32, 2.5364e+02 )
          ( 64, 1.9569e+02 )
          (128, 2.0276e+02 )
    };
  \end{loglogaxis}
\end{tikzpicture}
\hspace{4.5mm}
\begin{tikzpicture}[scale=0.75]
  \begin{loglogaxis}[
    title={\textsf{TE, $r=3$, $p_0=10$}},
    ylabel={Time [s]},
    grid=major,
    log basis x=2,
    xtick={1,2,4,8,16,32,64,128},
    xticklabels={1,2,4,8,16,32,64,128},
    ticklabel style={font=\small},
    legend style={at={(.98,.98)},anchor=north east,cells={anchor=west},font=\small}
    ]
    \addplot coordinates { 
          (  1, 1.1644e+02 )
          (  2, 7.6413e+01 )
          (  4, 4.8621e+01 )
          (  8, 2.2261e+01 )
          ( 16, 1.4908e+01 )
          ( 32, 9.6517e+00 )
          ( 64, 9.2191e+00 )
          (128, 1.1252e+01 )
    };
  \end{loglogaxis}
\end{tikzpicture}
\caption{\emph{Parallel execution time (in seconds) of the solver for varying number of MPI processes (up to 128), for problem is \ref{sec:res_CSD}. The target used is $\mu=5.3-0.25i$, where top and bottom plots correspond to TM and TE polariazations respectively. Left plots correspond to a discretization dominant in the $h$-strategy ($r=8$ levels of refinement, polynomial order $p_0=2$), while right plots are for a discretization dominant in the $p$-strategy ($r=3$ levels of refinement, polynomial order $p_0=10$). }}
\label{fig:performance}
\end{figure}
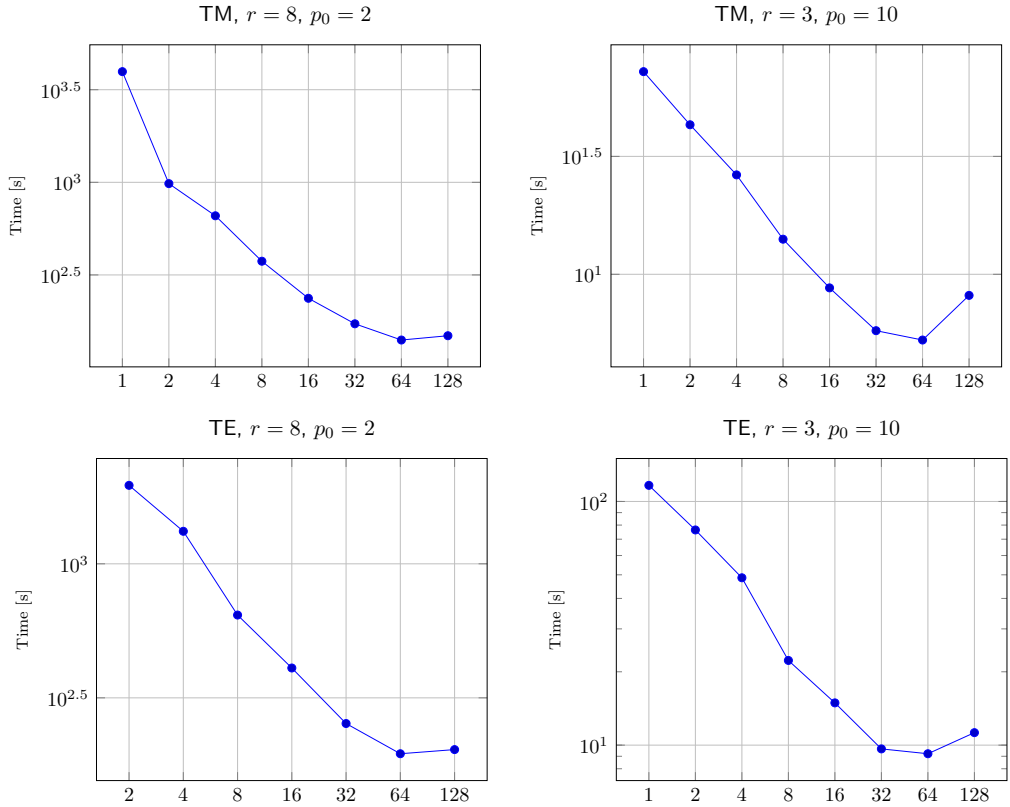

\begin{table}[]
\centering
\begin{tabular}{lll}
Problem             & Matrix size & \% of nonzeros \\ \hline
\textsf{Single coated disk, $r=8$, $p_0=2$}   &   4,288,929 & 0.00037 \\
\textsf{Single coated disk, $r=3$, $p_0=10$} &     173,725 & 0.062 \\ \hline
\end{tabular}
\caption{Dimensions and sparsity of the matrix that is factorized during the execution of the NLEIGS solver, for the two test cases of Fig.~\ref{fig:performance}. The number of uniform $h$ refinements is denoted with $r$. }
\label{tab:matsize}
\end{table}

\subsubsection{Results for the \emph{coated dimer} problem}\label{sec:res_CD}
Finally, we present results for the coated dimer problem described in Section \ref{sec:Cdimer}, from where the reference values $\om_j$ listed in Table \ref{tab:CD_reference} were computed from \eqref{eq:eig_matrix} with a very fine discretization.
In Figure \ref{fig:grid_cdimer}, we show part of the mesh utilized for this problem, and in colors we give the polynomial distribution $p_j$ per cell resulting by using the $p$-strategy \ref{sec:p_str}. The distributions shown correspond to $\mu=4.162-0.2648i$ with $p_0=7$ (left), and $\mu=2.9-0.422i$ with $p_0=10$ (right). We observe that the resulting a-priori strategy assigns lower polynomial degrees to cells with small diameters.
As seen from Figure \ref{fig:grid_cdimer}, the initial mesh contains a wide range of cell diameters. 
This property is exploited by the $p$-strategy \ref{sec:p_str}, because both $h$ and $p$ play a role when satisfying Goal \ref{def:goal}. 
The resulting a-priori refinement strategy features remarkable gains ranging from $35\%$ to $48\%$
compared to the classical $p$-FE. These gains depend on the selected $\mu$ and on the specific shape of the corresponding eigenfunctions. The error convergence for some of the computed eigenvalues is gathered in Figure \ref{fig:convergence_cdimer},
where we show convergence for both polarizations and different $\mu$ values. Similarly to the non-dispersive case, the application of the $p$-strategy \ref{sec:p_str} to this problem results in shorter pre-asymptotic phase of the error for the computed eigenpairs in both polarizations.
Finally, in the left panels of Figures \ref{fig:eigs_CD_TM}, and \ref{fig:eigs_CD_TE} we present computed eigenvalues from \eqref{eq:eig_matrix} by performing multiple shifts inside a relatively large spectral window. 
As expected, the location of the resulting eigenvalues resemble those from the \emph{single coated disk} in 
Figure \ref{fig:eigs_SCD}. Although being more densely populated, the spectral windows exhibit similar features like accumulations to poles, branch points, and similar location of resonances.
The Figures \ref{fig:eigs_CD_TM}, and \ref{fig:eigs_CD_TE} also include color plots for $\|E_j(x)\|$ corresponding to the $\om_j$ listed in Table \ref{tab:CD_reference}, where we have excluded the PML layer. These plots reveal the rich electromagnetic phenomena described by resonances and resonant modes.

\begin{figure}
	\begin{tikzpicture}[thick,scale=1.0, every node/.style={scale=0.9}]
		
		\draw(  0.00, 10.5) node {\includegraphics[scale=0.53]{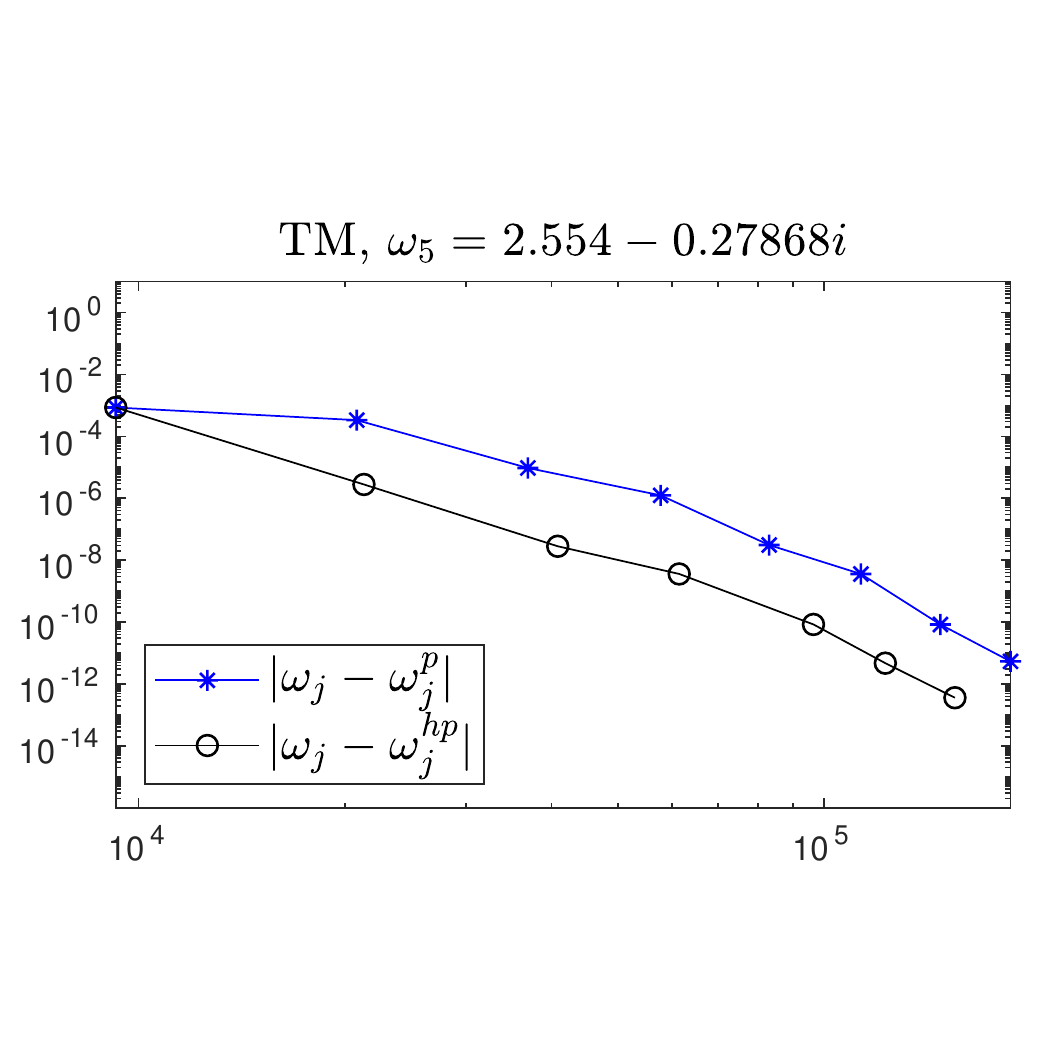}  };
		\draw(  5.50, 10.5) node {\includegraphics[scale=0.53]{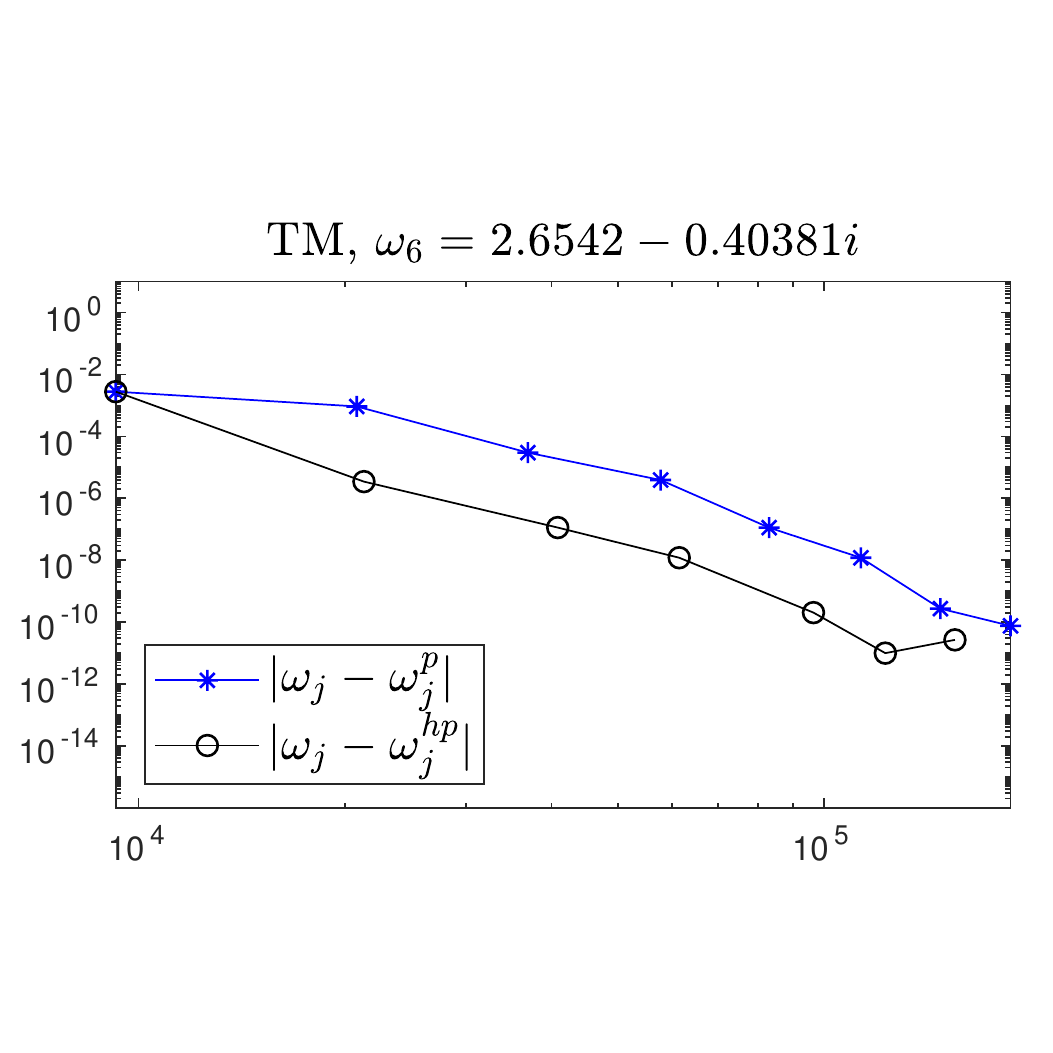}  };
		\draw( 11.00, 10.5) node {\includegraphics[scale=0.53]{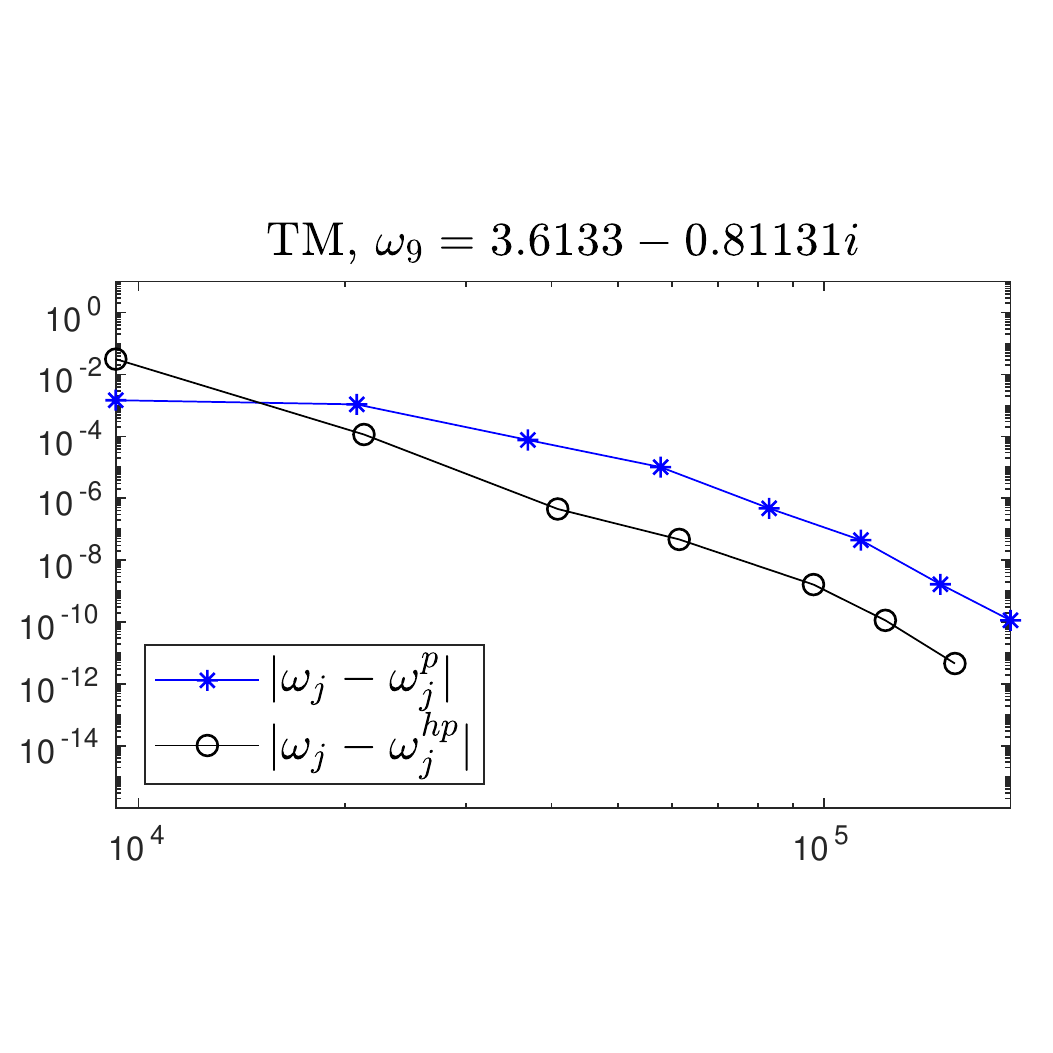}  };
		
		\draw(  0.00, 7.0) node {\includegraphics[scale=0.53]{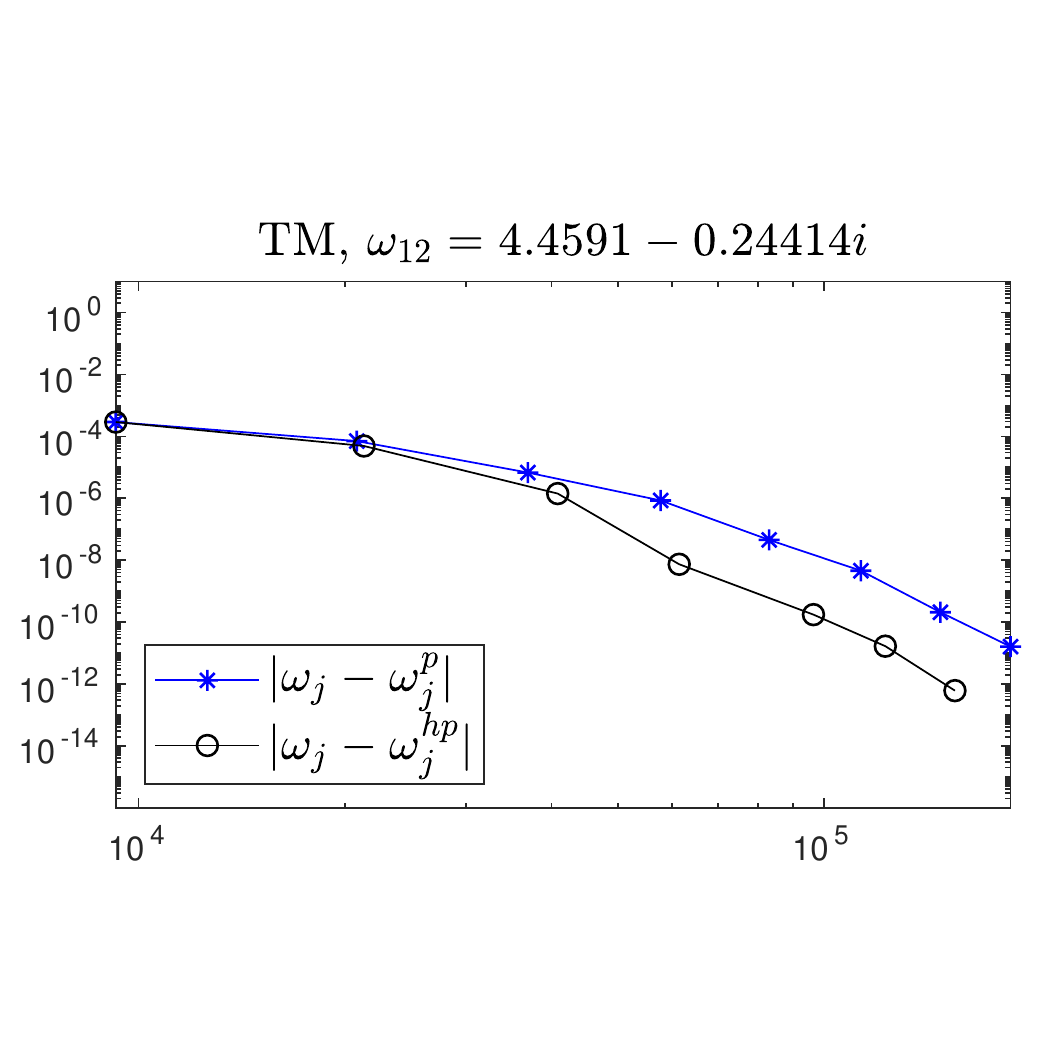}  };
		\draw(  5.50, 7.0) node {\includegraphics[scale=0.53]{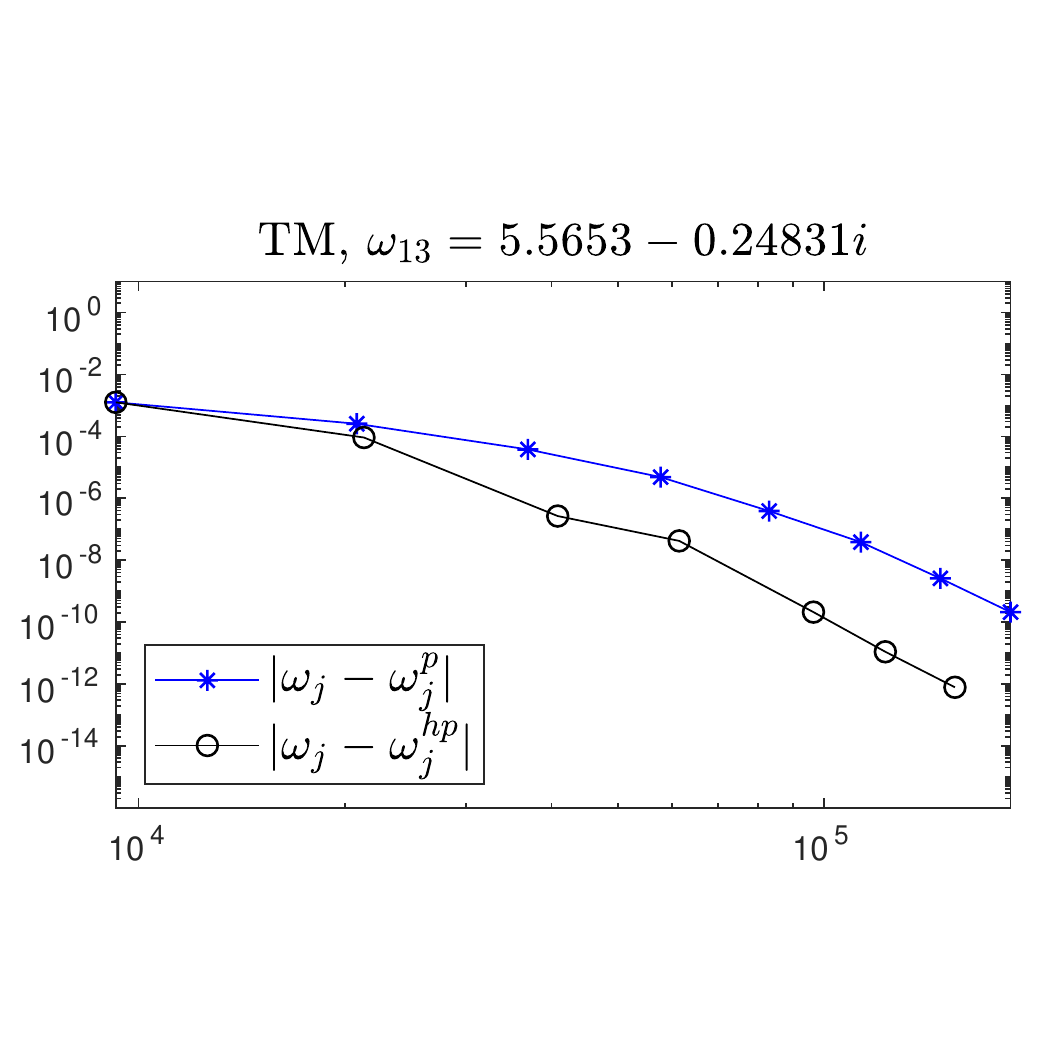}  };
		\draw( 11.00, 7.0) node {\includegraphics[scale=0.53]{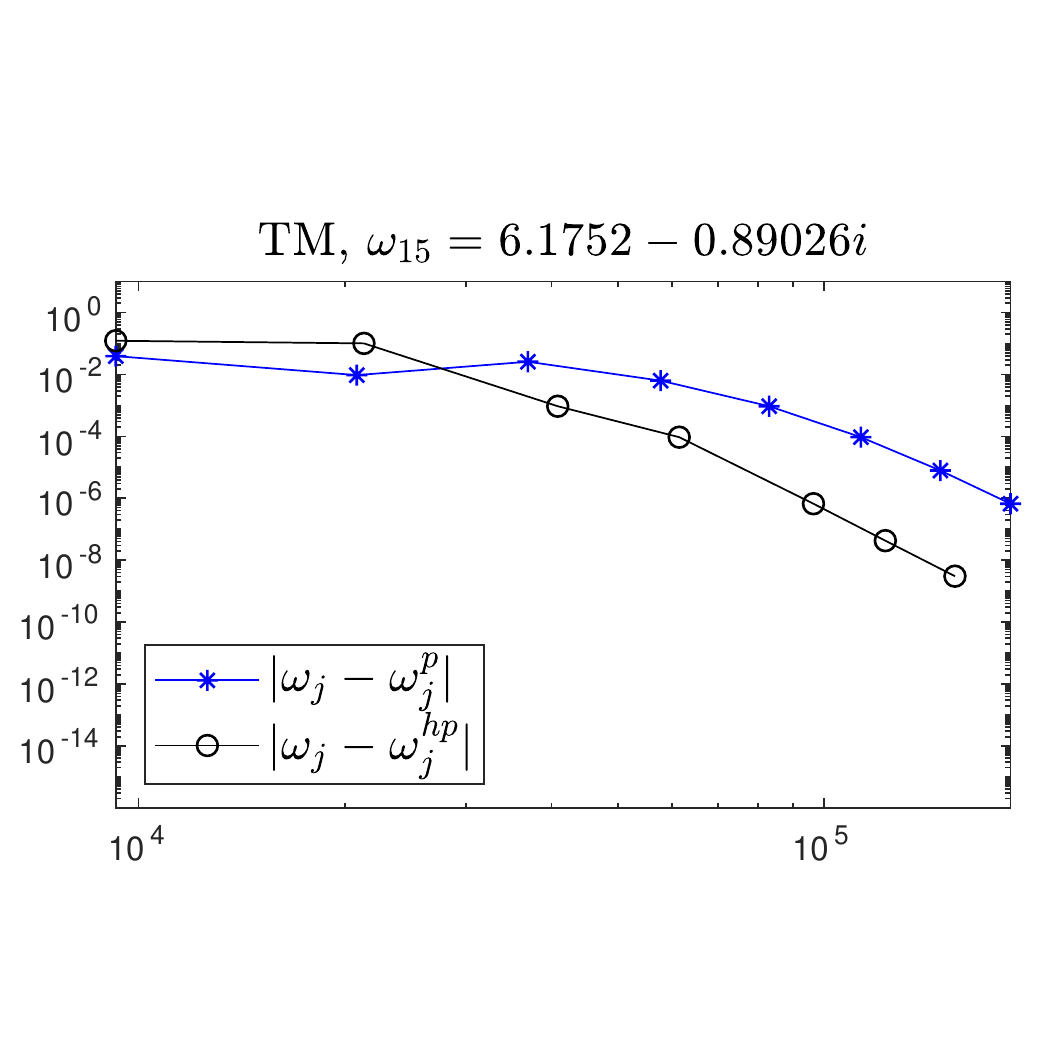}  };
		
		\draw(  0.00, 3.5) node {\includegraphics[scale=0.53]{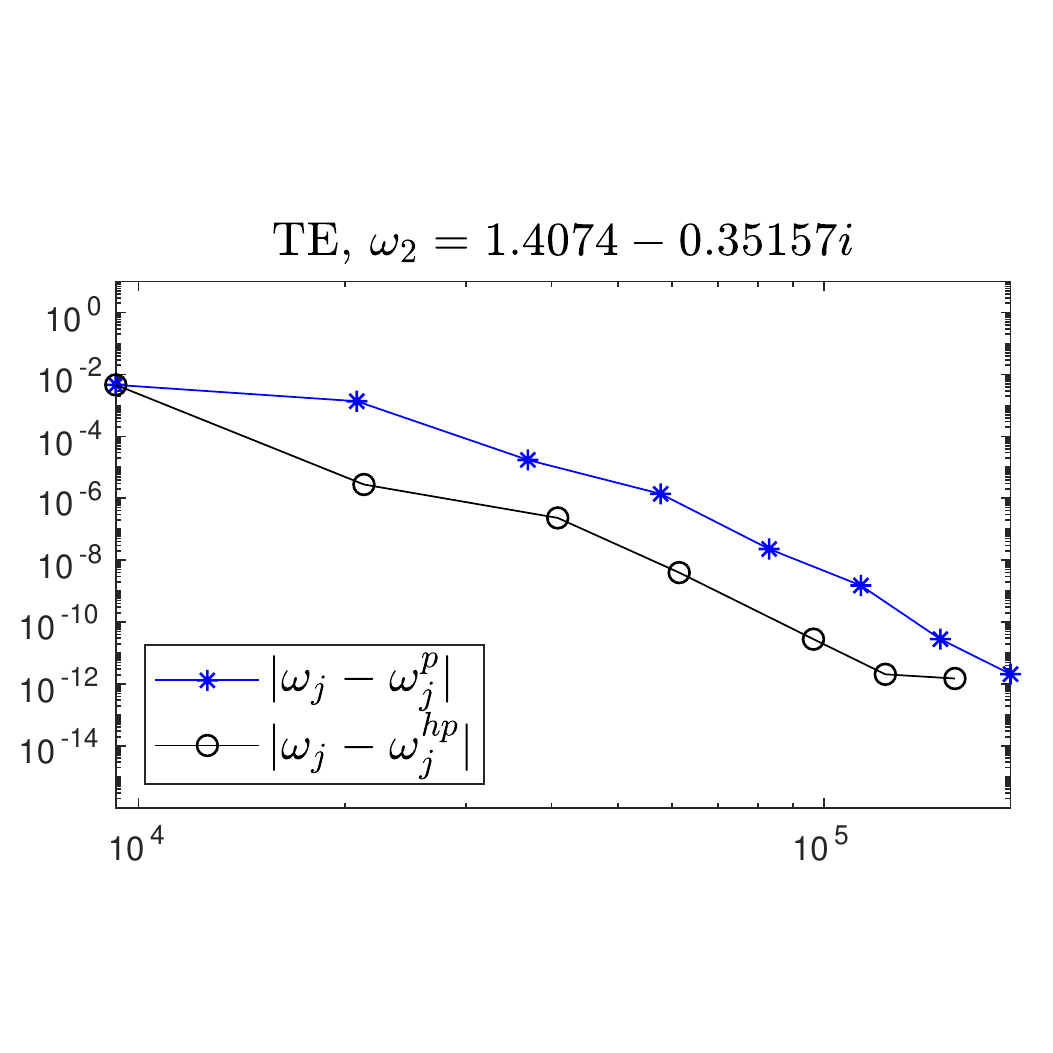}  };
		\draw(  5.50, 3.5) node {\includegraphics[scale=0.53]{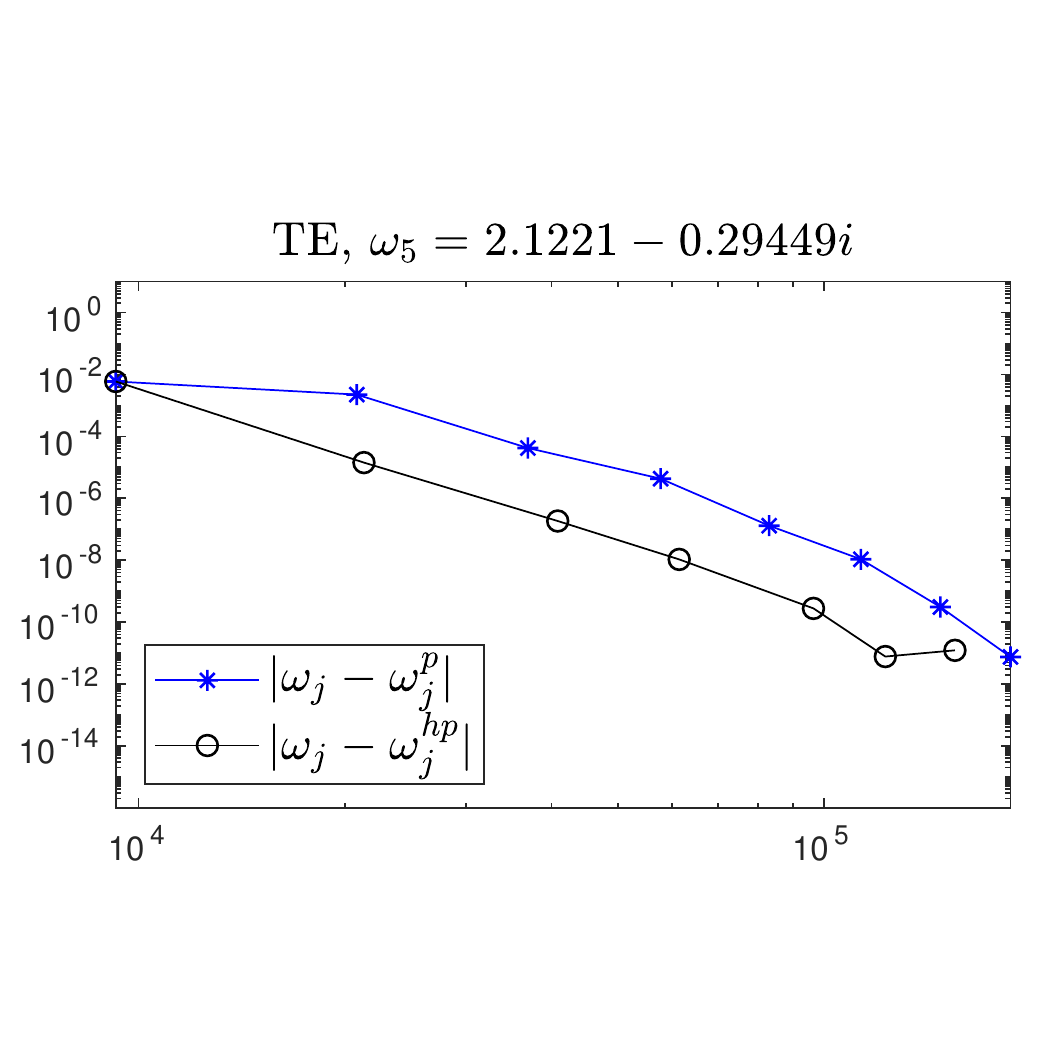}  };
		\draw( 11.00, 3.5) node {\includegraphics[scale=0.53]{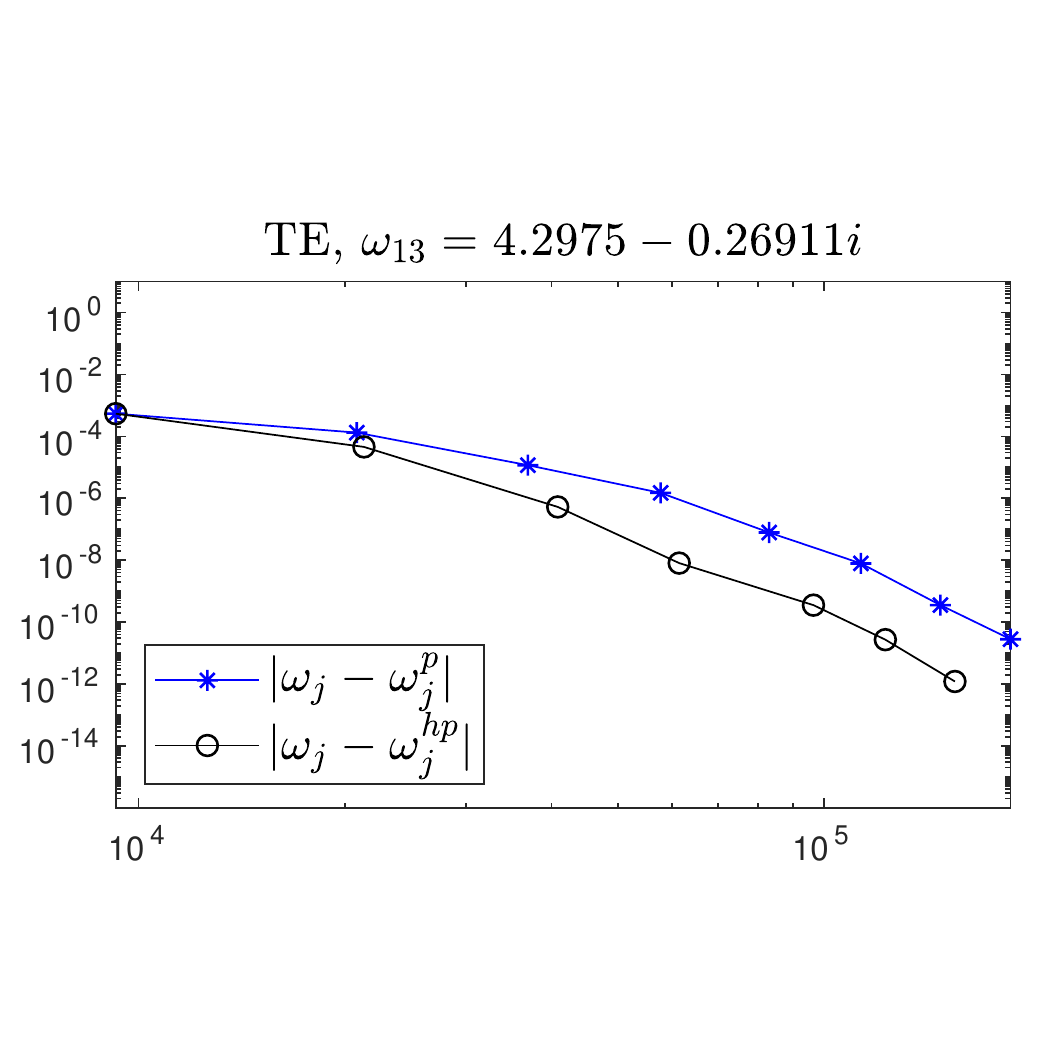}  };
		
		\draw(  0.00, 0.0) node {\includegraphics[scale=0.53]{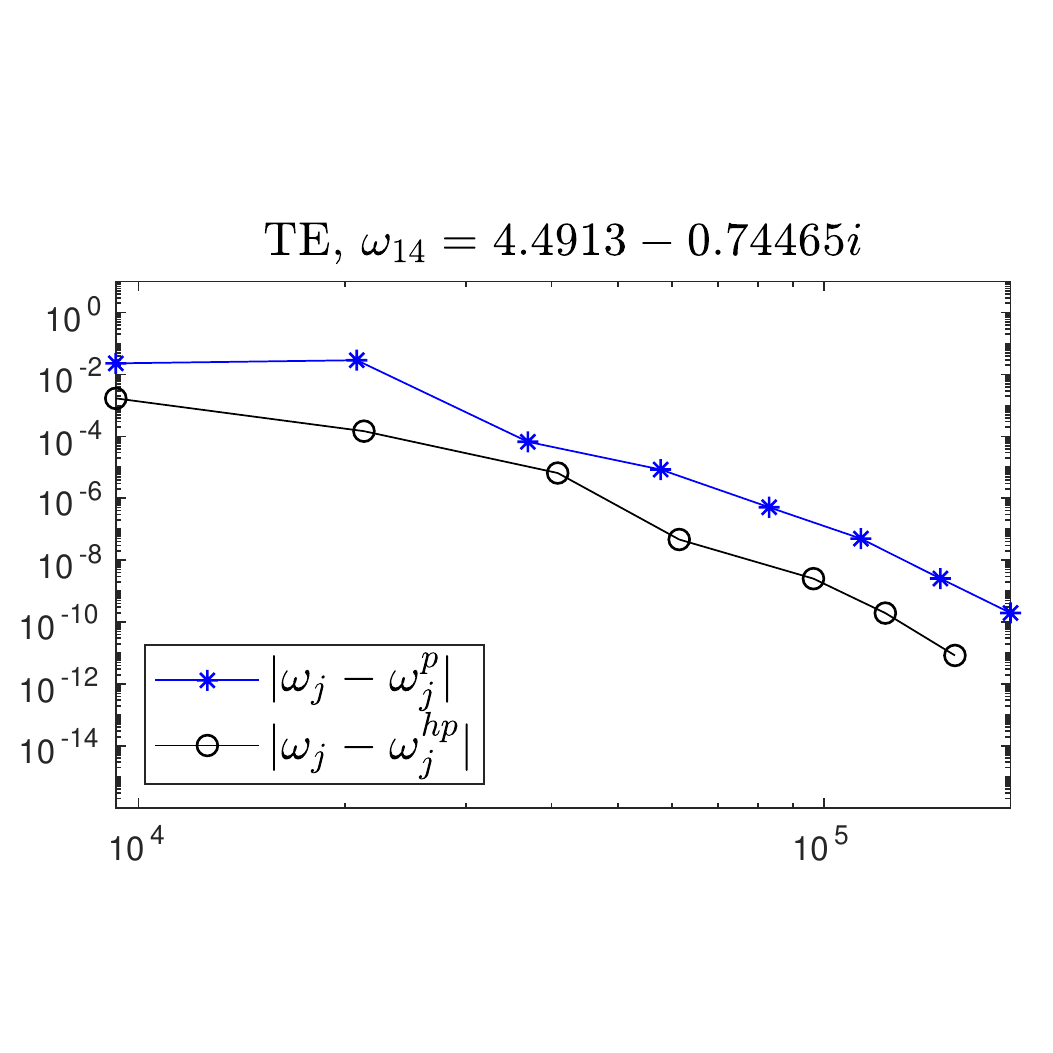}  };
		\draw(  5.50, 0.0) node {\includegraphics[scale=0.53]{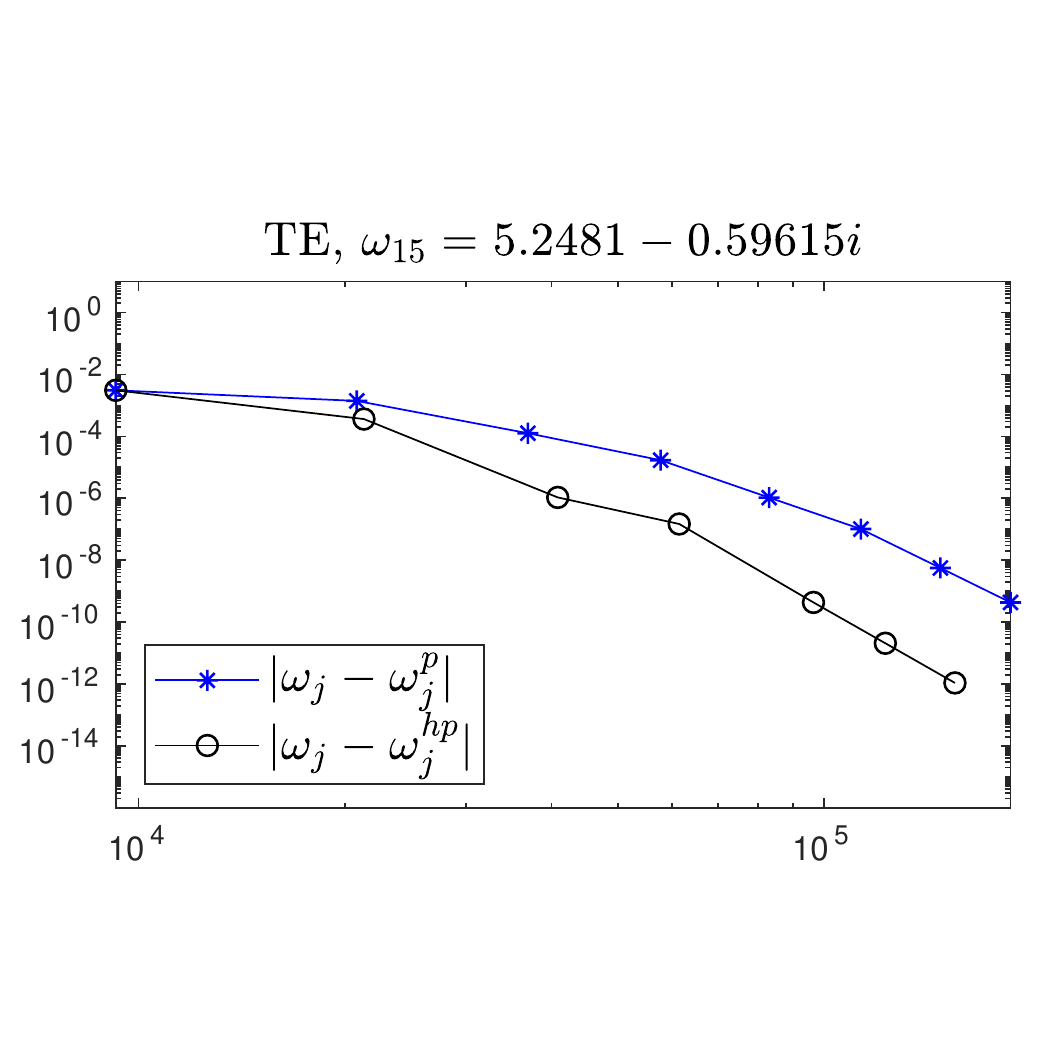}  };
		\draw( 11.00, 0.0) node {\includegraphics[scale=0.53]{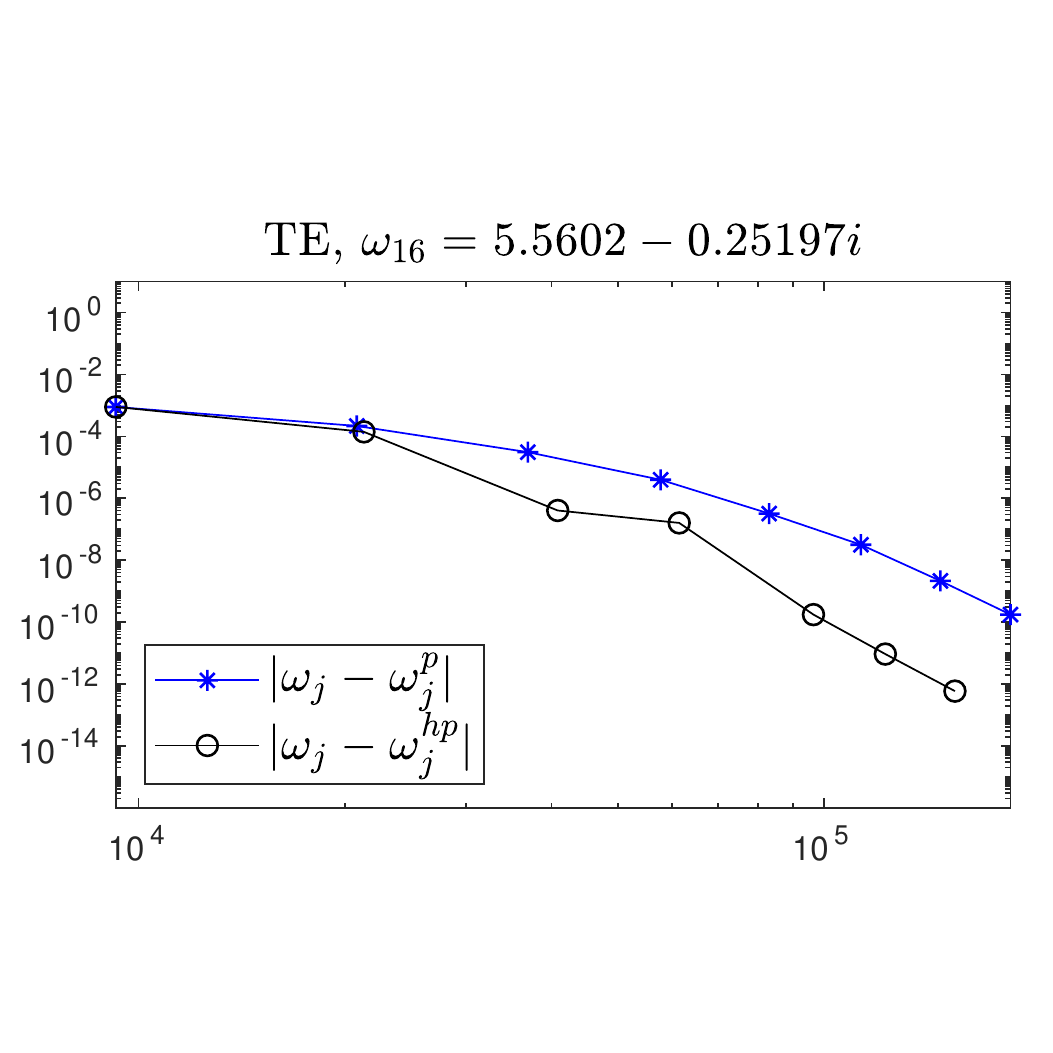}  };
		
	\end{tikzpicture}
	\vspace*{-3mm}
	\caption{\emph{Convergence $p$-FE plots (Relative errors vs. N) for TM and TE polarizations of the model problem presented in Sec. \ref{sec:Cdimer}. We mark with circles the a-priori $p$-strategy, and with stars classical $p$-FE refinements.}}
	\label{fig:convergence_cdimer}
\end{figure}

\section{Conclusions}
We have proposed an $hp$-refinement strategy for approximation of complex scattering resonances in optics. Numerical computations in demanding 1D and 2D cases indicate that the a-priori $hp$-FEM strategy results in a significant reduction of the pre-asymptotic phase in both $h$-FE and $p$-FE. The resulting non-linear matrix eigenvalue problem is solved by SLEPc’s state-of-the-art implementation of the nonlinear eigenvalue solver NLEIGS. This results in fast and highly accurate computations of resonances for metal-dielectric resonators.

\section*{Acknowledgments}
Juan C. Ara\'ujo and Christian Engström gratefully acknowledge the support of the Swedish Research Council under Grant No. 621-2012-3863. Carmen Campos and Jose E.~Roman were supported by the Spanish Agencia Estatal de Investigaci{\'o}n (AEI) under project SLEPc-HS (TIN2016-75985-P), which includes European Commission ERDF funds. The supercomputer Tirant 3 used in some of the computational experiments belongs to Universitat de Val\`encia. Juan C. Ara\'ujo acknowledges Andr\'ee Falgin Hultgren, for his contributions on the meshing routine.

\bibliographystyle{unsrt} 

\end{document}